\documentclass[12pt]{amsart}
\usepackage{txfonts}
\usepackage{subfig}
\usepackage{amscd,amssymb,mathabx}
\usepackage[arrow,matrix,graph,frame,poly,arc,tips]{xy}
\usepackage{graphicx,calrsfs}
\usepackage{mathtools}
\usepackage[dvipsnames]{xcolor}
\usepackage{tikz}
\usetikzlibrary{decorations.markings}
\usetikzlibrary{shapes}
\usetikzlibrary{arrows}
\usetikzlibrary{backgrounds}
\usetikzlibrary{calc}
\usepackage{mdwlist}
\usepackage{multirow}
\usepackage{enumerate}
\usepackage{verbatim}
\usepackage{etoolbox}
\AtEndEnvironment{proof}{\setcounter{claim}{0}}

\newcommand\BS{\operatorname{BS}}

\DeclarePairedDelimiter{\form}{\langle}{\rangle}

\newcommand\ba{\begin{align*}}
\newcommand\ea{\end{align*}}
\newcommand\be{\begin{enumerate}}
\newcommand\ee{\end{enumerate}}
\newcommand\bp{\begin{proof}}
\newcommand\ep{\end{proof}}
\newcommand\bpp{\begin{prop}}
\newcommand\epp{\end{prop}}
\newcommand\bpb{\begin{prob}}
\newcommand\epb{\end{prob}}
\newcommand\bd{\begin{defn}}
\newcommand\ed{\end{defn}}
\newcommand\bh{\begin{hint}}
\newcommand\eh{\end{hint}}

\newcommand\fform[1]{\langle\!\langle #1\rangle\!\rangle}

\newcommand\bC{\mathbb{C}}

\newcommand\bN{\mathbb{N}}

\newcommand\bR{\mathbb{R}}

\newcommand\bQ{\mathbb{Q}}

\newcommand\bZ{\mathbb{Z}}

\newcommand\BB{\mathcal{B}}

\newcommand\FF{\mathcal{F}}
\newcommand\GG{\mathcal{G}}

\newcommand\UU{\mathcal{U}}
\newcommand\VV{\mathcal{V}}

\newcommand\gap{\operatorname{gap}}
\newcommand\bag{\operatorname{bag}}
\renewcommand\Re{\operatorname{Re}}

\newcommand\Bag{\operatorname{Bag}}
\newcommand\Ball{\operatorname{Ball}}

\newcommand\cl{\operatorname{CovLen}}
\newcommand\cd{\operatorname{CovDist}}

\newcommand\Hom{\operatorname{Hom}}

\newcommand\supp{\operatorname{supp}}
\newcommand\Id{\operatorname{Id}}

\newcommand\Diffb{\operatorname{Diff}_+^{1+\mathrm{bv}}}

\newcommand\diam{\operatorname{diam}}

\newcommand\PSL{\operatorname{PSL}}

\DeclareMathOperator\Homeo{Homeo}

\newcommand\yt{\widetilde}
\newcommand\sse{\subseteq}
\newcommand\co{\colon}

\DeclareMathOperator\Fix{Fix}

\DeclareMathOperator\Diff{Diff}

\newcommand\rot{\operatorname{rot}}

\newcommand\bba{\mathrm{a}}
\newcommand\bbb{\mathrm{b}}
\newcommand\bbc{\mathrm{c}}
\newcommand\bbd{\mathrm{d}}
\newcommand\bbe{\mathrm{e}}

\renewcommand{\MR}[1]
{\href{http://www.ams.org/mathscinet-getitem?mr=#1}{MR#1}}

\def\thetitle{{Diffeomorphism groups of critical regularity}}
\def\theauthors{{Sang-hyun Kim and Thomas Koberda}}
\usepackage{hyperref}
\hypersetup{
   colorlinks=false,
   plainpages,
   urlcolor=black,
   linkcolor=black
   pdftitle=   \thetitle,
   pdfauthor=  {\theauthors}
}

\theoremstyle{theorem}
\newtheorem{thm}{Theorem}[section]
\newtheorem{lem}[thm]{Lemma}
\newtheorem{cor}[thm]{Corollary}
\newtheorem{prop}[thm]{Proposition}

\newtheorem{que}[thm]{Question}
\newtheorem*{claim*}{Claim}

\newtheorem{claim}{Claim}

\newtheorem*{mainthm}{\bf Main Theorem}

\theoremstyle{remark}
\newtheorem{exmp}[thm]{Example}
\newtheorem{rem}[thm]{Remark}

\theoremstyle{definition}
\newtheorem{defn}[thm]{Definition}
\newtheorem{prob}{Problem}[section]
\newtheorem{setting}[thm]{Setting}
\newtheorem{notation}[thm]{Notation}

\begin{document}
\title\thetitle
\date{\today}
\keywords{right-angled Artin group; algebraic smoothing; smoothability; H\"older continuity; infinite simple group}
\subjclass[2010]{Primary: 57M60; Secondary: 20F36, 37C05, 37C85, 57S05}

\author[S. Kim]{Sang-hyun Kim}
\address{School of Mathematics, Korea Institute for Advanced Study, Seoul, Korea}
\email{skim.math@gmail.com}
\urladdr{http://cayley.kr}

\author[T. Koberda]{Thomas Koberda}
\address{Department of Mathematics, University of Virginia, Charlottesville, VA 22904-4137, USA}
\email{thomas.koberda@gmail.com}
\urladdr{http://faculty.virginia.edu/Koberda}

\begin{abstract}
Let $M$ be a circle or a compact interval, and let $\alpha=k+\tau\ge1$ be a real number such that $k=\lfloor \alpha\rfloor$. 
We write $\Diff_+^{\alpha}(M)$ for the  
group of orientation preserving $C^k$ diffeomorphisms of $M$ whose $k^{th}$ derivatives are H\"older continuous with exponent $\tau$.
We prove that there exists a continuum of isomorphism types of finitely generated subgroups $G\le\Diff_+^\alpha(M)$  with the property that $G$ admits no 
 injective homomorphisms into
$\bigcup_{\beta>\alpha}\Diff_+^\beta(M)$.
We also show the dual result:
 there exists a continuum of isomorphism types of finitely generated subgroups $G$ of $\bigcap_{\beta<\alpha}\Diff_+^\beta(M)$
 with the property that
$G$ admits no injective homomorphisms into
$\Diff_+^\alpha(M)$.
The groups $G$ are constructed so that their commutator groups are simple.
We give some applications to smoothability of codimension one foliations and to homomorphisms between certain continuous groups of diffeomorphisms. For example, we show that if $\alpha\ge1$ is a real number not equal to $2$, then there is no nontrivial homomorphism $\Diff_+^\alpha(S^1)\to\bigcup_{\beta>\alpha}\Diff_+^{\beta}(S^1)$.
Finally, we obtain an independent result that the class of finitely generated subgroups of $\Diff_+^1(M)$ is not closed under taking finite free products.
\end{abstract}

\maketitle
\setcounter{tocdepth}{1}
\tableofcontents

\section{Introduction}\label{sec:intro}
Let $M$ be the circle $S^1=\bR/\bZ$
or a compact interval $I$.
A function $f\colon M\to \bR$ is \emph{H\"older continuous with exponent $\tau$} if there is a constant $C$ such that \[|f(x)-f(y)|\leq C|x-y|^{\tau}\] for all $x,y\in M$. In the case where $M= S^1$, we implicitly define $|x-y|$ to be the usual angular distance between $x$ and $y$. 

For an integer $k\ge1$ and for a smooth manifold $M$, we write $\Diff_+^{k+\tau}(M)$ for the group of orientation preserving $C^k$ diffeomorphisms of $M$ whose $k^{th}$ derivatives are H\"older continuous with exponent $\tau\in[0,1)$. 
For compactness of notation,
we will write $\Diff_+^{\alpha}(M)$ for $\Diff_+^{k+\tau}(M)$, where $k=\lfloor \alpha\rfloor$ and $\tau=\alpha-k$.
By convention, we will write $\Diff_+^0(M)=\Homeo_+(M)$.

The purpose of this paper is to study the algebraic structure of finitely generated groups in $\Diff_+^{\alpha}(M)$, as $\alpha$ varies.
We note that the isomorphism types of finitely generated subgroups in $\Diff_+^\alpha(I)$ coincide with those in $\Diff_c^\alpha(\bR)$, the group of compactly supported $C^\alpha$ diffeomorphisms on $\bR$; see Theorem~\ref{thm:embeddability}.

Let us denote by $\GG^{\alpha}(M)$ the class of countable subgroups of $\Diff_+^\alpha(M)$, considered up to isomorphism. 
It is clear from the definition that if $\alpha\leq\beta$ then $\GG^{\beta}(M)\subseteq\GG^{\alpha}(M)$. In general, it is difficult to determine whether a given element $G\in\GG^{\alpha}(M)$ also belongs to $\GG^{\beta}(M)$. 
A motivating question is the following:
\begin{que}\label{que0} Let $k\ge0$ be an integer. 
\be \item Does $\GG^k(M)\setminus\GG^{k+1}(M)$ contain a finitely generated group?
\item Does $\GG^k(M)\setminus\GG^{k+1}(M)$ contain a countable simple group? \ee\end{que}

The answer to the above question is previously known only for $k\le1$ in part (1), and only for $k=0$ in part (2). A first obstruction for the $C^1$--regularity comes from the Thurston Stability~\cite{Thurston1974Top}, which asserts that every finitely generated subgroup of $\Diff_+^1(I)$ is locally indicable. An affirmative answer to part (1) of Question~\ref{que0} follows for $k=0$ and $M=I$; that is, $\GG^0(I)\setminus\GG^1(I)$ contains a finitely generated group.
Using Thurston Stability, Calegari proved that $\GG^0(S^1)\setminus\GG^1(S^1)$ contains a finitely generated group; see~\cite{CalegariForcing} for the proof and also for a general strategy of ``forcing'' dynamics from group presentations. Navas~\cite{Navas2010} produced an example of a locally indicable group in $\GG^0(M)\setminus\GG^1(M)$; see also \cite{Calegari2008AGT}. 

A different $C^1$--obstruction can be found in the result of
Ghys~\cite{Ghys1999} and of Burger--Monod~\cite{BM1999}.
That is, if $G$ is a lattice in a higher rank simple Lie group then $G\not\in\GG^1(S^1)$. This result was built on work of Witte~\cite{Witte1994}.
More generally, Navas~\cite{Navas2002ASENS} showed that every countably infinite group $G$ with property (T) satisfies $G\not\in\GG^1(I)$ and $G\not\in\GG^{1.5+\epsilon}(S^1)$ for all $\epsilon>0$; it turns out that  $G\not\in\GG^{1.5}(S^1)$ by a result of Bader--Furman--Gelander--Monod~\cite{BFGM2007AM}.
The exact optimal bound for the regularity of property (T) groups is currently unknown.

Plante and Thurston~\cite{PT1976} proved that if $N$ is a nonabelian nilpotent group, then $N\notin\GG^2(M)$.
By Farb--Franks~\cite{FF2003} and Jorquera~\cite{Jorquera}, every finitely generated residually torsion-free nilpotent group belongs to $\GG^1(M)$.
For instance, the integral Heisenberg group belongs to $\GG^1(M)\setminus\GG^2(M)$. So, part (1) of Question~\ref{que0} also has an affirmative answer for the case $k=1$.

Another $C^2$--obstruction comes from the classification of right-angled Artin groups in $\GG^2(M)$~\cite{BKK2016,KKFreeProd2017}.
In particular, Baik and the authors proved that except for finitely many sporadic surfaces, no finite index subgroups of mapping class groups of surfaces belong to $\GG^2(M)$ for all compact one--manifolds $M$~\cite{BKK2016}; see also~\cite{FF2001,Parwani2008}. Mapping class groups of once-punctured hyperbolic surfaces belong to $\GG^0(S^1)$; see~\cite{Nielsen1927,HT1985,BowditchSakuma}.

Simplicity of subgroups often plays a crucial role in the study of  group actions~\cite{Epstein1970, Thurston1974BAMS,BM1997,JM2013}.
Examples of countable simple groups in  $\GG^0(I)\setminus\GG^1(I)$ turn out to be abundant in isomorphism types.
For us, a \emph{continuum} means a set that has the cardinality of $\bR$. In joint work of the authors with Lodha~\cite{KKL2017} and in joint work of the second author with Lodha~\cite{KL2017}, the existence of a continuum of isomorphism types of finitely generated groups and of countable simple groups in $\GG^0(I)\setminus\GG^1(I)$ is established. These results relied on work of Bonatti--Lodha--Triestino~\cite{BLT2017}.
In particular, part (2) of Question~\ref{que0} has an affirmative answer for $k=0$ and $M=I$.

\subsection{Summary of results}
Recall that $M\in\{I,S^1\}$. 
In this article, we give the first construction of finitely generated groups and simple groups in $\GG^{\alpha}(M)\setminus\GG^{\beta}(M)$.

\begin{mainthm} For all $\alpha\in [1,\infty)$, each of the sets \[\GG^\alpha(M)\setminus\bigcup_{\beta>\alpha}\GG^\beta(M), \quad \bigcap_{\beta<\alpha}\GG^\beta(M)\setminus \GG^\alpha(M)\] contains a continuum of finitely generated groups,
and also contains a continuum of countable simple groups.
\end{mainthm}
The Main Theorem gives an affirmative answer to Question~\ref{que0}.


\begin{rem}\label{rem:not a group} One has to be slightly careful interpreting the Main Theorem when $\alpha=1$. This is because the set $\Diff_+^\beta(M)$ is not a group for $\beta<1$.  
Using~\cite{DKN2007}, we will prove a stronger fact that  
$\GG^{\mathrm{Lip}}(M)\setminus\GG^1(M)$
contains the desired continua.
Here, $\GG^{\mathrm{Lip}}(M)$ denotes the set of isomorphism types of countable subgroups of $\Diff_+^{\mathrm{Lip}}(M)$,
the \emph{group} of bi-Lipschitz homeomorphisms.\end{rem}

\begin{rem}
It is interesting to note that in the case of $M=I$, the simple groups guaranteed by the Main Theorem for $\alpha>1$ are locally indicable, as follows easily from Thurston Stability. Thus, we obtain a continuum of countable, simple, locally indicable groups. The commutator subgroup of Thompson's group $F$ is one such example.
\end{rem}

If $G\le \Diff_+^{\alpha}(M)$ and if $\beta>\alpha$, an injective homomorphism $G\to\Diff_+^{\beta}(M)$ is called an \emph{algebraic smoothing} of $G$. 
The Main Theorem implies that for each $\alpha\ge 1$, there exists a finitely generated subgroup  $G\le\Diff^\alpha_+(M)$ that admits no algebraic smoothings beyond $\alpha$. 
Moreover, the finitely generated groups in the continua of the Main Theorem can always be chosen to be non-finitely-presented
as there are only countably many finitely presented groups up to isomorphism.






In Section~\ref{ss:modulus} we give the definition of concave moduli (of continuity), a strict partial order $\ll$ between them, and the symbol $\succ_k 0$. For instance, $\omega_\tau(x)=x^\tau$ is a concave modulus satisfying $\omega_\tau\succ_k0$ for each $\tau\in(0,1]$ and $k\in\bN$.
For a concave modulus $\omega$, we let $\Diff_+^{k,\omega}(M)$ denote the group of $C^k$--diffeomorphisms on $M$ whose $k^{th}$ derivatives are $\omega$-continuous. We also write $\Diff_+^{k,0}(M):=\Diff_+^{k}(M)$.
We denote by $\Diff_+^{k,\mathrm{bv}}(I)$ the group of diffeomorphisms $f\in\Diff_+^k(I)$ such that 
$f$ has bounded total variation.
Note that $\Diff_+^{k,\mathrm{bv}}(I)$ contains $\Diff_+^{k,\mathrm{Lip}}(I)$, the group of $C^k$--diffeomorphisms whose $k^{th}$ derivatives are Lipschitz.

For a concave modulus $\omega$ or for $\omega\in\{0,\mathrm{bv}\}$, 
the set of all countable subgroups of $\Diff_+^{k,\omega}(M)$ is denoted as $\GG^{k,\omega}(M)$.
We will deduce the Main Theorem from a stronger, unified result as can be found below.

\begin{thm}\label{thm:main0}
For each $k\in\bN$,
and for each concave modulus $\mu\gg\omega_1$, there exists a finitely generated group  
$Q=Q(k,\mu)\le\Diff_+^{k,\mu}(I)$ such that the following hold.
\be[(i)]
\item
 $[Q,Q]$ is simple
 and every proper quotient of $Q$ is abelian;
\item
if $\omega=\mathrm{bv}$, or if $\omega$ is a concave modulus satisfying $\mu\gg \omega\succ_k 0$, then
\[
[Q,Q]\not\in
\GG^{k,\omega}(I)\cup \GG^{k,\omega}(S^1).
\]
\ee
\end{thm}

Theorem~\ref{thm:main0} will imply the Main Theorem after making suitable choices of $\mu$ above.
See Section~\ref{ss:continua} for details.

We let $F_n$ denotes a rank--$n$ free group.
Let $\BS(1,2)$ denote the solvable Baumslag--Solitar group of type $(1,2)$; see Section~\ref{sec:background}.
In the case when $M=I$,
our construction for Theorem~\ref{thm:main0} builds on a certain quotient of the group 
\[
G^\dagger = (\bZ\times\BS(1,2))\ast F_2.\]
Let us describe our construction more precisely.
\begin{thm}\label{thm:main}
Let $k\in\bN$, and let $\mu$ be a concave modulus  such that $\mu\gg\omega_1$.
Then there exists a representation 
\[\phi_{k,\mu}\co G^\dagger\to \Diff_+^{k,\mu}(I)\]
such that the following hold.
\be[(i)]
\item\label{p:main-ker}
If $\omega=\mathrm{bv}$, or if $\omega$ is a concave modulus satisfying $\mu\gg \omega\succ_k 0$, then
for all representations
\[\psi\co G^\dagger\to \Diff_+^{k,\omega}(I)\] 
we have that 
\[
\ker\psi\setminus\ker\phi_{k,\mu}\ne\varnothing.\]
\item \label{p:cinf}
Every diffeomorphism $f\in\phi_{k,\mu}(G^\dagger)$ is  $C^\infty$ on $I\setminus\partial I$.
\ee
\end{thm}

We deduce that the group
\[\phi_{k,\mu}(G^\dagger)\]
 admits no injective homomorphisms into $\Diff_+^{k,\omega}(I)$.
We will then bootstrap this construction to produce simple groups  in Section~\ref{sec:consequences}.


We define the \emph{critical regularity on $M$} of an arbitrary group $G$ as
\[\operatorname{CritReg}_M(G):=\sup\{\alpha\mid G\in\GG^{\alpha}(M)\}.\]
Here, we adopt the convention $\sup\varnothing=-\infty$.
The \emph{critical regularity spectrum of $M$} that is defined as
\[
\mathcal{C}_M:=
\left\{\operatorname{CritReg}_M(G)\mid
G\text{ is a finitely generated group }\right\}\]
Another consequence of the Main Theorem is the following.
\begin{cor}\label{cor:crit-reg}
The \emph{critical regularity spectrum of $M$}, which is defined as
\[
\mathcal{C}_M:=
\left\{\operatorname{CritReg}_M(G)\mid
G\text{ is a finitely generated group }\right\},\]
coincides with $\{-\infty\}\cup[1,\infty]$.\end{cor}

Theorem~\ref{thm:main} gives the first examples of groups whose critical regularities are determined (and realizable) and belong to $(1,\infty)$. 
To the authors' knowledge, 
the critical regularities of the following three groups are previously known and finite. 
First, Navas proved that 
Grigorchuk--Machi group $\bar H$ of intermediate growth has critical regularity $1$, and that the critical regularity of $\bar H$ can be realized~\cite{Navas2008GAFA}.
Second, Castro--Jorquera--Navas proved (\cite{CJN2014}, combined with~\cite{PT1976}) that the integral Heisenberg group has critical regularity $2$ and this critical regularity cannot be attained.
Thirdly, Jorquera, Navas and Rivas~\cite{JNR2018} proved that the nilpotent group $N_4$ of $4\times 4$ integral lower triangular matrices with ones on the diagonal satisfies
\[\operatorname{CritReg}_I(N_4)=3/2.\]
It is not known whether or not the critical regularity $3/2$ of $N_4$ is realizable.

The case $G\in\GG^1(M)\setminus\GG^0(M)$ requires a suitable interpretation the critical regularity.
As we have mentioned in Remark~\ref{rem:not a group}, it is proved by Deroin, Kleptsyn and Navas that every countable subgroup $G$ of $\Homeo_+(M)$ is topologically conjugate to a group of bi--Lipschitz homeomorphisms~\cite{DKN2007}.
Thus, it is reasonable to say that $[0,1)$ is missing from from the critical regularity spectrum.



The authors proved in \cite{KKFreeProd2017} that 
for each integer $2\le k\le\infty$, 
the class of finitely generated group in $\GG^k(M)$ is not closed under taking finite free products.
From~\cite{BMNR2017MZ} and from the consideration of $\BS(1,2)$ actions in the current paper, we deduce the following augmentation for $k=1$.
We are grateful to A.~Navas for pointing us to the reference \cite{BMNR2017MZ} and telling us the proof of the following corollary for $M=I$. See Section~\ref{ss:univ} for details.

\begin{cor}\label{cor:c1-freeprod}
The group $(\bZ\times\BS(1,2))\ast\bZ$ does not embed into $\Diff_+^1(M)$.
In particular, the class of finitely generated subgroups of $\Diff_+^1(M)$ is not closed under taking finite free products.
\end{cor}

Though we concentrate primarily on countable groups, our results have applications to continuous groups. 
For a smooth manifold $X$ and for an $\alpha\ge1$, we let $\Diff_c^{\alpha}(X)_0$ denote the group of $C^\alpha$ diffeomorphisms of $X$ isotopic to the identity through a compactly supported $C^\alpha$ isotopy.
If $1\le \alpha<\beta$, then there is a natural embedding $\Diff_c^\beta(X)_0\to\Diff_c^{\alpha}(X)_0$ defined simply by the  inclusion.
The main result (and its proof) of~\cite{Mann2015} by Mann implies that if $X\in\{S^1,\bR\}$, and if $2<\alpha<\beta$ are real numbers, then there exists no injective homomorphisms $\Diff_c^{\alpha}(X)_0\to\Diff_c^{\beta}(X)_0$. We generalize this to all real numbers $1\le \alpha<\beta$.

\begin{cor}\label{cor:intro-continuous}
Let $X=\{S^1,\bR\}$.
Then arbitrary homomorphisms of the following types have abelian images:
\be
\item $\Diff_c^\alpha(X)_0\to\bigcup_{\beta>\alpha}\Diff_c^\beta(X)_0$, where $\alpha\ge1$;
\item $\Diff_c^\alpha(X)_0\to\Diff_c^{\lfloor \alpha\rfloor,\mathrm{bv}}(X)_0$, where $\alpha\ge1$;
\item $\bigcap_{\beta<\alpha}\Diff_c^\beta(X)_0\to \Diff_c^\alpha(X)_0$, where $\alpha>1$.
\ee
In addition, if $\alpha\ne2$ in parts (1) and (2), and if $\alpha>3$ in part (3), then all the above homomorphisms have trivial images.
\end{cor}

The Main Theorem has the following implication on the existence of unsmoothable foliations on 3--manifolds. This extends a previous result of Tsuboi~\cite{Tsuboi1987} and of Cantwell--Conlon~\cite{CC1988}, that is originally proved for integer regularities.
\begin{cor}\label{cor:foliation} Let $\alpha\ge1$ be a real number.
Then for every closed orientable $3$--manifold $Y$ satisfying $H_2(Y,\bZ)\ne0$, there exists a codimension--one $C^{\alpha}$ foliation $(Y,\FF)$ which is not homeomorphic to a $\bigcup_{\beta>\alpha}C^{\beta}$ foliation. \end{cor}

Here, a homeomorphism of foliations is a homeomorphism of the underlying foliated manifolds which respects the foliated structures.

\subsection{Notes and references}
\subsubsection{Automatic continuity}
K. Mann proved that if $X$ is a compact manifold then the group $\Homeo_0(X)$ of homeomorphisms isotopic to the identity has \emph{automatic continuity}, so that every homomorphism from $\Homeo_0(X)$ into a separable group is continuous~\cite{Mann2016GT}. She uses this fact to prove that $\Homeo_0(X)$ has critical regularity $0$ and hence has no algebraic smoothings. For discussions of a similar ilk, the reader may consult~\cite{MannETDS15} and~\cite{HurtadoGT15}.
The Main Theorem implies that the critical regularity of $\Diff_+^\alpha(M)$ is $\alpha$, for $M\in\{I,S^1\}$ and for  $\alpha\ge1$.

\subsubsection{Superrigidity}

Recall that Margulis Superrigidity says that under suitable hypotheses, a representation of a lattice $\Gamma$ in a higher rank Lie group $G$ is actually given by the restriction of a representation of $G$ to $\Gamma$ (see~\cite{MargulisBook1991}). For the continuous groups $\Diff_+^{\alpha}(M)$ which we consider here, there is no particularly clear analogue of a lattice. Nevertheless, some of the results proved in this paper are reminiscent of similar themes. Particularly, Corollary~\ref{cor:intro-continuous} is established by showing that all of the maps in question contain a countable simple group (perhaps a suitable analogue of a lattice) in their kernel, thus precluding the existence of a nontrivial homomorphism between the corresponding continuous groups.

\subsubsection{Topological versus algebraic smoothability}

The smoothability issues that we consider in this paper center around algebraic smoothability of group actions. There is a stronger notion of smoothability called \emph{topological smoothability}. 
A topological smoothing of a representation
\[
\phi\co G\to  \Diff_+^{\alpha}(M)\]
is a topological conjugacy of $\phi$ into $\Diff_+^{\beta}(M)$ for some $\beta>\alpha$; that is, the conjugation 
$h\phi h^{-1}$ of $\phi$ by some homeomorphism $h$ on $M$ such that we have $h\phi(G)h^{-1}\le \Diff_+^\beta(M)$.
A topological smoothing of a subgroup is obviously an algebraic smoothing, but not conversely; compare \cite{CJN2014} and \cite{JNR2018}. 
By a result of Tsuboi~\cite{Tsuboi1987}, there exists a two--generator solvable group $G$ and a faithful action $\varphi_k$ of $G$ on the interval such that $\varphi_k(G)\le \Diff_+^k(I)$ but such that $\varphi_k(G)$ is not topologically conjugate into $\Diff_+^{k+1}(I)$. Since $\varphi_k$ is injective, these actions are algebraically smoothable. See Section~\ref{ss:fol} regarding implications for foliations.

\subsubsection{Disconnected manifolds}

It is natural to wonder whether or not the results of this paper generalize to compact one--manifolds which are not necessarily connected; these manifolds are disjoint unions of finitely many intervals and circles (cf.~\cite{BKK2016,KKFreeProd2017}). It is not difficult to see that the results generalize. Indeed, if $G$ is a group of homeomorphisms of a compact disconnected one--manifold $M$, then a finite index subgroup of $G$ stabilizes all the components of $M$. We build a finitely generated group $G$ whose commutator subgroup $[G,G]$ is simple, and such that $[G,G]$ has the critical regularity  exactly $\alpha$ with respect to faithful actions on the interval or the circle. Some finite index subgroup of $G$ stabilizes each component of $M$, and since $[G,G]$ is infinite and simple, $[G,G]$ stabilizes each component of $M$. 
It follows that 
$G$ has critical regularity $\alpha$ with respect to faithful actions on $M$.

\subsubsection{Kernel structures}

In Theorem~\ref{thm:main}, let us fix $\epsilon\in(0,1)$ such that $\omega_\epsilon\ll\mu$.
It will be impossible to find a finite set $S\sse G^\dagger\setminus\ker\phi$ such that for all  $\psi\in\Hom(G^\dagger,\Diff_+^{k+\epsilon}(I))$ we have $S\cap\ker\psi\ne\varnothing$.
Indeed,  Lemma~\ref{lem:res} implies that for all finite set $S\sse G^\dagger\setminus\{1\}$
there exists a $C^{\infty}$ action of $G^\dagger$ on $\bR$ with a compact support such that $S$ does not intersect the kernel of this action.
So, one must consider an infinite set of candidates that could be a kernel element of such a $\psi$.

\subsection{Outline of the proof of Theorem~\ref{thm:main}}\label{subsec:outline}
Given a concave modulus  $\mu$,
we build a certain representation $\phi$ of the group $G^\dagger$ into $\Diff_+^{k,\mu}(I)$. 
For $\epsilon\in(0,1]$ satisfying  $\omega:=\omega_\epsilon\ll\mu$,
we also show that the group $\phi(G^\dagger)$ admits no algebraic smoothing into $\Diff_+^{k,\omega}(I)$. 
We remark that $\Diff_+^{k+1}(I)\le\Diff_+^{k,\omega}(I)$.

To study maps into $\Diff_+^{k,\omega}(I)$,
we use a measure of complexity of a diffeomorphism $f$, which is roughly the number of components of supports of generators of $G^\dagger$ needed to cover the support of $f$. 
We prove a key technical result governing this complexity;
this result is called the Slow Progress Lemma 
and applies to an action of an arbitrary finitely generated group on $I$.
To have a starting diffeomorphism with finite complexity, we build an element $1\neq u\in G^\dagger$ such that if $\psi\colon G^\dagger\to\Diff_+^1(I)$ is an arbitrary representation then the support of $\psi(u)$ is compactly contained in the support of $\psi(G^\dagger)$.

Next, we build an action $\phi$ of $G^\dagger$ so that certain judiciously chosen conjugates $w_juw_j^{-1}$ of $u$, which depend strongly on the regularity $(k,\mu)$, result in a sequence of diffeomorphisms $\phi(w_juw_j^{-1})$ whose complexity grows linearly in $j$. We show that under an arbitrary representation $\psi\colon G^\dagger\to\Diff_+^{k,\omega}(I)$, the complexity of $\psi(w_juw_j^{-1})$ grows more slowly than that of $\phi(w_juw_j^{-1})$, a statement which follows from the Slow Progress Lemma.
Thus for each $\psi$, we find an element $g\in G^\dagger$ which survives under $\phi$ but dies under $\psi$. In particular, $\phi(G^\dagger)$ cannot be realized as a subgroup of $\Diff_+^{k,\omega}(I)$.

\subsection{Outline of the paper}\label{subsec:outline2}

We strive to make this article as self--contained as possible. In Section~\ref{sec:analytic}, we build up the analytic tools we need. Section~\ref{sec:background} summarizes the dynamical background used in the sequel, and proves Corollary~\ref{cor:c1-freeprod}. Section~\ref{sec:psi} establishes the Slow Progress Lemma for a general finitely generated group action on intervals. In Section~\ref{sec:phi}, we fix a concave modulus  $\mu$, and construct a representation $\phi$ of the group $G^\dagger$ into $\Diff_+^{k,\mu}(I)$ with desirable dynamical properties and prove Theorem~\ref{thm:main}. In Section~\ref{sec:consequences}, we complete the proof of the Main Theorem and gather the various consequences of the main results.

\section{Probabilistic dynamical behavior}\label{sec:analytic}
Throughout this section and for the rest of the paper, we will let $I$ denote a nonempty compact subinterval of $\bR$. 
\emph{All homeomorphisms considered in this paper are assumed to be orientation preserving}.
We continue to let  $M=I$ or $M=S^1$.

We wish to develop the concepts of \emph{fast} and \emph{expansive} homeomorphisms (Definition~\ref{defn:fast-exp}). These concepts establish a useful relationship between the dynamical behavior of a diffeomorphism supported on $I$ and its analytic behavior, which is to say its regularity.

\subsection{Moduli of continuity}\label{ss:modulus}
We will use the following notion in order to guarantee the convergence of certain sequences of diffeomorphisms.

\bd\label{defn:mod}
\be
\item
A \emph{concave modulus of continuity} (or \emph{concave modulus}, for short)
means
a homeomorphism $\omega\co [0,\infty)\to[0,\infty)$
which is concave.
\item
Let $\omega$ be a concave modulus .
For $U\sse \bR$ or $U\sse S^1$,
we define the \emph{$\omega$--norm} of a map $f\co U\to \bR$ as
\[
[f]_{\omega}=\sup \left\{
\frac{|f(x)-f(y)|}{\omega(|x-y|)}
\co 
x,y\in U\text{ and }x\ne y
\right\}
.\]
We say $f$ is \emph{$\omega$-continuous} if $f$ has a bounded $\omega$--norm.
\ee
\ed

The notion of $\omega$--continuity depends only on the germs of $\omega$ for bounded functions,
as can be seen from 
the following easy observation.

\begin{lem}\label{lem:omega-cont}
Let $\omega$ be a concave modulus ,
and let $f\co U\to\bR$ be a bounded function for some $U\sse\bR$.
If there exist constants $K,\delta>0$
such that
\[
|f(x)-f(y)| \le K\cdot \omega(|x-y|)\]
for all $0<|x-y|\le \delta$,
then we have $[f]_\omega<\infty$.
\end{lem}

\begin{rem}
It is often assumed in the literature that a concave modulus $\omega(x)$ is defined only locally at $x=0$, namely on $[0,\delta]$ for some $\delta>0$~\cite{Mather1,Mather2}. This restriction does not alter the definition of $\omega$--continuity for compactly supported functions. The reason goes as follows. Suppose $\omega\co [0,\delta]\to[0,\omega(\delta)]$ is a strictly increasing concave homeomorphism. By an argument in the proof of Lemma~\ref{lem:conc-mod}, we can find a concave modulus $\mu\co [0,\infty)\to[0,\infty)$ such that \[\omega(s)\le \mu(s)\le \left(2+\delta/\omega(\delta)\right)\omega(s)\] for all $s\in[0,\delta]$.
By Lemma~\ref{lem:omega-cont}, we conclude that the $\omega$--continuity coincides with the $\mu$--continuity for a compactly supported function.
\end{rem}

The complex plane $\bC$ has a natural lexicographic order $<_\bC$;
that is, we write $z<_\bC w$  in $\bC$ if 
$\operatorname{Re}z<\operatorname{Re}w$,
or if
$\operatorname{Re}z=\operatorname{Re}w$ and 
$\operatorname{Im}z<\operatorname{Im}w$. 
For two complex numbers $a,b\in\bC$, we let 
\[(a,b]_\bC:=\{z\in\bC \mid a<_\bC z\le_\bC b\}.\]
In particular, we have that
\[
(0,1]_\bC := \{s\sqrt{-1}\mid s>0\}
\cup\{ \tau+s\sqrt{-1}\mid \tau\in(0,1), s\in\bR\}
\cup\{1+s\sqrt{-1}\mid s\le 0\}.\]
We similarly define $(a,b)_\bC$, together with the other types of intervals.

\begin{exmp}\label{ex:mod}
Let $z=\tau + s\sqrt{-1}\in\bC$ satisfy $z\in(0,1]_\bC$.
We set
\[\omega_z(x):= x^\tau\cdot \exp\left(-s \log (1/x) / \log\log (1/x)\right).\]
Then $\omega_z$ is a small perturbation of $\omega_\tau(x)=x^\tau=\exp(-\tau\log(1/x))$.
By simple computations of the derivatives, one sees that $\omega_z$ is a concave modulus defined for all small $x\ge0$.
See Figure~\ref{fig:om} for the graphs of $\omega_z$.
\end{exmp}

We will use the notation in Example~\ref{ex:mod} for the rest of the paper. 
The H\"older continuity of exponent $\tau\in(0,1)$ is equivalent to the $\omega_\tau$--continuity.

\begin{figure}[b!]
\subfloat[(a) $y=\omega_{0.05+0.2\sqrt{-1}}(x)$]
{\includegraphics[width=.4\textwidth]{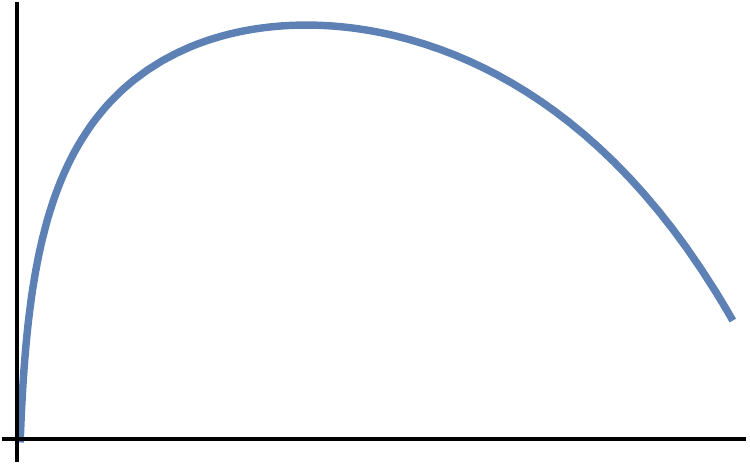}}
$\qquad\qquad$
\subfloat[(b) $y=\omega_{0.05-0.05\sqrt{-1}}(x)$]
{\includegraphics[width=.4\textwidth]{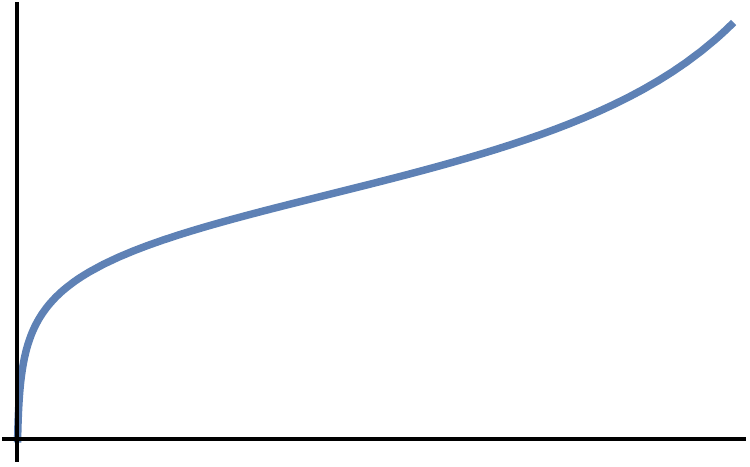}}
\caption{The graphs of $\omega_z$ along with their extrapolations (not drawn in scale). Note we only consider concave and strictly increasing portions $[0,\delta]$ of the above graphs.}
\label{fig:om}
\end{figure}

\begin{notation}\label{notation:eps}
\be
\item
Let $k\in\bN$, and let $\omega$ be a concave modulus.
We write
\[\omega\succ_k 0\]
if the following holds for some $\delta>0$:
\[
\lim_{t\to+0}\sup_{0<x<\delta}{t^{k-1}\omega(tx)}/{\omega(x)}=0.\]
\item
For two positive real sequences  $\{a_j\}$ and $\{b_j\}$,
we will write $\{a_j\} \precsim \{b_j\}$ if  $\{a_j / b_j\}$ is bounded.
\ee
\end{notation}

In particular, the expression $\omega\succ_k 0$ is vacuously true for $k>1$. Compare this condition to Mather's Theorem (Definition~\ref{defn:sup-sub} and Theorem~\ref{thm:mather}).

\begin{lem}\label{lem:omega}
The following hold for $k\in\bN$
and
for a concave modulus  $\omega$.
\be
\item
The function $x/\omega(x)$ is monotone increasing
on $[0,\infty)$.
\item
For all $C>0$ and $x\ge 0$, we have
$\omega(Cx)\le  (C+1)\omega(x)$.
\item\label{p:aibi} 
Assume that we have positive sequences $\{a_j\}$ and $\{b_j\}$ 
 such that 
 \[\{a_j^{k-1}\omega(a_j)\}\precsim\{b_j^{k-1}\omega(b_j)\}.\]
If $\omega\succ_k 0$,
then we have $\{a_j\}\precsim\{b_j\}$.
\ee
\end{lem}
\bp
Proofs of (1) and (2) are obvious from monotonicity and concavity.
Assume (3) does not hold.
Passing to a subsequence, we may assume  $\{t_j:=b_j/a_j\}$ converges to $0$. Then we have a contradiction because
\[
\frac{b_j^{k-1}\omega(b_j)}{ a_j^{k-1}\omega( a_j)}
=
t_j^{k-1}
\cdot
\frac{\omega(t_ja_j)}{\omega(a_j)}
\to 0\text{ as }j\to\infty.\qedhere\]
\ep

Suppose $\omega$ and $\mu$ are concave moduli. 
We define a strict partial order $\omega\ll\mu$ if 
\[
\lim_{x\to+0}\frac{\omega(x)\log^K(1/x)}{\mu(x)}=0\]
for all $K>0$.
Here, we use the notation 
\[
\log^K t = (\log t)^K.\]

\begin{lem}\label{lem:order}
If $z,w\in(0,1]_\bC$ satisfy $z<_\bC w$,
then $\omega_z\gg\omega_w$.
\end{lem}

\bp
Let $z={\sigma+s\sqrt{-1}}$ and $w={\tau+t\sqrt{-1}}$.
Then we have
\begin{align*}
&\lim_{x\to+0}\log\left(\omega_w(x)\log^K(1/x)/\omega_z(x)\right)\\
&=\lim_{x\to+0}
(\sigma-\tau)\log(1/x)-(t-s)\log(1/x)/\log\log(1/x)+K\log\log(1/x)\\
&=\lim_{y\to\infty}
(\sigma-\tau)y-(t-s)y/\log y+K\log y.
\end{align*}
From $z<_\bC w$, we see that the above limit equals $-\infty$. This is as desired.
\ep

Let $k\in\bN$ and let $\omega$ be a concave modulus.
A \emph{$C^{k,\omega}$--diffeomorphism} on $M$ is defined as a diffeomorphism $f$ of $M$ such that $f^{(k)}$ is $\omega$-continuous. We say the pair $(k,\omega)$ is a \emph{regularity} of $f$. If $\omega=\omega_\tau$ for some $\tau\in(0,1)$ then 
a $C^{k,\omega}$--diffeomorphism means a  $C^{k+\tau}$--diffeomorphism.
We have $C^{k,\omega_1}=C^{k,\mathrm{Lip}}$.

Let $f\co I=[p,q]\to\bR$ be a map. Recall that the \emph{(total) variation} of $f$ is given by \[\operatorname{Var}(f,I)=\sup_{p=x_1<\cdots<x_n=q} \sum_i |f(x_i)-f(x_{i-1})|,\] where the supremum is taken over all possible finite partitions of $I$. 
A function has bounded variation on $I$ if $\operatorname{Var}(f,I)$ is finite on $I$. 
If $M=S^1$, we use the same definition for  $\operatorname{Var}(f,I)$ with $p=q$.
We say that a diffeomorphism $f\co M\to M$ is $C^{k,\mathrm{bv}}$ if $f$ is $C^k$ and
if in addition we have $f^{(k)}$ has bounded variation.

Let $\omega$ be a concave modulus, or let $\omega=\mathrm{bv}$.
We write  for 
The set of all $C^{k,\omega}$ diffeomorphisms of $M$
is denoted as
\[\Diff^{k,\omega}_+(M),\]
which turns out to be a group for $k\in\bN$ (Proposition~\ref{prop:group}). 
We define $\GG^{k,\omega}(M)$ to be the set of the isomorphism classes of countable subgroups of $\Diff_+^{k,\omega}(M)$.

Note that 
\[\Diff^{k+\tau}_+(M)=\Diff^{k,\omega_\tau}_+(M).\]
We have that
\[\Diff^{k+1}_+(M)\le \Diff^{k,\omega_1}_+(M)=\Diff^{k,\mathrm{Lip}}_+(M)\le  \Diff^{k,\mathrm{bv}}_+(M)\le\Diff^{k}_+(M).\]
If we have two concave of moduli $\omega\ll\mu$, then we have
\[\Diff^{k,\omega}_+(M)\le\Diff^{k,\mu}_+(M).\]
In particular, if $z,w\in(0,1]_\bC$ satisfy $z<_\bC w$, then we see from Lemma~\ref{lem:order} that
\[\Diff^{k,\omega_{z}}_+(M)\ge\Diff^{k,\omega_w}_+(M).\]

\subsection{Fast and expansive homeomorphisms}\label{ss:fast}
From now on until Section~\ref{sec:consequences}, we will be mostly concerned with the case $M=I$.
For a measurable set $J\sse \bR$, we denote by $|J|$ its Lebesgue measure. 
We write $J'$ for the \emph{derived set} of $J$, which is to say the set of the accumulation points of $J$.
If $X$ is a set, we let $\# X$ denote its cardinality.

Let $f\co X\to X$ be a map on a space $X$.
We use the standard notations
\begin{align*}
\Fix f &=\{x\in X \mid f(x)=x\},\\
\supp f &=\{x\in X\mid f(x)\neq x\}=X\setminus\Fix f.
\end{align*}
The set $\supp f$ is also called the \emph{(open) support} of $f$.
We note the identity map $\Id\co \bR\to\bR$ satisfies
$\Id^{(j)}(x)=\delta_{1j}$ for $j\ge1$.

\bd\label{defn:fast-exp}
Let $f\co I\to I$ be a homeomorphism,
and let $J\sse I$ be a compact interval
such that $f(J)=J$.
We let $k\in\bN$.
\be
\item
We say $f$ is \emph{$k$--fixed on $J$} if one of the following holds:
\begin{itemize} \item $J\cap (\Fix f)'\ne\varnothing$, or \item $\#({J}\cap\Fix f)> k$. \end{itemize}
\item
We say $f$ is \emph{$\delta$--fast on $J$} for some $\delta>0$ if  \[\sup_{y\in J}\frac{|f(y)-y|}{|J|}\ge\delta.\]
\item
We say $f$ is \emph{$\lambda$--expansive on $J$} for some $\lambda>0$ if  \[\sup_{y\in J}\frac{|f(y)-y|}{d(\{y,f(y)\},\partial J)}\ge\lambda.\]
\ee
\ed

We note that $f$ has one of the above three properties if and only if so does $f^{-1}$.
Note also that $f$ is $\lambda$--expansive on $J=[p,q]$ if and only if
there exists some $y\in J$ satisfying one of the following (possibly overlapping) alternatives:
\begin{enumerate}[(i)]
\item[(E1)]
$p<y<f(y)<q$ and $f(y)-y \ge \lambda(y-p)$;
\item[(E2)]
$p<y<f(y)<q$ and $f(y)-y \ge \lambda(q-f(y))$;
\item[(E3)]
$p<f(y)<y<q$ and $y-f(y) \ge \lambda(f(y)-p)$;
\item[(E4)]
$p<f(y)<y<q$ and $y-f(y) \ge \lambda(q-y)$.
\end{enumerate}

For a set $A\sse\bN$, 
we define its \emph{natural density} 
as 
\[d_\bN(A)=\lim_{N\to\infty}\#(A\cap[1,N])/N,\]
if the limit exists. 
A crucial analytic tool of this paper is 
the following probabilisitic description of fast and expansive homeomorphisms.

\begin{thm}\label{thm:est}
Let $k\in\bN$, and let $\omega\succ_k 0$ be a concave modulus.
Suppose we have
\be[(i)]
\item a diffeomorphism
$f\in\Diff_+^{k,\omega}(I)\cup\Diff_+^{k,\mathrm{bv}}(I)$;
\item\label{sup1} a sequence $\{N_i\}\sse\bN$
such that $\sup_{i\in\bN}N_i  (1/i)^{k-1}\omega(1/i)<\infty$;
\item\label{sup2} a sequence of compact intervals $\{J_i\}$ in $I$
such that
$f$ is $k$--fixed on each $J_i$
and such that 
 $\sup_{i\in\bN}\#\{j\in\bN\mid J_i\cap J_j\ne\varnothing\}<\infty$.
\ee
Then for each $\delta>0$ and $\lambda>0$, the following set has the natural density zero:
\[
A_{\delta,\lambda}
=\left\{
i\in\bN \mid
f^{N_i}\text{ is }\delta\text{--fast or }\lambda\text{--expansive on }J_i
\right\}.\]
\end{thm}

The proof of the theorem is given in Section~\ref{ss:proof-thmest}.

\subsection{Proof of Theorem~\ref{thm:est}}\label{ss:proof-thmest}
Let $k$ and $\omega$ be as in Theorem~\ref{thm:est}.
We first note a classical result in number theory.
\begin{lem}\label{lem:Salat}
For sets $A,B\sse \bN$, the following hold.
\be
\item
If $d_\bN(A)=1$ for some $A\sse \bN$ and if $i\in\bN$, then $d_\bN\left((A-i)\cap\bN\right)=1$.
\item
If $d_\bN(A)=d_\bN(B)=1$ for some $A,B\sse\bN$, then
 $d_\bN(A\cap B)=1$.
\item (\cite{Salat1964MZ,Moser1958})
If $\sum_{i\in A}1/i$ is convergent, then $d_\bN(A)=0$.
\ee
\end{lem}

Fastness and expansiveness constants of ``roots'' of a diffeomorphism behave like arithmetic and geometric means, respectively:
\begin{lem}\label{lem:exp-N}
Let $f\in\Homeo_+(J)$ for some compact interval $J$,
and let $N\in\bN$.
\be
\item
If $f^N$ is $\delta$--fast for some $\delta>0$,
then $f$ is $(\delta/N)$--fast.
\item
If $f^N$ is $\lambda$--expansive for some $\lambda>0$,
then $f$ is $((\lambda+1)^{1/N}-1)$--expansive.
\ee
\end{lem}
\bp
Let us write $J=[p,q]$.

(1) For some $y\in J$ we have
\[\delta |J| \le |f^N(y)-y|\le\sum_{i=0}^{N-1} |f^{i+1}(y)-f^{i}y|.\]
Hence there exists some $y'=f^i(y)$ such that $|f(y')-y'|\ge \frac\delta{N} |J|$.

(2)
Assume the alternative (E1) holds as described after Definition~\ref{defn:fast-exp}.
That is, 
\[p<y<f(y)<q\]
for some $y\in J$ such that $ f^N(y)-y\ge \lambda (y-p)$. Note that
\[
\lambda +1 \le \frac{f^N(y)-p}{y-p} 
=\prod_{i=0}^{N-1} \frac{f^{i+1}(y)-p}{f^{i}(y)-p}.\]
So, for some $y'=f^i(y)$, we have
\[(\lambda+1)^{1/N}\le \frac{f(y')-p}{y'-p}=\frac{f(y')-y'}{y'-p}+1.\]
This is the desired inequality. The other alternatives are similar.\ep

\begin{lem}\label{lem:kfix-root}
For a $C^{k}$--map $f\co I\to \bR$, the following hold.
\be
\item\label{part:acc}
If $x\in(\Fix f)'$
and 
$j=0,1,\ldots,k$, then we have:
\[f^{(j)}(x)=\Id^{(j)}(x).\]
\item
If $f$ is $k$--fixed on a compact interval $J\sse I$, then $(f-\Id)^{(j)}$ has a root in $J$ for each $j=0,1,\ldots,k$.
\ee
\end{lem}
\bp 
For each $j\in\{0,1,\ldots,k\}$, we define
\[S_j\co =S_j(f)=\{x\in I\mid f^{(j)}(x)=\Id^{(j)}(x)\}.\]
(1) We have $S_j'\sse S_j$.
It now suffices for us to show the following:
\[
S_0'=(\Fix f)'\sse S_1'\sse\cdots \sse S_k'.\]
Let us assume $x\in S_j'$ for some $0\le j<k$.
Then there exists a sequence $\{x_i\}\sse S_j\setminus\{x\}$ converging to $x$. There exists $y_i$ between $x_i$ and $x$ such that
\[
f^{(j+1)}(y_i)=\frac{f^{(j)}(x_i)-f^{(j)}(x)}{x_i-x}
=\frac{\Id^{(j)}(x_i)-\Id^{(j)}(x)}{x_i-x}=\delta_{0j}=\Id^{(j+1)}(y_i).\]
Since $y_i\in S_{j+1}$ converges to $x$, we see that $x\in S_{j+1}'$.
This proves $S_j'\sse S_{j+1}'$.

(2)
By part (1), it suffices to consider the case that $\#(J\cap \Fix f)\ge k+1$.
We inductively observe 
that $(f-\Id)^{(j)}$ has at least $(k+1-j)$ roots for each $j=0,1,\ldots,k$
 by the Mean Value Theorem.
\ep

\begin{lem}\label{lem:est-ck}
Let $J\sse I$ be a compact interval, and let $\delta,\lambda>0$.
Suppose $f\in\Diff_+^k(I)$ is $k$--fixed on $J$.
\be
\item
If $f$ is $\delta$--fast on $J$, then 
\[
 \sup_J|f^{(k)}-\Id^{(k)}| \cdot|J|^{k-1}\ge\delta.\]
If, furthermore, $f$ is $C^{k,\omega}$ then we have
\[   \left[f^{(k)}\right]_{\omega}\cdot|J|^{k-1}\omega(|J|)\ge\delta.\]
\item
If $f$ is  $\lambda$--expansive on $J$, then
\[
\max\left( \sup_J |f^{(k)}-\Id^{(k)}| ,\sup_J|(f^{-1})^{(k)}-\Id^{(k)}| \right) \cdot |J|^{k-1}\ge \lambda.\]
If, furthermore, $f$ is $C^{k,\omega}$ then we have
\[
\max\left(\left[f^{(k)}\right]_{\omega},\left[({f^{-1}})^{(k)}\right]_{\omega}\right) 
\cdot |J|^{k-1}\omega(|J|)\ge \lambda.\]\ee
\end{lem}
\bp
For each $j\le k$, Lemma~\ref{lem:kfix-root} implies that there exists $s_j\in J$ satisfying
\[
f^{(j)}(s_j)=\Id^{(j)}(s_j).\]
Let $y_0\in J$ be arbitrary.
We see (cf. Lemma~\ref{lem:ckw}) that
\begin{align*}
\left|f(y_0)-y_0\right|&
=\left|\int_{t_1=s_0}^{y_0}
\int_{t_2=s_1}^{t_1}
\cdots
\int_{t_k=s_{k-1}}^{t_{k-1}} \left(
f^{(k)}(t_k)-f^{(k)}(s_k)\right)dt_k\;dt_{k-1}\cdots dt_1\right|
\\
&
\le
\sup_J|f^{(k)}-\Id^{(k)}|
\cdot |y_0-s_0|\cdot |J|^{k-1}.\end{align*}

(1) Pick $y_0\in J$ such that $|f(y_0)-y_0|\ge\delta|J|$.
We see
\[\delta|J|\le |f(y_0)-y_0|\le   \sup_J|f^{(k)}-\Id^{(k)}| \cdot |J|^{k}.\]
If $f$ is $C^{k,\omega}$, then we further deduce that
\[
\delta|J|
\le
\sup_{t\in J}|f^{(k)}(t)-f^{(k)}(s_k)|\cdot |J|^{k}
\le 
 \left[f^{(k)}\right]_{\omega}\cdot|J|^{k}\omega(|J|).\]

(2)
Write $J=[p,q]$. Assume the alternative (E1) holds for $y_0\in J$;
that is,
\[
\lambda(y_0-p) \le f(y_0)-y_0.\]
By applying the same estimate for $s_0=p$, 
we see that
\[\lambda\le \frac{f(y_0)-y_0}{y_0-p}
\le
  \sup_J|f^{(k)}-\Id^{(k)}|
\cdot    |J|^{k-1}    .\]
If $f$ is $C^{k,\omega}$, we further have
\[
\lambda\le
 \sup_{t\in J}|f^{(k)}(t)-f^{(k)}(s_k)|
 \cdot
 |J|^{k-1}
 \le
 \left[f^{(k)}\right]_{\omega}
\cdot   |J|^{k-1}\omega(|J|).\]
The other alternatives can be handled in the same manner; in particular, we use the diffeomorphism $g=f^{-1}$ for (E2) and (E3).
\ep

\bp[Proof of Theorem~\ref{thm:est}: $C^{k,\omega}$ case]
We assume $f\co I\to I$ is a $C^{k,\omega}$--diffeomorphism.
Let $\delta,\lambda>0$, and define
\begin{align*}
L&=
\max\left( \left[f^{(k)}\right]_{\omega}, \left[(f^{-1})^{(k)}\right]_{\omega}\right),\\
A_\delta  &=\{ i\in\bN\mid f_i^{N_i}\text{ is }\delta\text{--fast on }J_i\},\\
B_\lambda  &=\{ i\in\bN\mid f_i^{N_i}\text{ is }\lambda\text{--expansive on }J_i\}.\end{align*}
We let $K>0$ be the larger value of the suprema in
the conditions (\ref{sup1}) and (\ref{sup2}).
The following claim is obvious from (\ref{sup2}) and from a maximality argument. 
\begin{claim}\label{cla:disj}
The sequence of intervals $\{J_i\}$ can be partitioned into at most $K$ collections such that each collection consists of disjoint intervals.
In particular, we have
 \[\sum_i |J_i| \le K|I|.\]
\end{claim}

It now suffices for us to establish the two claims below.
\begin{claim}\label{cla:adelta}
$d_\bN(A_\delta)=0$.
\end{claim}

By Lemmas~\ref{lem:exp-N} and~\ref{lem:est-ck}, we have that
\[\{(1/i)^{k-1}\omega(1/i)\co i\in A_\delta\}\precsim\{  1/{N_i} \co i\in A_\delta\}\precsim\{|J_i|^{k-1}\omega(|J_i|) \co i\in A_\delta\}.\]
By Lemma~\ref{lem:omega} (\ref{p:aibi}),
there exists $L'>0$ such that
 $1/i\le L' |J_i|$ for $i\in A_\delta$.
So,
\[
\sum_{i\in A_\delta}1/i
\le
\sum_{i\in A_\delta}L'|J_i|\le L'K|I|<\infty.\]
Lemma~\ref{lem:Salat} now implies the claim.

\begin{claim}\label{cla:bdelta}
$d_\bN(B_\lambda)=0$.
\end{claim}

There is a constant $K_0>0$ such that 
\[ 
\log\left(1+K_0 (1/i)^{k-1}\omega(1/i)\right)
\le
 K_0 (1/i)^{k-1}\omega(1/i) \le
 {\log(\lambda+1)}/{N_i} 
.\]
Hence,
 Lemmas~\ref{lem:exp-N} and~\ref{lem:est-ck} imply that
\[
\{ (1/ i)^{k-1}\omega(1/i) \co i\in B_\lambda\}
\precsim
\{
(\lambda+1)^{1/N_i}-1
\co i\in B_\lambda\}
\precsim
\{
|J_i|^{k-1}\omega(|J_i|)
\co i\in B_\lambda\}.
\]
As in Claim~\ref{cla:adelta}, we have $\sum_{B_\lambda} 1/i<\infty$ and $d_\bN(B_\lambda)=0$.
\ep

\bp[Proof of Theorem~\ref{thm:est}: $C^{k,\mathrm{bv}}$ case]
We now assume $f$ is a $C^{k,\mathrm{bv}}$--diffeomorphism.
Let us closely follow the proof of $C^{k,\omega}$ case, using the same notation.
In particular, we define the same sets $A_\delta$ and $B_\lambda$.

For each $i\in \bN$, we pick $x_i,y_i,z_i\in J_i$ such that
\[
|f^{(k)}(x_i)-\delta_{1k}|=\sup_{J_i}|f^{(k)}-\delta_{1k}|,\quad
|(f^{-1})^{(k)}(y_i)-\delta_{1k}|=\sup_{J_i}|(f^{-1})^{(k)}-\delta_{1k}|.\]
and 
$f^{(k)}(z_i)=\Id^{(k)}(z_i)=\delta_{1k}$.
Again, it suffices to prove the following two claims.
\setcounter{claim}{3}
\begin{claim}
$d_\bN(A_\delta)=0$.
\end{claim}

By Lemmas~\ref{lem:exp-N} and~\ref{lem:est-ck}, we have that
\[\{(1/i)^{k-1}\omega(1/i)\co i\in A_\delta\}
\precsim\{  1/{N_i} \co i\in A_\delta\}
\precsim\{
\sup_{J_i}|f^{(k)}-\delta_{1k}|\cdot |J_i|^{k-1}
\co i\in A_\delta\}.\]
By Claim~\ref{cla:disj}, we see
\[
\sum_i |f^{(k)}(x_i)-\delta_{1k}|
=\sum_i |f^{(k)}(x_i)-f^{(k)}(z_i)|
\le K\operatorname{Var}(f^{(k)},I)<\infty.\]
So, for some constant $K_0,K_1>0$
we deduce from H\"older's inequality that
\begin{align*}
\sum_{i\in A_\delta} \frac{1}{i}
&\le
\sum_{i\in A_\delta} 
\left(\frac{K_0\omega(1/i)}{i^{k-1}}\right)^{1/k}
\le
K_1\sum_{i\in\bN} |J_i|^{({k-1})/k} 
\cdot|f^{(k)}(x_i)-\delta_{1k}|^{1/{k}}\\
&\le
K_1\left(\sum_{i\in\bN} |J_i|\right)^{({k-1})/k}
\left(\sum_{i\in\bN} |f^{(k)}(x_i)-\delta_{1k}|\right)^{1/{k}}<\infty.
\end{align*}
We conclude from Lemma~\ref{lem:Salat} that $d_\bN(A_\delta)=0$.

\begin{claim}
$d_\bN(B_\lambda)=0$.
\end{claim}

We apply Lemma~\ref{lem:est-ck}
and also the proof of Claim~\ref{cla:bdelta}.
For each $i\in\bN$ we put
\[
M_i = |f^{(k)}(x_i)-\delta_{1k}|+|(f^{(-1)})^{(k)}(y_i)-\delta_{1k}|.\]
We have
\[
\{ (1/ i)^{k-1}\omega(1/i) \co i\in B_\lambda\}
\precsim
\{
 (\lambda+1)^{1/N_i}-1 \co i\in B_\lambda\}
\precsim
\{M_i\cdot |J_i|^{k-1}\co i\in B_\lambda\}.\]
By Proposition~\ref{prop:group}, we have
\[
\sum_i |(f^{-1})^{(k)}(y_i)-\delta_{1k}|
\le K\operatorname{Var}((f^{-1})^{(k)},I)<\infty.\]
We again apply H\"older's inequality.
For some constant $K_0,K_1>0$, we see
\begin{align*}
\sum_{i\in B_\lambda} \frac{1}{i}
&\le
\sum_{i\in B_\lambda} 
\left(\frac{K_0\omega(1/i)}{i^{k-1}}\right)^{1/k}
\le
K_1\sum_{i\in\bN} |J_i|^{({k-1})/k} 
M_i^{1/{k}}\\
&\le
K_1\left(\sum_{i\in\bN} |J_i|\right)^{({k-1})/k}
\left(\sum_{i\in\bN} M_i\right)^{1/{k}}<\infty.
\end{align*}
We obtain $d_\bN(B_\lambda)=0$.
\ep

\subsection{Diffeomorphisms of optimal regularity}\label{ss:optimal}
Let us now describe a method of constructing a fast diffeomorphism of a specified regularity on a given support.
\begin{thm}\label{thm:optimal}
We let $k\in\bN$, let $\delta\in(0,1)$ and let $\mu$ be a concave modulus satisfying $\mu\gg\omega_1$.
Suppose that $\{J_i\}_{i\in\bN}$ is a disjoint collection of compact intervals such that $J_i\sse I\setminus\partial I$,
and that $\{N_i\}_{i\in\bN}\sse\bN$ is a sequence such that 
\[\inf_{i\in\bN}N_i\cdot  |J_i|^{k-1}\mu\left(|J_i|\right)\ge1.\]
Then there exists  $f\in\Diff^{k,\mu}_+(\bR)$ satisfying the following:
\be[(i)]
\item\label{p:supp}
$\displaystyle \supp f=\{x\in\bR\mid f(x)>x\}= \cup_i  (J_i\setminus\partial J_i)$;
\item\label{p:fast} $f^{N_i}$ is $\delta$--fast on $J_i$ for all $i$;
\item\label{p:inf} if an open neighborhood $U$ of $x\in \bR$ intersects only finitely many $J_i$'s, then  $f$ is $C^\infty$ at $x$.
\ee
\end{thm}

Since $I$ is compact, it is necessary that $\sum_i|J_i|<\infty$.
From the above theorem we will deduce that some $C^{k,\mu}$ diffeomorphism is ``faster" than all $C^{k,\omega}$ diffeomorphisms for $\omega\ll\mu$,
in a precise sense as described in Corollary~\ref{cor:optimal}.

Throughout Section~\ref{ss:optimal}, we will fix the following constants.
\begin{setting}\label{setting:kmu}
Let $k,\delta,\mu$ be as in Theorem~\ref{thm:optimal}.
Pick a constant $\epsilon_0\in(0,1)$ and put
\[C= 1/({1+8\epsilon_0}),
\quad D  =(1-C)/2,\quad \delta_0=(1-\epsilon_0)C.\]
A priori, we will choose $\epsilon_0$ so small that we have estimates
\[D\le 1/10,\quad \delta_0 \ge \max(\delta,9/10).\]
We also pick $\ell_0^*\in(0,\epsilon_0]$ such that $\mu(\ell_0^*)\le\epsilon_0$.
\end{setting}

We will prove Theorem~\ref{thm:optimal} through a series of lemmas. Let us first note
the following standard construction of a bump function $\Psi$; see Figure~\ref{fig:bump} (a).

\begin{lem}\label{lem:bump}
There exists an even, $C^\infty$ map $\Psi\co \bR\to\bR$ such that the following hold:
\begin{itemize}
\item
$\Psi(t)=0$ if $t\le-1$ or $t\ge1$;
\item
$\Psi(0)=1$;
\item
$\Psi'(t)>0$ if $t\in(-1,0)$;
\item
$\int_\bR \Psi=1$.
\end{itemize}
\end{lem}

For $U\sse M$ and for $m\in\bN\cup\{0\}$,
 the \emph{$C^m$--norm} of $f\co U\to \bR$  is defined as 
\[\|f\|_{m,\infty} = \sup_{0\le j\le m}\|f^{(j)}\|_\infty
=\sup\{| f^{(j)}(x) | \co x\in U\text{ and }0\le j\le k\}.\]
Let us introduce a constant 
\[K_0=C\left(\frac{2}{ D}\right)^{k+1}\|\Psi\|_{k,\infty}.\]
The following technical lemma establishes the existence of a bump function with a long flat interval and with a controlled $C^{k}$--norm. See Figure~\ref{fig:bump} (b).

\begin{lem}\label{lem:bump2}
For each $\ell\in(0,\ell_0^*]$, 
there exists a $C^\infty$ map $g\co \bR\to\bR$ such that
\be[(i)]
\item
$g(t)\begin{cases}
=0 &\text{if }\ t\le0\text{ or }t\ge\ell,\\
\textrm{is strictly increasing}&
\text{if }\ 0<t< D\ell,\\
=C\ell^{k}\mu(\ell) &\text{if }\  D\ell\le t\le(1- D)\ell,\\
\text{is strictly decreasing}&
\text{if }\ (1- D)\ell<t<\ell.\\
\end{cases}$
\item\label{p:primehalf} $|g'(t)|\le 1/2$ for all $t\in\bR$.
\item\label{part:gk}
$\|g\|_{k,\infty}\le K_0\mu(\ell)$.
\item\label{part:xytau}
$|g^{(k)}(x)-g^{(k)}(y)|\le K_0\mu(|x-y|)$ for all $x,y\in\bR$.
\ee\end{lem}

\begin{figure}[b!]
  \tikzstyle {a}=[postaction=decorate,decoration={%
    markings,%
    mark=at position .65 with {\arrow{stealth};}}]
\subfloat[(a) $y=\Psi(x)$]
{\includegraphics[width=.3\textwidth]{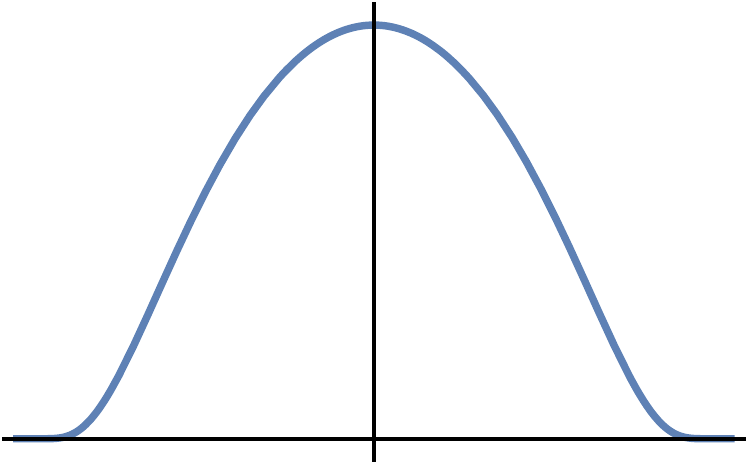}}
$\qquad\qquad$
\subfloat[(b) $y=g(x)$]
{\includegraphics[width=.5\textwidth]{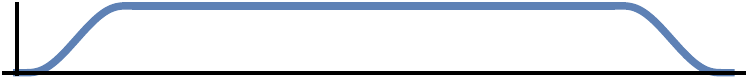}}
\caption{Scaled bump functions.}
\label{fig:bump}
\end{figure}

\bp
There exists a unique  $C^\infty$ map $g$ satisfying the following conditions:
\[g(t)=\begin{cases}
0 &\text{if }t\le0\text{ or }t\ge\ell,\\
C\ell^{k}\mu(\ell)\int_{-\infty}^{2t/( D\ell)-1}\Psi &
\text{if }t\le\ell/2,\\
C\ell^{k}\mu(\ell) &\text{if } D\ell\le t\le(1- D)\ell,\\
C\ell^{k}\mu(\ell)\int_{-\infty}^{2(\ell-t)/( D\ell)-1}\Psi &
\text{if }t\ge\ell/2.
\end{cases}\]
Hence, we have (i).

If $t\in(0,\ell/2)$, then
\[g'(t)=C\ell^{k}\mu(\ell)\left(\frac{2}{ D\ell}\right)\Psi\left(\frac{2t}{ D\ell}-1\right)
\le\frac{2C}{ D}\ell^{k-1}\mu(\ell)
\le\frac{2C}{ D}\epsilon_0^k
=\frac12\epsilon_0^{k-1}
\le
1/2.
\]
It follows that $g'(t)\in[0,1/2]$.
Since we have the symmetry  $g(t)=g(\ell-t)$, we obtain (\ref{p:primehalf}).
We see
$\|g\|_\infty=C\ell^{k}\mu(\ell)\le C\mu(\ell)\le K_0\mu(\ell)$.
If $t\le\ell/2$ and $i\ge1$, then
\[
\|g^{(i)}\|_\infty
\le
C\ell^{k}\mu(\ell)\left(\frac{2}{ D\ell}\right)^i 
\|\Psi^{(i-1)}\|_\infty
\le C\ell^{k}\mu(\ell)\left(\frac{2}{ D\ell}\right)^k\|\Psi\|_{k,\infty}
\le K_0\mu(\ell).
\]
The condition (\ref{part:gk}) follows.

To verify (\ref{part:xytau}), let us estimate $|g^{(k)}(x)- g^{(k)}(y)|$.
We have that $g^{(k)}=0$ on 
\[(-\infty,0)\cup(D\ell,(1-D)\ell)\cup(\ell,\infty).\]
Using the symmetry $|g^{(k)}(x)|=|g^{(k)}(\ell-x)|$,
we may only consider $x\in[0, D\ell]$. 
Note that $\Psi^{(k-1)}(-1)=\Psi^{(k-1)}(1)=0$.
Since $\Psi^{(k-1)}$ is Lipschitz, we have that
\[\left|\Psi^{(k-1)}\left(\frac{2x}{ D\ell}-1\right)\right|
\le
 \|\Psi^{(k)}\|_\infty 
\min\left(
 \frac{2x}{ D\ell},
 \frac{2D\ell-2x}{ D\ell} 
\right).\]
So, we have an inequality
\begin{align*}
|g^{(k)}(x)|
=
C\mu(\ell)\left(\frac{2}{D}\right)^k
\left|\Psi^{(k-1)}\left(\frac{2x}{ D\ell}-1\right)\right|
\le
\frac{K_0\mu(\ell)}{\ell}\min(x,D\ell-x).
\end{align*}
We now have the following three possibilities for $y$.

\emph{Case 1. $y\in (-\infty,0]\cup [D\ell,(1- D)\ell]\cup[\ell,\infty)$.}

Since we have $\min(x,D\ell-x)<\ell$, we see
\[
|g^{(k)}(x)-g^{(k)}(y)|
\le
K_0(\mu(\ell)/\ell) \min(x,D\ell-x)
\le
K_0 \mu(\min(x,D\ell-x))
\le 
K_0 \mu(|x-y|).\]

\emph{Case 2. $y\in[0, D\ell]$.}

We see that
\[
|g^{(k)}(x)-g^{(k)}(y)|
\le
C\mu(\ell)\left(\frac{2}{ D}\right)^k 
\left(\frac{2}{ D\ell}\right)\|\Psi\|_{k,\infty}\cdot |x-y|
\le{K_0} \mu(|x- y|).\]

\emph{Case 3.
 $y\in((1- D)\ell,\ell)$.}

Since $D\le 1/10$,
we have
$x+\ell-y\le 2D\ell\le \ell-2D\ell\le y-x<\ell$.
We see that
\[
|g^{(k)}(x)-g^{(k)}(y)|
\le
|g^{(k)}(x)|+|g^{(k)}(\ell-y)|
\le
K_0(\mu(\ell)/\ell)(x+\ell-y)
\le K_0\mu(|x-y|).\qedhere
\]
\ep

\begin{lem}\label{lem:construction}
For each compact interval $J\sse \bR$ with $0<\ell:=|J|\le \ell_0^*$,
there exists a diffeomorphism $f\in\Diff^\infty_+(\bR)$ 
satisfying the following:
\begin{enumerate}[(A)]
\item\label{part:f-supp} $\supp f= J\setminus\partial J$;
\item\label{part:fprime} $\inf_\bR f'(x)\ge 1/2$;
\item \label{part:f-ck}
$\|f-\Id\|_{k,\infty}\le K_0 \mu(\ell)$;
\item\label{part:f-Nx} for each $N\ge 1/({\ell^{k-1}\mu(\ell)})$, we have 
\[
\sup_J|f^N-\Id|
\geq \delta_0\ell.\]
\item\label{part:f-kx}  $|f^{(k)}(x)-f^{(k)}(y)|\le K_0\mu(|x-y|)$ for all $x,y\in\bR$.
\end{enumerate}
\end{lem}
\bp
We may assume $J=[0,\ell]$.
Let $g$ be as in Lemma~\ref{lem:bump2}, and put $f=\Id+g$.
By symmetry and the condition (\ref{p:primehalf}) on $g$, we have
\[
f'(t) = 1+g'(t)\ge 1/2\] for all $t$.
We have (\ref{part:fprime}), and in particular, $f$ is a $C^\infty$ diffeomorphism.

The claims (\ref{part:f-supp}), (\ref{part:f-ck}) and (\ref{part:f-kx}) are immediate from Lemma~\ref{lem:bump2}.
Observe that
  \[\frac{(\ell- D\ell)- D\ell}{C\ell^{k}\mu(\ell)}=\frac{1-2 D}{C\ell^{k-1}\mu(\ell)}=\frac{1}{\ell^{k-1}\mu(\ell)}.\] 
For each $N\geq 1/(\ell^{k-1}\mu(\ell))$,
we have that
\begin{align*}
|f^N( D\ell)- D\ell|
&=\sum_{i=0}^{N-1} |f^{i+1}( D\ell)-f^i( D\ell)|
\ge C\ell^{k}\mu(\ell)\left\lfloor \frac1{\ell^{k-1}\mu(\ell)}\right\rfloor
\\
&
=
C\ell(1 - \ell^{k-1}\mu(\ell))
\ge
C\ell(1 - \epsilon_0)
=
\delta_0\ell.
\end{align*}
This establishes the claim (\ref{part:f-Nx}), and hence the conclusion of the lemma.
\end{proof}

\bp[Proof of Theorem~\ref{thm:optimal}]
Put $\ell_i=|J_i|$.
As $\sum_i \ell_i<\infty$, there exists $i_0$ 
such that $\ell_i\le \ell^*$ for all $i\ge i_0$.
For each $i\ge i_0$, 
we apply Lemma~\ref{lem:construction}
to obtain $f_i\in\Diff^\infty_+(\bR)$ with:
\be[(A)]
\item\label{part:fiji}
$\supp f_i = J_i\setminus\partial J_i$;
\item\label{part:fiprime} $\inf_\bR |f_i'(x)|\ge 1/2$;
\item\label{part:fiid}
$\|f_i-\Id\|_{k,\infty}\le K_0\mu(\ell_i)$;
\item\label{part:fin}
$f_i^N$ is $\delta_0$--fast on $J_i$ 
for all $N\ge 1/(\ell_i^{k-1}\mu(\ell_i))$;
\item\label{part:fitau} $|f_i^{(k)}(x)-f_i^{(k)}(y)|\le K_0\mu(|x-y|)$ for all $x,y\in\bR$.
\ee

For each $n\ge i_0$, consider the composition
\[
F_n=\prod_{i=i_0}^n f_i.\]
For $m\ge n\ge i_0$, we have that
\[\|F_m-F_n\|_{k,\infty}\le \sup\{|f_i^{(j)}(x)-\Id^{(j)}(x)|\co i>n,x\in J_i,0\le j\le k\}\le K_0\sup_{i>n} \mu(\ell_i).\]
Hence $\{F_n\}$ uniformly converges to a $C^k$ map $F\co \bR\to\bR$ in the $C^k$--norm~\cite{Evans-GSM}.

Since $F$ is the composition of infinitely many homeomorphisms with disjoint supports, we see $F$ is also a homeomorphism. In particular, we see $\supp F = \cup_{i\ge i_0} (J_i\setminus\partial J_i)$.
Moreover, $F'(x)=\lim_{n\to\infty} F_n'(x)\ge 1/2$ for all $x\in \bR$. It follows that $F$ is a $C^k$ diffeomorphism.

\begin{claim*}
For all $x,y\in\bR$ we have
\[ |F^{(k)}(x)-F^{(k)}(y)| \le 2K_0\mu(|x-y|).\]
\end{claim*}
In order to prove the claim, we may assume $x\in J_i$ for some $i\ge i_0$.
If $y\in J_i$, then the condition~(\ref{part:fitau}) implies the claim.
If $y\not\in\supp F$, then we can find  $x_0\in\partial J_i$
such that $|x-y|\ge |x-x_0|$. So,
\[
|F^{(k)}(x)-F^{(k)}(y)|
=|f_i^{(k)}(x)-f_i^{(k)}(x_0)|
\le K_0 \mu(|x-x_0|)\le K_0\mu(|x-y|).\]
Finally, if $y\in J_j$ for some $i\ne j\ge i_0$, then 
we can find $x_0\in \partial J_i$ and $y_0\in \partial J_j$ such that
$|x-y| \ge |x-x_0|+|y-y_0|$.
As $\mu$ is increasing, we see that
\begin{align*} |F^{(k)}(x)-F^{(k)}(y)|&\le
|f_i^{(k)}(x)-f_i^{(k)}(x_0)|+|f_j^{(k)}(y)-f_j^{(k)}(y_0)|\\
&\le
K_0 \mu(|x-x_0|)+K_0\mu(|y-y_0|)
\le 2K_0\mu(|x-y|).
\end{align*}
Hence, the claim is proved. We have that $F\in\Diff_+^{k,\mu}(\bR)$.

Finally, we can pick $F^{*}\in\Diff_+^\infty(\bR)$ such that:
\begin{itemize}
\item
$\supp F^{*} = \{x\in \bR\mid F^{*}(x)>x\}=\bigcup\{J_i\setminus\partial J_i\mid 1\le i<i_0\}$;
\item
$F^{*}$ is $\delta_0$--fast on $J_i$ for $1\le i<i_0$.
\end{itemize}
Then the diffeomorphism $f= F\circ F^{*}\in\Diff_+^{k,\mu}(I)$ satisfies the conclusions (\ref{p:supp}) and (\ref{p:fast}).
To see the conclusion (\ref{p:inf}), observe from the hypothesis that either 
\begin{itemize}
\item
$x\in J_i\setminus\partial J_i$ for some $i$, or 
\item $f=\Id$ locally at $x$, or
\item $x\in \partial J_i$ for some $i$, and some open neighborhood $U$ of $x$ satisfies $U\cap J_j=\varnothing$ for all $j\ne i$.
\end{itemize}
In all cases, $f$ coincides with some $f_i$ locally at $x$, and hence, is locally $C^\infty$.
\ep

\begin{rem}
In the above proof, the modulus of continuity was used to guarantee a uniform convergence of partially defined diffeomorphisms. 
This idea can be found in the construction of  a \emph{Denjoy counterexample}, which is a $C^{1+\epsilon}$ diffeomorphism $f\co S^1\to S^1$ such that $f$ is not conjugate to a rotation and such that $f$ has an irrational rotation number. Denjoy's Theorem implies that there are no such $C^{1+\mathrm{bv}}$ examples~\cite{Denjoy1932,Navas2011}.
\end{rem}

We note the following consequence of Theorem~\ref{thm:optimal}.

\begin{cor}\label{cor:optimal}
Let $K^*>0$,
and let $\{J_i\}_{i\in\bN}$ be a collection of disjoint compact intervals contained in the interior of $I$ satisfying
\[|J_i|=\left({(i+K^*)\log^2(i+K^*)}\right)^{-1}.\]
Then for $k\in\bN$ and for a concave modulus $\mu\gg\omega_1$,
 there exists
\[
f\in \Diff^{k,\mu}_+(\bR)\setminus
\left(
\bigcup_{0\prec_k \omega\ll\mu}
\Diff^{k,\omega}_+(\bR)
\cup
\Diff_+^{k,\mathrm{bv}}(\bR)
\right)
\]
such that
 $\supp f= \cup_i (J_i\setminus\partial J_i)$.
\end{cor}

\bp
Let us write $\ell_i=|J_i|$ and 
\[
N_i=\left\lceil
1/(\ell_i^{k-1}\mu(\ell_i))
\right\rceil
=
\left\lceil
(i+K^*)^{k-1}\log^{2k-2}(i+K^*)
/\mu(\ell_i))
\right\rceil.
\]
We have $f\in\Diff_+^{k,\mu}(I)$
as given by Theorem~\ref{thm:optimal}
with respect to $\{J_i\}$ and some $\delta\in(0,1)$.
Let us pick $\omega$ such that $0\prec_k \omega\ll\mu$.

\begin{claim*}
$\displaystyle
\lim_{i\to\infty} N_i(1/i)^{k-1}\omega(1/i)=0$.
\end{claim*}

For all sufficiently large $i$, we have
\[
\mu(\ell_i)\ge\mu
(1/({4i\log^2 i}))
\ge
 \mu(1/i)/({4\log^2 i}).\]
So we see that 
\begin{align*}
&N_i(1/i)^{k-1}\omega(1/i)
\le
\frac{2(i+K^*)^{k-1}\log^{2k-2}(i+K^*)\omega(1/i)}
{i^{k-1}\mu(\ell_i)}
\\
&\le 
\frac{8(i+K^*)^{k-1}\log^{2k-2}(i+K^*)}
{ i^{k-1}\log^{2k-2}i}
\cdot
\frac{\log^{2k} i\cdot \omega(1/i)}
{\mu(1/i)}\to0.
\end{align*}

Note that $\partial J_i$ are accumulated fixed points of $f$. 
Since $f^{N_i}$ is $\delta$--fast on $J_i$ for all $i$, Theorem~\ref{thm:est} implies that
$f$ is not $C^{k,\omega}$.
For $C^{k,\mathrm{bv}}$, we simply set $\omega=\omega_1$
and apply Theorem~\ref{thm:est} again.
\ep

\subsection{More on natural density}
For $N\in\bN$, let us use the notation  \[ [N]^*:=\{0,1,\ldots,N-1\}. \] 
We will need the following properties of density--one sets.
\begin{lem}\label{lem:long}
\be
\item
If $A\sse\bN$ satisfies $d_\bN(A)=1$,
then for each $s\in\bN$ we have 
\[d_\bN\{i\in\bN\co i+[s]^*\sse A\}=1.\]
\item
Let $\beta_0\in\bN$,
and let $X,Y\sse\bN$.
Assume that $d_\bN\left(X\cup\left((Y-\beta)\cap\bN\right)\right)=1$ for each integer $\beta\ge \beta_0$. Then we have $d_\bN(X\cup Y)=1$.
\ee
\end{lem}

\bp[Proof of Lemma~\ref{lem:long}]
(1)
We can rewrite the given set as
\[
\{i\in\bN\co i+[s]^*\sse A\}=A\cap (A-1)\cap \cdots\cap (A-(s-1)).\]
The conclusion follows from the first two parts of Lemma~\ref{lem:Salat}.

(2)
Pick an arbitrary integer $N\ge \beta_0$. For each $\beta\in\bN$, define
\[
S_1^{N,\beta}=
\left\{
m\in\bN\mid
m+[N]^*\sse X\cup(Y-\beta)
\right\}.\]
Part (1) implies that $d_\bN(S_1^{N,\beta})=1$ for each $\beta\ge\beta_0$. 
So, we have a density--one set
\[
S_2^N=
\bigcap_{\beta=\beta_0}^{N-1} S_1^{N,\beta}.\]

Suppose $m\in S_2^N$. If $m\le i<j\le m+N-1$ and $i,j\not\in X\cup Y$,
then $m\not\in S_1^{N,j-i}$. In particular, $j-i\le\beta_0-1$.
We obtain that
\[
\#\{ i\in m+[N]^*\mid i\not\in X\cup Y\}\le \beta_0.\]

Hence, for each $s\in\bN$ and $t\in[N]^*$ we compute
\[
(N-\beta_0) \cdot \#\left( S_2^N\cap (t+N[s]^*)\right)
\le
\#\left\{ m\in t+[Ns-N+1]^* \mid
m\in X\cup Y\right\}.\]
By summing up the above for $t\in[N]^*$, we have
\[
(N-\beta_0) \#\left(S_2^N\cap[Ns]^*\right)
\le
N \#\left\{ m\in [Ns]^* \mid
m\in X\cup Y\right\}.\]
After dividing both sides by $N^2 s$ and sending $s\to\infty$, we see that
\[
1-\frac{\beta_0}N\le \liminf_{s\to\infty}\frac{\#\left((X\cup Y)\cap[Ns]^*\right)}{Ns}.\]
Since $N$ is arbitrary, we have $d_\bN(X\cup Y)=1$.
\ep

\section{Background from one--dimensional dynamics}\label{sec:background}
In this section, we gather the relevant facts regarding one--dimensional dynamics that we require in the sequel.

\subsection{Covering distance and covering length}\label{ss:cdcl}
Throughout Section~\ref{ss:cdcl}, 
we let $G$ be a group with a finite generating set $V$,
and let  $\psi\colon G\to\Homeo_+(I)$ be an action.
We develop some notions of complexity of an element in $\psi(G)$ which will be useful for our purposes.

We use the notation
\[\supp\psi:=\supp\psi(G)=\bigcup_{g\in G}\supp\psi(g)=\bigcup_{v\in V}\supp\psi(v).\] 
Note that $\supp\psi$ may have multiple components.  Define
\[\mathcal{V}:=\bigcup_{v\in V}\pi_0 \supp \psi(v).\]
Then $\mathcal{V}$ is an open cover of $\supp\psi$ consisting of intervals.

For a nonempty subset $A\sse I$,
we define its \emph{$\mathcal{V}$--covering length} as
\[
\cl_\mathcal{V}(A)=\inf\{\ell\in\bN\mid A\sse A_1\cup \cdots \cup A_\ell,\quad\text{each }A_i\text{ is in }\mathcal{V}\}.\]
Here, we use the convention $\inf\varnothing=\infty$.
We also let $\cl_\mathcal{V}(\varnothing)=0$.
We define 
the \emph{$\mathcal{V}$--covering distance} of $x,y\in I$ as
\[
\cd_\mathcal{V}(x,y)=
\begin{cases}
\cl_\mathcal{V}\left([\min\{x,y\},\max\{x,y\}]\right),&\text{if }x\ne y;\\
0.&\text{if }x=y.\end{cases}\]
That is to say, once a generating set for $G$ has been fixed, $\cd_\VV(x,y)$ is the least number of components of supports of generators of $G$ needed to traverse the interval from $x$ to $y$. 
Also, if $x$ and $y$ lie in different components of $\supp \psi(G)$, then the covering distance between them is necessarily infinite. 
We let $\cd_{\mathcal{V}}(x,x)=0$.

Both covering distance and covering length depend not just on $G$ and $\psi$ but also on a generating set $V$.
When the meaning is clear, we will often omit $\VV$, and write $\cl(A)$ and $\cd(x,y)$.
We will also write $gx:=\psi(g).x$ for $g\in G$ and $x\in I$.

Covering distance behaves well in the sense that it satisfies the triangle inequality:
\begin{lem}\label{lem:triangle}
For $x,y,z\in I$ and for $A,B\sse I$, the following hold.
\be
\item
$\cd(x,y)<\infty$ if and only if $x$ and $y$ are contained in the same component of $\supp\psi$.
\item $\cl(A\cup B)\le \cl(A)+\cl(B)$.
\item $\cd(x,y)\le \cd(x,z)+\cd(z,y)$.
\ee\end{lem}

\bp
Part (1) is clear. 
For part (2), assume 
\[\{U_1,\ldots,U_n\},\{V_1,\ldots,V_m\}\sse\VV\] are open covers of $A$ and $B$ 
which witness the fact that $\cl(A)=n$ and $\cd(B)=m$ respectively.
Then \[\{U_1,\ldots,U_n,V_1,\ldots,V_m\}\] cover the interval $A\cup B$.
Part (3) follows from part (2).
\ep

If $1\neq w\in G$, we define the \emph{syllable length} of $w$, written $||w||$, to be
 \[||w||=\min\{\ell \mid w=v_1^{n_1}v_2^{n_2}\cdots v_\ell^{n_\ell}\},\] where $v_i\in V$ and $n_i\in\bZ$ for each $1\leq i\leq \ell$. 
 The following lemma relates the algebraic structure of the given group $G=\form{V}$ with the dynamical behavior of actions of $G$:

\begin{lem}\label{lem:syll-cover}
For each $x\in I$ and $w\in G$, we have $\cd(x,wx)\leq ||w||$.
\end{lem}

Here we are implicitly measuring the covering distance with respect to the generating set $V$ of $G$.

\begin{proof}[Proof of Lemma~\ref{lem:syll-cover}]
Clearly we may assume that $x\in\supp\psi$, since otherwise there is nothing to prove. We proceed by induction on $||w||$. If $||w||=1$ then $w = v^n$ for some $v\in V$ and $n\in\bZ$.
Then either $x=v^nx$ or $x\in J\in\pi_0 \supp \psi(v)$. 
It follows that $\cd(x,v^nx)$ is $0$ or $1$.
Now assume $||w||=\ell\ge2$.
We can write $w=v^n\cdot w'$, where $||w'||=\ell-1$.
By induction, $\cd(x,w'x)\leq \ell-1$. 
As $\cd(w'x,v^n\cdot w'x)\leq 1$, the estimate follows from Lemma~\ref{lem:triangle}.
\end{proof}

Let  $(U_1,\ldots,U_n)$ be a sequence of open intervals in $\bR$ such that $U_i\cap U_j=\varnothing$ whenever $|i-j|\geq 2$, and such that $U_i\cap U_{i+1}$ is a nonempty proper subset of both $U_i$ and $U_{i+1}$ for $1\leq i\leq n-1$. 
Then we say $(U_1,\ldots,U_n)$ is a \emph{chain of intervals} in $\bR$.
Figure~\ref{f:coint} gives an example of a chain of four intervals.

A finite set $\FF$ of intervals is also called a chain of intervals if $\FF$ becomes so after a suitable reordering.
Chains of intervals arise naturally when we consider an open cover of a compact interval. 
The proof of the following lemma is straightforward.
\begin{lem}\label{lem:chain-cover}
If $\mathcal{U}$ is a collection of open intervals such that $I\sse \bigcup\mathcal{U}$,
then a minimal subcover $\mathcal{V}\sse\mathcal{U}$ of $I$ is a chain of finitely many open intervals.
\end{lem}

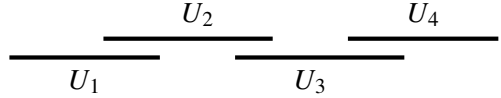
\begin{figure}[h!]
{\begin{tikzpicture}[ultra thick,scale=.5]
\draw (-5,0) -- (-1,0);
\draw (-3,0) node [below] {\small $U_1$}; 
\draw (-2.5,.5) -- (2,.5);
\draw (0,.5) node [above] {\small $U_2$}; 
\draw (1,0) -- (5.5,0);
\draw (3,0) node [below] {\small $U_3$}; 
\draw (4,.5) -- (8,.5);
\draw (6,.5) node [above] {\small $U_4$}; 
\end{tikzpicture}}
\caption{A chain of four intervals.}
\label{f:coint}
\end{figure}

When we discuss a chain of intervals, we assume those intervals are open.
It will be useful for us to be able to move points inside a connected component of $\supp \psi(G)$ efficiently in the following sense, which provides a converse to Lemma~\ref{lem:syll-cover}:

\begin{lem}\label{lem:slide}
Suppose $x<y\in U\in\pi_0 \supp \psi(G)$ satisfy $\cd(x,y)=N\in\bN$. Then there exists an element $g\in G$ such that $gx>y$ and such that $||g||= N$.
\end{lem}

We remark that ideas in a very similar spirit to Lemma~\ref{lem:slide} were used extensively in~\cite{KKL2017}.

\begin{proof}[Proof of Lemma~\ref{lem:slide}]
Let $\{U_1,\ldots,U_N\}$ be intervals such that $U_i\in\pi_0 \supp\psi(v_i)$ for some $v_i\in V$ for each $i$, and such that these intervals witness the fact that $\cd(x,y)=N$. Lemma~\ref{lem:chain-cover} implies $\{U_i\}$ is a chain.
Renumbering these intervals if necessary, we may assume that $x\in U_1\setminus U_2$, that $y\in U_N\setminus U_{N-1}$, and that \[\inf U_i<\inf U_{i+1}<\sup U_i<\sup U_{i+1}\] for each $i$  (cf. Figure~\ref{f:coint}).
Note that we allow $\sup U_{i-1}=\inf U_{i+1}$.

For a suitable choice of $n_1$, we have $v_1^{n_1}x=x_2\in U_2$. By induction, we have that $v_i^{n_i}x_i=x_{i+1}\in U_{i+1}$ for a suitable choice of $n_i$. Once $v_{N-1}^{n_{N-1}}\cdots v_1^{n_1}x=x_N\in U_N$, we apply a suitable power of $v_N$ to $x_N$ to get $v_N^{n_N}x_N>y$. Then \[g=v_N^{n_N}v_{N-1}^{n_{N-1}}\cdots v_1^{n_1}\] clearly has syllable length at most $N$ and satisfies $gx>y$. Lemma~\ref{lem:syll-cover} implies that $||g||= N$. 
\end{proof}

\subsection{A residual property of free products}\label{ss:res}
For a compact interval $J\sse\bR$,
we let $\Diff_0^\infty(J)$ denote  the group of  $C^{\infty}$--diffeomorphisms of $\bR$ supported in $J$. 
One can identify $\Diff_0^\infty(J)$ with the group of $C^\infty$--diffeomorphisms on $J$ which are $C^\infty$--tangent to the identity at $\partial J$.
For a group $G$ and a subset $S\sse G$, we let $\fform{S}$ denote the normal closure of $S$.

\begin{lem}\label{lem:res}
Suppose $G\le \Diff_0^\infty(I)$ has a connected support,
and suppose
\[
1\ne g\in (G\times\form{s})\ast\form{t}\cong (G\times\bZ)\ast\bZ.\]
Then there exists a representation
\[
\phi_g\co (G\times\form{s})\ast\form{t}\to\Diff_0^\infty(I)\] 
with a connected support such that 
$\phi_g(g)\ne1$
and
such that $\supp\phi_g(G)\cap\supp\phi_g(s)=\varnothing$.
Furthermore, we can require that $\phi_g(G)\cong G$.
\end{lem}

\bp[Proof of Lemma~\ref{lem:res}]
We have embeddings
\[
\rho_+\co G\to\Diff^\infty_0[0,1],\quad
\rho'_+\co \form{s}\to\Diff^\infty_0[0,1],\]
with full supports. 
Let $\rho_-$ and $\rho'_-$ denote the ``opposite''
representations of $\rho_+$ and $\rho'_+$, respectively.
That is,
we let
$\rho_-(g)(x) = 1 - \rho_+(g)(1-x)$
and similarly for $\rho_-'$.

After a suitable conjugation, we may assume
\[
g=t^{p_\ell} (g_\ell s^{q_\ell})\cdots t^{p_1}(g_1s^{q_1})\]
for some $\ell\in\bN,g_i\in G$
and $p_i,q_i\in\bZ$. 
For each $i$, we can further require that
$p_i\ne0$, and that either $g_i\ne1$ or $q_i\ne0$.
There exists a representation
\[
\rho_i\co G\times\form{s}\to\Diff_0^\infty[ {2i-1}, {2i}] \]
and a point $x_{2i-1}$
such that \[ {2i-1}< x_{2i-1} < x_{2i} :=\rho_i(g_is^{q_i})(x_{2i-1}) <  {2i} .\]
Here, $\rho_i$ is $C^\infty$--conjugate to $\rho_\pm$ if $g_i\ne1$, 
and to $\rho'_\pm$ otherwise.
In particular, 
we require $\supp\rho_i(G\times\form{s})=( {2i-1}, {2i})$.

We pick $x_{2\ell+1}$ and $z_i$ so that 
\begin{align*}
& 1<x_1<z_1<x_2< 2< 3<x_3<z_2<x_4< 4< 5<\cdots\\
&< {2\ell-1}<x_{2\ell-1}<z_\ell<x_{2\ell}< {2\ell}< {2\ell+1}<x_{2\ell+1}<z_{\ell+1}< {2\ell+2}.\end{align*}
We can find a $C^\infty$--action 
\[
\rho_0\co \form{t}\to \Diff_0^\infty[ 1, {2\ell+2}]\]
 such that  $\supp\rho_0( t)= \cup_{i=1}^{\ell} (z_i,z_{i+1})$
 and such that $\rho_0(t^{p_i})(x_{2i})=x_{2i+1}$.
We put
\[\phi_g:=\prod_{i=1}^\ell \rho_i \ast \rho_0\co G\ast\bZ\to\Diff_0^\infty[ 1, {2\ell+2}].\]
The nontriviality of $\phi_g(g)$ comes from a Ping--Pong argument for free products (cf.~\cite{Koberda2012,BKK2014}); that is, $\phi_g(g)(x_1)=x_{2\ell+1}>x_1$. 
The first conclusion follows from
\[
\supp\phi_g=\supp\rho_0\cup\left(\cup_i \supp\rho_i\right)=( 1,z_{\ell+1}).\]

We may assume $g_i\ne 1$ for at least one $i$.
This is because, the above construction also works for a finite set $A\sse G\setminus\{1\}$ after setting $g$ as a suitable concatenation of the elements in $A$. In particular, $\rho_i\restriction_G$ and $\phi_g\restriction_G$ are faithful. Here, the symbol $\restriction$ denotes the restriction of a representation.
\ep

\subsection{Centralizers of diffeomorphisms}\label{ss:cent}
We recall the following standard result. It was proved for $C^2$ maps by Kopell~\cite{Kopell1970} and generalized later to $C^{1+\mathrm{bv}}$ maps by Navas~\cite{Navas2010} in his thesis.

\begin{thm}[Kopell's Lemma; see~\cite{Kopell1970}]\label{thm:kopell}
Let $f,g\in\Diffb[0,1)$ be nontrivial commuting diffeomorphisms. 
If $\Fix f\cap (0,1)=\varnothing$, then $\Fix g\cap (0,1)=\varnothing$.
\end{thm}

We continue to let $M\in\{I,S^1\}$.
We say  $f\in\Homeo_+(M)$ is \emph{grounded} if $\Fix f\neq\varnothing$. In particular, every homeomorphism of $I$ is grounded. 
An important and relatively straightforward corollary of Kopell's Lemma is the following fact:

\begin{lem}[Disjointness Condition; see~\cite{BKK2016}]\label{lem:disjoint}
Let $f,g\in\Diffb(M)$ be commuting grounded diffeomorphisms, where $M\in\{I,S^1\}$, and let $U$ and $V$ be components of $\supp f$ and $\supp g$ respectively. Then either $U\cap V=\varnothing$ or $U=V$.
\end{lem}

If $\omega$ is a concave modulus or if $\omega\in\{0,\mathrm{bv},\mathrm{Lip}\}$, then
 we define the \emph{$C^{k,\omega}$--centralizer group} of $G\le\Homeo_+(M)$ as
\[
Z^{k,\omega}(G):=\{ h\in \Diff^{k,\omega}_+(M) \co [g,h]=1\text{ for all }g\in G\}.\]
Let $Z^{k,\omega}(g):=Z^{k,\omega}(\form{g})$ for $g\in\Homeo_+(M)$.
We write $\Fix G = \cap_{g\in G}\Fix g$.

Let $\BS(1,m)$ denote the Baumslag--Solitar group of type $(1,m)$, given as below.

\begin{lem}\label{lem:BMNR}
Suppose we have an integer $m>1$ and 
a representation \[\rho\co \BS(1,m)=\form{x,y\mid xyx^{-1}=y^m}\to \Diff_+^1(I).\]
\be
\item\label{p:faithful})
If $\rho(y)\ne1$, then $\rho$ is faithful.
\item\label{p:BMNR} (\cite{BMNR2017MZ})
We have that $\supp Z^1(\rho\form{x,y})\cap\supp \rho(y)=\varnothing$.
\ee
\end{lem}
\bp (1) 
Suppose $g\in\ker\rho\setminus\{1\}$. We may write $g = y^p x^q$ for some $p,q\in\bZ$
so that
\[
x g x^{-1} = (x y x^{-1})^p x^q=y^{2p}x^q\in\ker\rho.\]
It follows that $\rho(y^{p})=1$. Since $\rho(y)\ne1$, 
we see that $p=0$ and $\rho(x)=1$. But then, we have $\rho(y)=\rho(y^m)=1$. This is a contradiction.

(2)
We may assume $\rho$ is faithful by part (1). 
The case $m=2$ precisely coincides with \cite[Proposition~1.8]{BMNR2017MZ}.
The proof for the case $m>2$ is essentially identical.
\ep

If $g\in \Diff_+^{1+\mathrm{bv}}(S^1)$ is an infinite order element having a finite orbit, then every element in $Z^{1+\mathrm{bv}}(g)$ has a finite orbit
and every element in $[Z^{1+\mathrm{bv}}(g),Z^{1+\mathrm{bv}}(g)]$ is grounded;
see~\cite{FF2001} and~\cite{BKK2016}. This is a dynamical consequence of classical theorems of H\"older~\cite{Holder1996} and of Denjoy~\cite{Denjoy1932}, combined with Kopell's Lemma.
In this paper, we will need a $C^1$--analogue of this consequence, as described below.
The role of $\form{g}$ is now played by the group $\BS(1,2)$.

\begin{lem}\label{lem:c1-denjoy}
Suppose we have an isomorphic copy of $\BS(1,2)$ given as
\[B=\form{x,y \mid x y  x ^{-1}=y^2} \le\Diff_+^1(S^1).\]
Then the following hold.
\be
\item
The $C^1$--centralizer group $Z^1(B)$ of $B$ has a finite orbit.
\item\label{p:z0fi}
For some finite index subgroup $Z_0$ of $Z^1(B)$, we have 
 $\supp Z_0\cap \supp y=\varnothing$.
\item
We have  
$\supp [Z^1(B),Z^1(B)]\cap\supp y=\varnothing$.
\ee
\end{lem}

For $g\in\Homeo_+(S^1)$, we consider an arbitrary lift 
$\tilde g\co \bR\to\bR$ and define
the \emph{rotation number} of $g$ as
\[
\rot(g):=\lim_{n\to\infty}\frac{\tilde g^n(0)}n\in\bR/\bZ.\]

\bp[Proof of Lemma~\ref{lem:c1-denjoy}]
For some $m\in\bN$, the group $B_0=\form{ x ^m,y}\cong\BS(1,2^m)$ has a global fixed point;
this is due to~\cite[Theorem 1]{GL2011}.
We have a nonempty collection of open intervals:
\[\mathcal{A}=\{J\in\pi_0\supp B_0\co \text{ the restriction of }B_0\text{ on } J\text{ is nonabelian}\}.\]
We may regard  $ B_0\le \Diff_+^1[0,1]$.
It follows from~\cite[Theorem 1.7]{BMNR2017MZ} that $\mathcal{A}$ is a finite set. 
Since $Z^1(B) \le Z^1(B_0)$,
 the group $Z^1(B)$ permutes $\mathcal{A}$ and has a finite orbit inside $X = \bigcup_{J\in\mathcal{A}}\partial J\sse S^1$.
This proves part (1).

Let $Z_0$ be the kernel of the above homomorphism \[Z^1(B)\to\Homeo^+(X).\]
Since every element of $Z_0$ fixes $\partial J$ for $J\in\mathcal{A}$,  we can regard $\form{Z_0,B_0}\le\Diff_+^1[0,1]$. Lemma~\ref{lem:BMNR} implies part (2).

Part (3) is not essential for the content of this paper, but we include it here for completeness and for its independent interest. 
To see the proof, 
note first that the finite cyclic group action $\rho_0\co Z^1(B)/Z_0\to\Homeo^+(X)$ is free.
By a variation of H\"older's Theorem given in~\cite[Corollary 2.3]{KKFreeProd2017},
there exists a free action $\rho\co Z^1(B)/Z_0\to\Homeo^+(S^1)$  extending $\rho_0$
such that $\rot\circ\rho$ is a monomorphism;  see also~\cite{FF2001}.
We have a commutative diagram as below:
\[
\xymatrix{
&&& \Homeo^+(X)\\
1\ar[r] &
Z_0\ar[r] &
Z^1(B)\ar[r]^p\ar[ru]\ar[d]_{\rot} &
Z^1(B)/Z_0\ar[u]^{\rho_0}_{\mathrm{free}}\ar[r]\ar[d]_\rho^{\mathrm{free}} &
1 \\
& & S^1 & \Homeo^+(S^1)\ar[l]_>>>>>>{\rot}
}\]

Let $g\in[Z^1(B),Z^1(B)]$.
The commutativity of the lower square implies that $\rot$ restricts to a homomorphism on $Z^1(B)$.
In particular, we have that $\rot(g)=0$ and that $g$ is grounded.
Since $g$ centralizes $B$, and since $\Fix B_0\ne\varnothing$, 
we see that $\Fix\form{B_0,g}\ne\varnothing$.
So, we may regard $\form{B_0,g}\le\Diff_+^1(I)$. 
Lemma~\ref{lem:BMNR} implies that $\supp g\cap\supp y=\varnothing$, as desired.
\ep

\subsection{A universal compactly--supported diffeomorphism}\label{ss:univ}
Throughout this paper, we will fix a finite presentation:
\[G^\dagger=(\bZ\times\BS(1,2))\ast F_2=\left(\form{\bbc}\times\form{\bba,\bbe\mid \bba\bbe\bba^{-1}=\bbe^2}\right)\ast\form{\bbb,\bbd}.\] See Figure~\ref{f:gdagger}.
We let $V^\dagger=\{\bba,\bbb,\bbc,\bbd,\bbe\}\sse G^\dagger$.

\begin{figure}[h!]
\tikzstyle {Av}=[red,draw,shape=circle,fill=red,inner sep=1.5pt]
\tikzstyle {Bv}=[blue,draw,shape=circle,fill=blue,inner sep=1.5pt]
\tikzstyle {Cv}=[Maroon,draw,shape=circle,fill=Maroon,inner sep=1.5pt]
\tikzstyle {Dv}=[PineGreen,draw,shape=circle,fill=PineGreen,inner sep=1.5pt]
\tikzstyle {Ev}=[Orange,draw,shape=circle,fill=Orange,inner sep=1.5pt]
\tikzstyle {B}=[postaction=decorate,decoration={%
    markings,%
    mark=at position .55 with {\arrow{stealth};}}]
\begin{tikzpicture}[thick]
\draw [double,B] (0,-1) -- (2,-1);
\draw (0,-1) node [Av] {} node [left]  {\small $\bba$}  (2,-1) node [Ev] {} node [right] {\small $\bbe$} -- (1,0)  node [Cv] {} node  [above]  {\small $\bbc$} -- (0,-1);
\draw (0,0) node [Bv] {} node  [above]  {\small $\bbb$}; 
\draw (2,0) node [Dv] {} node  [above] {\small $\bbd$};
\end{tikzpicture}
\caption{The relators of $G^\dagger$.
The horizontal double edge denotes the relator $\bba\bbe\bba^{-1}=\bbe^2$
and the other two edges denote commutators.}
\label{f:gdagger}
\end{figure}
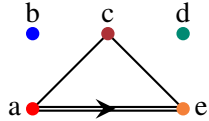

Whenever we have an action $\psi$ of $G^\dagger$ on $I$, 
we will define the covering length and the covering distance by the following open cover of $\supp\psi(G^\dagger)$: \[\VV=\bigcup_{v\in V^\dagger}\pi_0 \supp\psi(v).\]

If $\psi\colon G^\dagger\to\Homeo_+(I)$ is a representation and $f\in\psi(G^\dagger)$, there is little reason to believe that $\cl(\supp f)<\infty$, even if we restrict to a component of $\supp \psi(G^\dagger)$. In order to use the covering length of a support as a meaningful notion of complexity of a diffeomorphism, we need to find an element $1\neq u_0\in G^\dagger$ for which $\cl(\supp \psi(u_0))<\infty$.

We will build such an element $u_0\in  G^\dagger$.
We say a set $A\sse \bR$ is \emph{compactly contained} in a set $B\sse \bR$ if there exists a compact set $C$ such that $A\sse C\sse B$.

\begin{lem}[$abt$--lemma; {\cite[Theorem 3.1]{KKFreeProd2017}}]\label{lem:abt}
Let $M\in\{I,S^1\}$.
Suppose
 $\alpha,\beta,t\in\Diff^1_+(M)$ satisfy that
\[\supp \alpha\cap\supp \beta =\varnothing.\]
\be
\item
Then $\form{\alpha,\beta,t}$ is not isomorphic to $\bZ^2\ast\bZ$.
\item
If $M=I$, then the support of 
\[
u = [ [\alpha^t,\beta \cdot\beta^t \cdot\beta^{-1}],\alpha]\]
is compactly contained in $\supp\form{\alpha,\beta,t}$.
\ee
\end{lem}
 
\bp
Part (1) is stated as Theorem 3.1 of~\cite{KKFreeProd2017}. 
We summarize the proof of part (2), which is transparent from~\cite{KKFreeProd2017}.
We first consider $\gamma =\alpha^t$, $\delta =\beta^t$ and $\phi= [\gamma ,\beta \delta \beta^{-1}]$.
We have that $\supp \gamma \cap \supp \delta =\varnothing$.
By Lemma 3.10 of~\cite{KKFreeProd2017},
we have
$\supp\phi\setminus\supp \beta $ is compactly contained in $\supp \gamma \cup\supp \delta $.
Since $u=[\phi,\alpha]$, we see that
\begin{align*}
\overline{\supp u}&\sse\supp\phi\cup\supp \alpha\cup\overline{\supp\phi\cap\supp \alpha}\\
&\sse \supp\phi\cup\supp \alpha\cup\overline{\supp\phi\setminus\supp \beta }\\
&\sse \supp \alpha\cup\supp \beta \cup\supp \gamma \cup\supp \delta \sse\supp\form{\alpha,\beta ,t}.\qedhere
\end{align*}
\ep

We can now deduce Corollary~\ref{cor:c1-freeprod} in Section~\ref{sec:intro}.
The authors were told by A. Navas of the following proof for $M=I$.

\bp[Proof of Corollary~\ref{cor:c1-freeprod}]
Suppose we have a faithful representation
\[\psi\co(\form{\bbc}\times\form{\bba,\bbe})\ast\form{\bbd}\cong(\bZ\times\BS(1,2))\ast\bZ \to \Diff_+^1(M).\]
Consider first the case when $M=I$. 
By Lemma~\ref{lem:BMNR}, we see that 
$\supp\psi(\bbc)\cap\supp\psi(\bbe)=\varnothing$.
It follows from Lemma~\ref{lem:abt} that 
\[
\psi\form{\bbc,\bbe,\bbd}\not\cong\bZ^2\ast\bZ\cong \form{\bbc,\bbe,\bbd}.\]
This is a contradiction, for $\psi$ is faithful.

Assume $M=S^1$. By Lemma~\ref{lem:c1-denjoy} (\ref{p:z0fi}), we have some $p\in\bN$ such that 
\[\supp\psi(\bbc^p)\cap\supp\psi(\bbe)=\varnothing.\]
We again deduce a contradiction from Lemma~\ref{lem:abt}, for we have
\[
\psi\form{\bbc^p,\bbe,\bbd}\not\cong\bZ^2\ast\bZ\cong \form{\bbc^p,\bbe,\bbd}.\qedhere\]
\ep

We will apply $abt$--lemma to the triple $(\bbc,\bbe,\bbd)$. For this, we let
\begin{align*}
&\alpha=\bbc,\quad
\beta=\bbe,\quad \gamma=\alpha^\bbd=\bbd^{-1}\bbc \bbd,\quad\delta=\beta^\bbd=\bbd^{-1}\bbe \bbd,\\
&u^\dagger=[[\gamma,\beta \delta \beta^{-1}],\alpha]=\left[\left[\bbc^\bbd,\bbe\cdot\bbe^\bbd\cdot\bbe^{-1}\right],\bbc\right]\in G^\dagger\setminus\{1\}.
\end{align*}

\begin{lem}\label{lem:universal}
Let $u^\dagger\in \form{\bbc,\bbd,\bbe}\le G^\dagger$  be as above. Then for each representation
 \[\psi\colon \form{\bba,\bbc,\bbd,\bbe}\to \Diff^1_+(I),\] 
the set  $\supp \psi(u^\dagger)$ is compactly contained in $\supp\psi\form{\bbc,\bbd,\bbe}$. 
In particular, for each $U\in\pi_0 \supp\psi(u^\dagger)$ we have $\cd(\inf U,\sup U)<\infty$.
\end{lem}

\bp
Since $\psi(\bbc)\in Z^1(\form{\bba,\bbe})$, we see from Lemma~\ref{lem:BMNR} (\ref{p:BMNR}) that 
$\supp\psi(\bbc)\cap\supp\psi(\bbe)=\varnothing$.
Lemma~\ref{lem:abt} implies the desired conclusion.
\ep
\subsection{Simplicity and diffeomorphism groups}

We will require some classical results about the simplicity of certain groups of diffeomorphisms of manifolds.
For a manifold $X$, we let  $\Diff_c^{k,\omega}(X)_0$ denote the set of $C^{k,\omega}$ diffeomorphisms isotopic to the identity through compactly supported isotopies; this set is indeed a group~\cite{Mather1}.
Note that
\[
\Diff_c^{k,\omega}(S^1)_0=\Diff_+^{k,\omega}(S^1),\quad
\Diff_c^{k,\omega}(\bR)_0=\Diff_c^{k,\omega}(\bR).\]

\bd\label{defn:sup-sub}
Let $\omega$ be a concave modulus.
\be
\item
We say $\omega$ is \emph{sup-tame} if 
$\lim_{t\to+0}\sup_{0<x<\delta}
{t\omega(x)}/{\omega(tx)}=0$ for some $\delta>0$;
\item
We say $\omega$ is \emph{sub-tame} if 
$
\lim_{t\to+0}\sup_{0<x<\delta}{\omega(tx)}/{\omega(x)}=0$ for some $\delta>0$.
\ee
\ed

Mather~\cite{Mather1,Mather2} proved the simplicity of $\Diff_+^k(X)$, where $X$ is an $n$--manifold and $k\ne n+1$. The following is a straightforward generalization from his argument.
\begin{thm}[Mather's Theorem~\cite{Mather1,Mather2}]\label{thm:mather}
Suppose $X$ is a smooth $n$--manifold without boundary.
Let $k\in\bN$,
and let $\omega$ be a concave modulus
satisfying the following:
\begin{itemize}
\item
if $k=n$, then we further assume $\omega$ is sup-tame;
\item
if $k=n+1$, then we further assume $\omega$ is sub-tame.
\end{itemize}
Then the group $\Diff_c^{k,\omega}(X)_0$ is simple.
\end{thm}

In Example~\ref{ex:mod}, we have defined a concave modulus $\omega_z$ for each $z\in(0,1]_\bC$.
\begin{lem}\label{lem:sup-sub}
We have the following.
\be
\item
The concave modulus $\omega_{s\sqrt{-1}}$ is sup-tame for  $s>0$;
\item
The concave modulus $\omega_{1+s\sqrt{-1}}$ is sub-tame for $s\le0$;
\item
The concave modulus $\omega_{\tau+s\sqrt{-1}}$ is sup-and sub-tame for $\tau\in(0,1)$ and $s\in\bR$.
\ee\end{lem}
\bp
Let $t,x>0$. We  substitute $T = \log (1/t)$ and $X=\log (1/x)$.

(1) Put $\omega=\omega_{s\sqrt{-1}}$ for some $s>0$. 
There exists some $c\in(X,X+T)$ such that
\begin{align*}
\frac{t\omega(x)}{\omega(tx)}
&=t\exp\left(-s\frac{\log (1/x)}{\log\log(1/x)}
+s\frac{\log (1/tx)}{\log\log(1/tx)}\right)\\
&= \exp\left(-T-s\frac{X}{\log X}
+s\frac{T+X}{\log(T+X)}\right)
=\exp\left(-T+sT\frac{\log c - 1}{\log^2 c}\right).
\end{align*}
Pick a sufficiently small $\delta>0$ such that $K:=\log(1/\delta)$
satisfies $K>1/e^2$ and $s(\log K-1)/\log^2 K<1/2$.
Since $c>X\ge K$, we have that
\[t\omega(x)/\omega(tx)
\le
\exp(-T+sT(\log K-1)/\log^2K)\le \exp(-T/2).\]
It follows that 
 $\sup_{0<x<\delta} t\omega(x)/\omega(tx)\to0$ as $t\to0$.

(2) Put $\omega=\omega_{1+s\sqrt{-1}}$ for some $s\le 0$. 
We again compute
\[
\frac{\omega(tx)}{\omega(x)}
= \exp\left(-T
-s\frac{T+X}{\log(T+X)}
+
s\frac{X}{\log X}
\right).
\]
We then proceed exactly as in (1).

(3) 
Put $\omega=\omega_{\tau+s\sqrt{-1}}$. We define
\[
\mu(x)=x^{-\tau/2}\omega(x)=\omega_{\tau/2+s\sqrt{-1}},
\quad
\nu(x)=x^{(1-\tau)/2}\omega(x)=\omega_{(1+\tau)/2+s\sqrt{-1}}\]
for all small $x>0$.
We see from Lemma~\ref{lem:omega} (1) that
\begin{align*}
{\omega(tx)}/{\omega(x)}&=t^{\tau/2}\cdot{\mu(tx)}/{\mu(x)}\le t^{\tau/2}\to0.\\
{t\omega(x)}/\omega(tx)&=t^{1+(1-\tau)/2}\cdot{\nu(x)}/{\nu(tx)}\le t^{(1-\tau)/2}\to0.\qedhere
\end{align*}
\ep
\begin{cor}\label{cor:mather}
Let $X$ be a smooth $n$--manifold without boundary,
and let $k\in\bN$.
If some $z\in(0,1]_\bC$ satisfies $\operatorname{Re}(k+z)\ne n+1$, 
then  the group $\Diff_c^{k,\omega_z}(X)_0$ is simple.
\end{cor}
\bp
We use Lemma~\ref{lem:sup-sub} and Mather's Theorem.
If $\Re z\in(0,1)$, then $\omega_z$ is sup-and sub-tame, and so, $\Diff_c^{k,\omega_z}(X)_0$ for all $k\in\bN$.
If $z=s\sqrt{-1}$ for some $s<0$, then $\omega_z$ is sup-tame; 
in this case, $\Diff_c^{k,\omega_z}(X)_0$ is simple for all integer $k\ne n+1$.
If $z=1+s\sqrt{-1}$ for some $s\ge0$, then $\omega_z$ is sub-tame and $\Diff_c^{k,\omega_z}(X)_0$ for all integer $k\ne n$.
The conclusion follows.
\ep

We will later use the following form of simplicity results. The proof is given in Appendix (Theorem~\ref{thm:continuous2}).

\begin{thm}\label{thm:continuous}
For each $X\in\{S^1,\bR\}$, the following hold.
\be
\item
If $\alpha\ge1$ is a real number, then 
every proper quotient 
of $\Diff_c^\alpha(X)_0$
is abelian.
If, furthermore, $\alpha\ne2$, then $\Diff_c^\alpha(X)_0$ is simple.
\item
If $\alpha>1$ is a real number, then every proper quotient 
of $\bigcap_{\beta<\alpha}\Diff_c^\beta(X)_0$
is abelian.
If, furthermore, $\alpha>3$, then $\bigcap_{\beta<\alpha}\Diff_c^\beta(X)_0$ is simple.
\ee
\end{thm}
\subsection{Locally dense copies of Thompson's group $F$}\label{ss:dense-F}

Recall that Thompson's group $F$ is defined to be the group of piecewise linear homeomorphisms of the unit interval $[0,1]$ such that the discontinuities of the first derivatives lie at dyadic rational points, and so that all first derivatives are powers of two. It is well--known that Thompson's group $F$ is generated by two elements (see~\cite{CFP1996,BurilloBook}). 

We will denote the standard piecewise linear representation of $F$ as
\[
\rho_F\co F\to \Homeo_+[0,1].\]
A typical choice of a generating set for $F$ is $\{x_0,x_1\}$, 
which are determined by the breakpoints data:

\begin{align*}
&\rho_F(x_0).(0,1/4,1/2,1)=(0,1/2,3/4,1),\\
&\rho_F(x_1).(0,1/2,5/8,3/4,1)=(0,1/2,3/4,7/8,1).
\end{align*}

Recall that a group action on a topological space is \emph{minimal} if every orbit is dense.
The action $\rho_F$ is minimal on $(0,1)$, but it has an even stronger property: the diagonal action of $\rho_F$ on 
\[X=\{(x,y)\in (0,1)\times (0,1)\mid x<y\}\]
is minimal. This follows from the transitivity of $F$ on a pair of dyadic rationals in $X$; see~\cite{CFP1996} and~\cite{BurilloBook}. 

Alternatively, the action $\rho_F$ on $(0,1)$ is \emph{locally dense}~\cite{Brin2004GD}. The general definition of local density is not important for our purposes. For a \emph{chain group} $G\le\Homeo^+[0,1]$ (see Remark~\ref{rem:chain} below for a definition), the local density of the action of $G$ on $(0,1)$ is equivalent to the minimality of the action of $G$ on $X$, which in turn is equivalent to the minimality of the action of $G$ on $(0,1)$; this is proved in~\cite[Lemma 6.3]{KKL2017}.
Thompson's group $F$ is an example of a chain group (Corollary~\ref{cor:chain-smooth}).

We will require the following result:

\begin{thm}[Ghys--Sergiescu,~\cite{GS1987}]\label{thm:ghys-sergiescu}
The standard piecewise--linear realization $\rho_F$ of Thompson's group $F$ is topologically conjugate to a $C^{\infty}$ action on $[0,1]$ such that each element is $C^\infty$ tangent to the identity at $\{0,1\}$.
\end{thm}

The original construction of Ghys--Sergiescu is a $C^\infty$ action of Thompson's group $T$ for a circle; the above theorem is an easy consequence by restricting their action on an interval. Let us denote this action as
\[\rho_{\mathrm{GS}}\co F\to \Diff_0^\infty[0,1].\]
Note $\rho_{\mathrm{GS}}(F)$ acts minimally on $(0,1)$.
There exists a homeomorphism $h_{\mathrm{GS}}\co [0,1]\to [0,1]$ such that for all $g\in F$ we have
\[ \rho_{\mathrm{GS}}(g) = h_{\mathrm{GS}}\circ\rho_F(g)\circ h_{\mathrm{GS}}^{-1}.\]
It will be convenient for us to denote $a_i = \rho_{\mathrm{GS}}(x_i)$ for $i=0,1$.

\begin{cor}\label{cor:chain-smooth}
There exists a chain of two intervals $(U_1,U_2)$ and $C^{\infty}$ diffeomorphisms $f_1$ and $f_2$ supported on $U_1$ and $U_2$ respectively
 such that $\form{f_1,f_2}=\rho_{\mathrm{GS}}(F)$.
\end{cor}

\bp It is routine to check that $f_1=a_1^{-1}a_0$ and $f_2=a_1$ satisfy the conclusion.
See~\cite{KKL2017} for details. \ep

\begin{rem}\label{rem:chain}
More generally, if $(U_1,\ldots,U_n)$ is a chain of intervals and if  $f_1,\ldots,f_n\in \Homeo_+(\bR)$ satisfy that $\supp f_i=U_i$ for each $i$, then the group $\langle f_1,\ldots,f_n\rangle$ is called a \emph{pre--chain group} (cf.~\cite{KKL2017}). The group $\langle f_1,\ldots,f_n\rangle$ is called a \emph{chain group} if moreover we have $\langle f_i,f_{i+1}\rangle\cong F$ for each $1\le i<n$. If $\langle f_1,\ldots,f_n\rangle$ is a pre--chain group then for all sufficiently large $N$, we have $\langle f_1^N,\ldots,f_n^N\rangle$ is a chain group~\cite{KKL2017}.
\end{rem}

\section{The Slow Progress Lemma}\label{sec:psi}
Throughout this section, we assume the following.
Let $k\in\bN$, 
and
 let $G$ be a group with a finite generating set $V$.
We will consider an arbitrary representation $\psi$ of $G$ given in one of the following two types:
\begin{itemize}
\item
$\psi\co G\to\Diff_+^{k,\omega}(I)$,
where 
$\omega\succ_k 0$ is some concave modulus;
\item
$\psi\co G\to\Diff_+^{k,\mathrm{bv}}(I)$,
in which case we will put $\omega=\omega_1$.
\end{itemize}

We denote by $\|h\|$ the syllable length of $h\in G$ with respect to $V$ as in Section~\ref{ss:cdcl}.
We also use the notation $\mathcal{V} = \cup_{v\in V}\pi_0 \supp\psi(v)$.

Suppose we have sequences $\{N_i\}_{i\in\bN}\sse\bN$ and $\{v_i\}_{i\in\bN}\sse V$ such that 
the following two conditions hold.
First, for some $K>0$ we assume
\begin{equation}\label{A1}\tag{A1}
\sup_{i\in\bN}N_i(1/i)^{k-1}\omega(1/i)\le K.\end{equation}
Second, for each $v\in V$ we assume the following set has a well-defined natural density:
\begin{equation}\label{A2}\tag{A2}
\bN_v:=\{i\in\bN\mid v_i = v\}.\end{equation}
Let us define a sequence of words $\{w_i\}_{i\ge0}\sse G$ by $w_0=1$
and  \[w_i=v_i^{N_i}\cdot w_{i-1}.\]

The main content of this section is the following:

\begin{lem}[Slow Progress Lemma]\label{lem:stutter}
For each $x\in I$, we have the following:
\[\lim_{i\to\infty} \left(i - \cd_{\mathcal{V}}(x,\psi(w_i)x)\right)=\infty.\]
\end{lem}

The proof of the lemma occupies most of this section. 
As a consequence of this lemma, we will then describe a group theoretic obstruction for algebraic smoothing.
\begin{rem}
The statement of the Slow Progress Lemma is \emph{topological}.
In other words, even after $\psi$ is replaced by an arbitrary topologically conjugate representation, the same conclusion holds.
\end{rem}

\subsection{Reduction to limit superior}
For brevity, we simply write $\cl$ and $\cd$ for $\cl_{\mathcal{V}}$ and $\cd_{\mathcal{V}}$.
 We write
$gx=\psi(g)x$  for $g\in G$ and $x\in I$.

\begin{lem}\label{lem:limsup}
Let $x\in I$. Then the following are equivalent:
\be[(i)]
\item $\limsup_{i\to\infty} (i-\cd(x,w_i x))=\infty$;
\item $\lim_{i\to\infty} \left(i - \cd(x,w_i  x)\right)=\infty$.
\ee
\end{lem}
\bp
Assume (ii) does not hold.
There exists $M_0>0$ 
and an infinite set $A\sse\bN$ such that for all $a\in A$ we have
\[
a - \cd(x,w_a x)<M_0.\]
If (i) is true, then we have an increasing sequence $\{j(s)\}_{s\in\bN}$ such that 
\[\lim_{s\to\infty} (j(s)-\cd(x,w_{j(s)}x))=\infty.\]
For each $s\in\bN$, let us choose $a(s)\in A$ such that $j(s)<a(s)$.
We see that
\[
\cd(x,w_{a(s)}x)-\cd(x,w_{j(s)}x)
\le \cd(w_{j(s)}x,w_{a(s)}x)
\le a(s) - j(s),\]
\[j(s)-\cd(x,w_{j(s)}x)
\le a(s) - \cd(x,w_{a(s)}x) <M_0.\]
This is a contradiction, and (i)$\Rightarrow$(ii) is proved. The converse  is immediate.
\ep

\subsection{Markers of covering lengths}
In order to prove Lemma~\ref{lem:stutter} by contradiction, let us make the following standing assumption of this section: there exists a point $x\in U\in\pi_0 \supp\psi(G)$ and a real number $M_0>0$ such that the sequence $\{ x_i:=w_i x\}_{i\ge0}$ satisfies
 \begin{equation}\label{A3}\tag{A3}
\text{for all }i\ge0,\text{ we have }i-M_0\le \cl[x,x_i)\le i.\end{equation}
By Lemma~\ref{lem:limsup},
it suffices for us to deduce  a contradiction from \eqref{A3}.

The sequence $\{x_i\}$ accumulates at $\partial U$. 
Since the sequence cannot accumulate simultaneously at the both endpoints of $U$ by assumption \eqref{A3},
we may make an additional assumption:
\begin{equation}\label{A4}\tag{A4}
\lim_{i\to\infty} x_i=\sup U.\end{equation}

For each $i\in\bN$, we define
\[z^*_i=\sup\{z\in[x,\sup U)\mid \cl[x,z)\le i\}\}.\]
The point $z^*_i$ is the ``length--$i$ marker'' of covering lengths 
in the following sense.

\begin{lem}\label{lem:marker}
\be
\item
Define $h\co (x,\sup U)\to \bN$ by $h(z):=\cl[x,z)$.
Then $h$ is a surjective, monotone increasing, left-continuous function.
\item
For all $1\le i<i+j$, we have
\begin{align*}
 \cl[z^*_i,z^*_{i+j})&=j,\\
  \cl[z^*_i,z^*_{i+j}]&=j+1.\end{align*}
\item
There exists $M_1,M_2>0$ such that for all $i\ge M_1$ 
we have that \[z^*_{i-M_2}< x_{i}< z^*_{i-M_2+1}.\]
\ee\end{lem}
\bp
(1) Monotonicity of $h$ is clear. 
For the left--continuity and surjectivity, it suffices to show
$\cl[x,z^*_i)=i$.
Let us define
\[
z'_i =
\begin{cases}
\sup\{\sup J\mid x\in J\in\mathcal{V}\}&\text{ if }i=1,\\
 \sup\{\sup J\mid z'_{i-1}\in J\in\mathcal{V}\}&\text{ if }i\ge2.
\end{cases}\]
Since each point in $I$ belongs to at most $|V|$ intervals in $\mathcal{V}$, each $z_i'$ is realized as $\sup J$ for some $J\in\mathcal{V}$.

We claim that 
$z^*_i=z_i'$
and that
$\cl[x,z^*_i)=i$
for each $i\in\bN$. The case $i=1$ is trivial. Let us assume the claim for $i-1$. Then we have $\cl[x,z_i')=i$ and $z_i'\le z^*_i$.
If $z_i'<z^*_i$ then there exists $t\in(z_i',z^*_i)$ such that $\cl[x,t)=i$.
But whenever $t\in J\in\mathcal{V}$ we have $z'_{i-1}\not\in J$, by the choice of $z_i'$. This shows $\cl[x,t)>i$, a contradiction. Hence the claim is proved.

(2) Note that
\[
\cl[z^*_i,z^*_{i+j})\ge \cl[x,z^*_{i+j})-\cl[x,z^*_i)=j.\] The opposite inequality is immediate from the definition of $z_i'$. For the second equation, it suffices to further note that $\cl[x,z^*_i]=i+1$.

(3) By \eqref{A3}, the following holds for all but finitely many $i$:
\[\cl[x,x_{i+1})=\cl[x,x_i)+1.\]
For such an $i$, 
we have that $x_i\in(z^*_{j-1},z^*_{j}]$ 
and $x_{i+1}\in(z^*_{j},z^*_{j+1}]$ for $j=\cl[x,x_i)$.
If $x_i=z^*_{j}$, then $x_{i+1}<z^*_{j+1}$
and moreover, $x_{i+\ell}<z^*_{j+\ell}$ for all $\ell\in\bN$.
\ep
Let us write $z_i=z^*_{i-M_2}$. After increasing $M_0$ if necessary, 
we have the following for all $i\ge M_0$ and $j>0$:
\begin{equation}\label{A5}\tag{A5}
\cl[z_i,z_{i+j})=j=
\cl[z_i,z_{i+j}]-1\quad
\text{ and }\quad x_{i-1}<z_i<x_i.
\end{equation}
We may also assume:
\begin{equation}\label{A6}\tag{A6}
\cl[x,x_{M_0})>8k.
\end{equation}

Consider the set of ``significant generators'' and their minimum density:
\begin{align*}
V_1 &=\{v\in V\mid d_\bN(\bN_v)>0\},\\
\delta_1&=\min\{d_\bN(\bN_v)\mid v\in V_1\}/2.\end{align*}
By further increasing $M_0$, we may require:
\begin{equation}\label{A7}\tag{A7}
\#(\bN_v\cap[1,N])\ge \delta_1 N
\end{equation}
for all $v\in V_1$ and $N\ge M_0$.
We note the following.
\begin{lem}\label{lem:v-count}
Let $v\in V_1$,
and let $\bN_v=\{
j_1<j_2<j_3<\cdots\}$.
Then there exists a constant $K_1\ge K$ such that whenever $m\in\bN$ satisfies $j_m\ge M_0$, we have
\[
N_{j_m} \le K_1 m^{k-1}/\omega(1/m).\]
\end{lem}
\bp
Note that
\[m = \#(\bN_v\cap [1,j_m])\ge \delta_1 j_m.\]
Hence, we have $j_m\le m/\delta_1$.
Lemma~\ref{lem:omega} implies that
\[
\omega(1/j_m)\ge \omega(\delta_1/m)\ge \delta_1\omega(1/m).\]
The desired inequality is now immediate.
\ep

\subsection{Estimating gaps}
Let $i\ge M_0$. Since
\[x_{i-1} <z_i<x_i=v_i^{N_i}x_{i-1},\]
we can find $J_i\in\pi_0 \supp \psi(v_i)$ such that 
$\{x_{i-1},x_i\}\sse J_i$.
We define
\begin{align*}
p_i &= \inf\{
z\in(\inf U, \inf J_i] \mid \#\left([z,\inf J_i]\cap\Fix \psi(v_i)\right)\le k
\},\\
q_i &= \sup\{
z\in[\sup J_i,\sup U)\mid \#\left([\sup J_i,z]\cap\Fix \psi(v_i)\right)\le k
\}.\end{align*}
As illustrated in Figure~\ref{fig:jm}, we will write
\[
L_i = [p_i,\sup J_i],\quad
R_i = [\inf J_i, q_i],\quad
J_i^*=[p_i,q_i].\]
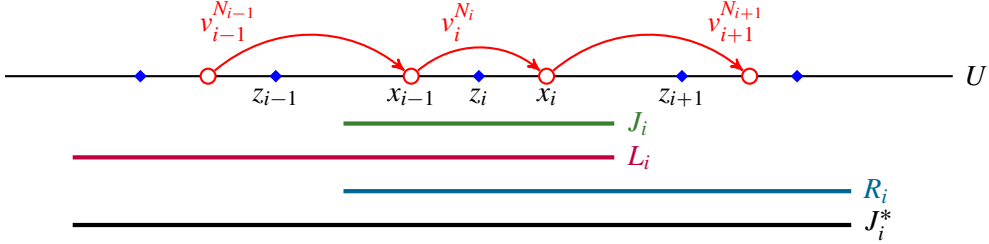
\begin{figure}[htb!] 
  \tikzstyle {bv}=[blue,draw,shape=diamond,fill=blue,inner sep=1pt]
  \tikzstyle {wv}=[red,draw,shape=circle,fill=white,inner sep=2pt]
\begin{tikzpicture}[>=stealth',auto,node distance=3cm, thick, scale=.9]
\draw  [thick] (-6,0) 
--
(0,0) node (x1) [wv] {} node [below]  {\small $x_{i-1}$}
--(1,0) node (z1) [bv] {} node [below]  {\small $z_{i}$}
-- (2,0)  node  (x2) [wv] {} node [below] {\small $x_i$}
-- (8,0) node [right] {\small $U$};
\draw (-4,0) node [bv] {} (-3,0) node  (x3) [wv] {}  (-2,0) node [bv] {}
 node [below]  {\small $z_{i-1}$}
(4,0) node [bv] {}
 node [below]  {\small $z_{i+1}$}
(5,0) node (x4)  [wv] {}  (5.7,0) node [bv] {};
\draw [ultra thick,OliveGreen] (-1,-.7) -- (3,-.7) node [right] {\small $J_i$};
\draw [ultra thick,purple] (-5,-1.2) -- (3,-1.2) node [right] {\small $L_i$};
\draw [ultra thick,MidnightBlue] (-1,-1.7) -- (6.5,-1.7) node [right] {\small $R_i$};
\draw [ultra thick] (-5,-2.2) -- (6.5,-2.2) node [right] {\small $J_i^*$};
\path (x1) edge [->,bend left,red,out = 40,in=140] node  {} (x2);
\path (x3) edge [->,bend left,red,out = 40,in=140] node  {} (x1);
\path (x2) edge [->,bend left,red,out = 40,in=140] node  {} (x4);
\draw (.7,.8) node [red] {\small $v_i^{N_i}$};
\draw (-2.7,.8) node [red] {\small $v_{i-1}^{N_{i-1}}$};
\draw (4.8,.8) node [red] {\small $v_{i+1}^{N_{i+1}}$};

\end{tikzpicture}\caption{Intervals from supports.}\label{fig:jm}
\end{figure}

Roughly speaking, $L_i$ is obtained from $J_i$ by successively attaching adjacent components of $\supp \psi(v_i)$ on the left
until we have included at least $k+1$ fixed points of $\psi(v_i)$
or an accumulated fixed point of $\psi(v_i)$.
By \eqref{A4} and~\eqref{A6}, the intervals $L_i$ and $R_i$ are compactly contained in $U$. 
\begin{lem}\label{lem:multiplicity}
For each $i\in\bN\cap[M_0,\infty)$, the following hold.
\be
\item
The map $\psi(v_i)$ is $k$--fixed on $L_i$ and also on $R_i$.
\item
We have that $\{x_{i-1},z_i,x_i\}\sse J_i\sse L_i\cap R_i$,
 $z_{i+1}\not\in J_i$ and $z_i\not\in J_{i+1}$.
\item
We have that
\[\sum_{j\ge M_0}\left(|L_j|+|R_j|\right)\le 2k|V| \cdot|I|.\]
\item
$
\#\{j\ge M_0\mid v_j=v_i\text{ and }J_j^*\cap J_i^*\ne\varnothing\}\le 4k$.
\ee
\end{lem} \bp 
Parts (1) and (2) are obvious from the definition and from the fact that $\cl[z_i,z_{i+1}]=2$.

For part (3),
suppose $x\in A\in\pi_0 \supp\psi(v)$ for some $v\in V$. 
There exist at most $2k$ indices $i\ge M_0$ such that
$v_j=v$ and such that  $A\sse L_j\cup R_j$.
Hence, the total number of $L_i$'s and $R_i$'s containing a given arbitrary point $x$ is at most $2k|V|$. Part (4) follows similarly.
\ep

Let us pick an integer $C\ge 8k$.
We call each $x_i$ as a \emph{ball}, and the interval $[z_i,z_{i+C})$ as a \emph{bag} (of size $C$). 
For each $m\ge M_0$, we define \begin{align*} \bag(m)&=[z_m,z_{m+C}),\\ \gap(m)&=[x_m,x_{m+C-1}]. \end{align*}
See Figure~\ref{fig:gap}.

For each $\delta>0$ and $v\in V$, we let
\begin{align*}
\operatorname{Ball}_\delta
&=
\left\{
i\in\bN\cap[ M_0,\infty)
\mid
\sup_{L_i}|\psi(v_i^{N_i})-\Id|
<\delta|L_i|
\text{ and }
\sup_{R_i}|\psi(v_i^{N_i})-\Id|<\delta|R_i|
\right\},\\
\operatorname{Bag}_\delta
&=
\left\{
i\in\bN\cap[M_0,\infty)
\mid
[i,i+C]\cap\bZ\sse \operatorname{Ball}_\delta
\right\}.\end{align*}
Intuitively speaking, 
$\operatorname{Ball}_\delta$ is the collection of balls which are 
$\delta$--fast neither 
on $L_i$ nor on $R_i$.
Also,
$\operatorname{Bag}_\delta$ is the set of bags which ``involve'' only balls from 
$\operatorname{Ball}_\delta$.
We now use the analytic estimate from Section~\ref{sec:analytic}:

\begin{lem}\label{lem:ball-bag}
For each $\delta>0$, the sets $\operatorname{Ball}_\delta$ and $\operatorname{Bag}_\delta$ have the natural density one.
\end{lem}
\bp
Let $v\in V_1$.
By Lemmas~\ref{lem:v-count} and~\ref{lem:multiplicity},
we can apply Theorem~\ref{thm:est} to $f = \psi(v)$. 
We see that \[\lim_N \frac{\#(\operatorname{Ball}_\delta\cap\bN_v\cap[0,N])}{\#(\bN_v\cap[0,N])}= 1.\]
It follows that $d_\bN(\operatorname{Ball}_\delta)=1$.
By Lemma~\ref{lem:long} (1), we have  $d_\bN(\Bag_\delta)=1$.
\ep

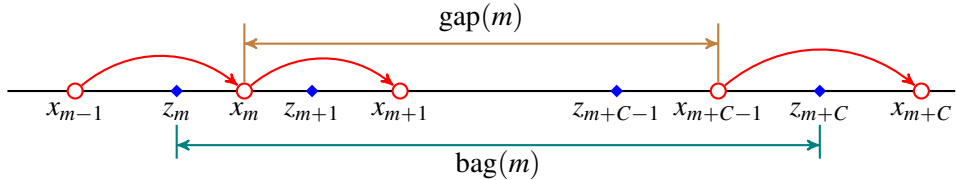
\begin{figure}[h!]
  \tikzstyle {bv}=[blue,draw,shape=diamond,fill=blue,inner sep=1pt]
  \tikzstyle {wv}=[red,draw,shape=circle,fill=white,inner sep=2pt]
\begin{tikzpicture}[>=stealth',auto,node distance=3cm, thick,scale=.9]
\draw  (-6.5,0) 
--
(-5.5,0) node (x1) [wv] {} node [below]  {\small $x_{m-1}$}
--(-4,0) node (z1) [bv] {} node [below]  {\small $z_{m}$}
-- (-3,0)  node  (x2) [wv] {} node [below] {\small $x_m$}
--(-2,0) node (z2) [bv] {} node [below]  {\small $z_{m+1}$}
-- (-.7,0)  node  (x3) [wv] {} node [below] {\small $x_{m+1}$}
--(2.5,0) node (z3) [bv] {} node [below]  {\small $z_{m+C-1}$}
-- (4,0)  node  (x4) [wv] {} node [below] {\small $x_{m+C-1}$}
--(5.5,0) node (z4) [bv] {} node [below]  {\small $z_{m+C}$}
-- (7,0)  node  (x5) [wv] {} node [below] {\small $x_{m+C}$}
-- (7.5,0);
\path (x1) edge [->,bend left,red,out = 40,in=140] node  {} (x2);
\path (x2) edge [->,bend left,red,out = 40,in=140] node  {} (x3);
\path (x4) edge [->,bend left,red,out = 40,in=140] node  {} (x5);
\draw [brown] (x2) -- (-3,1);
\draw [brown] (x4) -- (4,1);
\path (-3,.7) edge [brown,<->] node  {}  (4,.7);
\draw (.5,.7) node [above] {\small $\mathrm{gap}(m)$};
\draw [teal] (-4,-.5) -- (-4,-1);
\draw [teal] (5.5,-.5) -- (5.5,-1);
\path (-4,-.8) edge [teal,<->] node  {}  (5.5,-.8);
\draw (.75,-.7) node [below] {\small $\mathrm{bag}(m)$};
\end{tikzpicture}%
\caption{The gap in a bag.}
\label{fig:gap}
\end{figure}

\begin{lem}\label{lem:gap-bag}
For each $\delta\in(0,\frac{1}{2C}]$ and $m\in\Bag_\delta$,
we have $|\gap(m)|\le 2\delta|\bag(m)|$.
\end{lem}
\bp
Let $i\in[m+2,m+C-2]\cap\bZ$.
From Lemma~\ref{lem:multiplicity}
and
from the fact that 
\[\max(\cl(L_i),\cl(R_i))\le 2k+1,\]
we see that
either $L_i\sse\gap(m)$ or $R_i\sse\gap(m)$.
As $m\in\Bag_\delta$, we have
$i\in\Ball_\delta$ and hence,
\[|x_i-x_{i-1}| = |v_i^{N_i}x_{i-1}-x_{i-1}| <
\delta\min(|L_i|,|R_i|)
\le \delta|\gap(m)|.\]

By a similar argument,
\begin{align*}
|x_{m+1}-x_m| +|x_{m+C-1}-x_{m+C-2}| 
&<\delta \left(|R_{m+1}| + | L_{m+C-1}| \right)\\
&=\delta \left(|R_{m+1}\cup L_{m+C-1}| \right)
\le \delta |\bag(m)|.\end{align*}

By summing up $|x_i-x_{i-1}|$ for $i=m+1,\ldots,m+C-1$, we obtain that
\begin{align*}
|\gap(m)|&\le (C-3)\delta|\gap(m)|  + \delta|\bag(m)|,\\
|\gap(m)|&\le\frac{\delta}{1- (C-3)\delta}|\bag(m)|
\le 2\delta |\bag(m)|.\qedhere
\end{align*}
\ep

Recall $J_m^*=L_m\cup R_m$.
For each $\lambda>0$, we define
\begin{align*}
D_{C,\lambda}:=\{m\in\bN\cap[M_0,\infty)\co &
\text{ either }
|x_m-x_{m-1}|>\lambda |x_m-\sup J^*_m|\\
&\text{ or } |x_{m+C}-x_{m+C-1}|>\lambda |x_{m+C-1}-\inf J^*_{m+C}|\}.
\end{align*}

\begin{lem}\label{lem:deep}
If $\delta\in(0,\frac{1}{2C}]$ and $2\delta(1+\lambda)\le 1$,
then $\Bag_\delta\sse D_{C,\lambda}$.
\end{lem}
\bp
Assume that $m\in \Bag_\delta\setminus D_{C,\lambda}$.
By Lemma~\ref{lem:gap-bag}, we have
\begin{align*}
|\bag(m)|&< |x_m-x_{m-1}|+|x_{m+C}-x_{m+C-1}|+|\gap(m)|\\
&\le
\lambda |x_m-\sup J^*_m| + \lambda |x_{m+C-1}-\inf J^*_{m+C}|
+|\gap(m)|\\
&\le (1+\lambda)|\gap(m)|\le 2\delta(1+\lambda)|\bag(m)|.\end{align*}
This is a contradiction.
\ep

\begin{lem}\label{lem:deeper}
For all $\lambda\ge1$, the following set has the natural density one.
\[E_\lambda=\left\{ m\in\bN\cap[M_0,\infty)\mid \psi(v_m^{N_m})\text{ is }\lambda\text{--expansive on }J_m^*\right\}.\]
\end{lem}

\bp 
We may assume $\lambda>8k$. For $\delta>0$, we define
\begin{align*}
X_\lambda&=\{m\in\bN\cap[M_0,\infty) \co |x_m-x_{m-1}|>\lambda|x_m-\sup J^*_m|\},\\
Y_\lambda&=\{m\in\bN\cap[M_0,\infty) \co |x_m-x_{m-1}|>\lambda|x_{m-1}-\inf J^*_m|\}.
\end{align*}
Then we see
\[D_{C,\lambda}=X_\lambda\cup \left((Y_\lambda-C)\cap[M_0,\infty)\right).\]

Lemmas~\ref{lem:ball-bag} and~\ref{lem:deep} imply that  $d_\bN(D_{C,\lambda})=1$,
Hence by Lemma~\ref{lem:long}, we obtain that 
$d_\bN(X_\lambda\cup Y_\lambda)=1$.
This implies $d_\bN(E_\lambda)=1$.
 \ep

\bp[Completing the proof of the Slow Progress Lemma]
We see from Lemma~\ref{lem:v-count} and Theorem~\ref{thm:est} that
\[\lim_N \frac{\#(E_\lambda\cap\bN_v\cap[0,N])}{\#(\bN_v\cap[0,N])}= 0\]
for each $v\in V_1$. 
This implies $d_\bN(E_\lambda)=0$, contradicting Lemma~\ref{lem:deeper}.
Hence the assumption~\eqref{A3} is false and the proof is complete.
\ep

\subsection{Consequences of the Slow Progress Lemma}
The following is the main obstruction of algebraic smoothing in the Main Theorem.

\begin{lem}\label{lem:kernel-psi}
Let $u\in G$ and let $U\in\pi_0 \supp \psi(G)$.
If $\supp\psi(u)\cap U$ is compactly contained in $U$,
then for each real number $T_0>0$
and for all sufficiently large $i\in\bN$,
 there exists $h_i\in G$ such that the following hold:
\be[(i)]
\item
$\|h_i\|<2i-T_0$;
\item
$U\cap\supp\psi[w_i uw_i^{-1} , h_i w_i uw_i^{-1} h_i^{-1}]=\varnothing$.
\item
For each $v\in V$ and for at least one $h'\in \{v\cdot h_i,v^{-1}\cdot h_i\}$,
 we have 
\[U\cap\supp\psi[w_i uw_i^{-1} , h' w_i uw_i^{-1} (h')^{-1}]=\varnothing;\]
\ee
\end{lem}

\bp
Let $u, U$ and $T_0$ be given as in the hypothesis.
We write 
\[
x=\inf(\supp\psi(u)\cap U),\quad
y=\sup(\supp\psi(u)\cap U).\]
Put $T = \cd(x,y)$. 
By the Slow Progress Lemma, whenever $i\gg0$ we have
\[\cd(x,w_i x)<i-(T_0+T),\quad \cd(y,w_i y)<i-(T_0+T).\]
\[
\cd(w_ix,w_iy)\le 2i-2(T_0+T)+T< 2i-T_0.\]
Put $u_i = w_i uw_i^{-1}$.
Since $\supp \psi(u_i)\cap U\sse (w_ix,w_iy)$,
we see from Lemma~\ref{lem:slide} that there exists $h_i\in G$ with $\|h_i\|<2i - T_0$ satisfying $h_iw_ix>w_iy$. Furthermore, for each $v\in V$ there is a $s(v)\in\{1,-1\}$ such that $v^{s(v)}h_iw_ix\ge h_iw_ix >w_iy$.
We see that 
\[
(\supp \psi(u_i)\cap U)\cap h (\supp \psi(u_i)\cap U)=\varnothing\]
if $h=h_i$ or if $h=v^{s(v)}h_i$ for some $v\in V$.
This gives the desired relations.
\ep

\section{A dynamically fast subgroup of $\Diff_+^{k,\mu}(I)$}\label{sec:phi}
Recall we have defined $G^\dagger$ in Section~\ref{ss:univ}.
We will now build a representation $\phi\colon G^\dagger \to\Diff_0^{k,\mu}(I)$ such that $\supp\phi(G^\dagger )$ is connected and 
$\phi(G^\dagger)$
admits no injective homomorphisms into $\Diff^{k,\omega}_+(I)$ for all $0\prec_k\omega\ll\mu$.

The criticality of the regularity will be encoded in a dynamically fast condition described as follows.
As in Lemma~\ref{lem:universal}, we let $1\neq u_0\in G^\dagger $ be given such that $\supp\phi(u_0)$ is compactly contained in $\supp\phi(G^\dagger)$. We build a sequence a elements $\{w_i\}_{i\geq 0}\sse G^\dagger $ which depend on $k,\mu$ such that, after replacing $u_0$ by a suitable conjugate $u$ in $G^\dagger $ if necessary, we have \[\cd(\inf\supp\phi(w_i u w_i^{-1}),\sup\supp\phi(w_iu w_i^{-1}))\geq 2i.\] We build the representation $\phi$ in several steps.

\subsection{Setting up notation}\label{ss:global}
Let us prepare some notation which we will use throughout this section. 
We fix $k\in\bN$ and $\mu\gg\omega_1$. We put $\delta=9/10$ and recall the notation 
\[\{\epsilon_0,\delta_0,\ell_0^*,\ell_i,N_i\}\] 
from Setting~\ref{setting:kmu} and from Corollary~\ref{cor:optimal}.
Namely, we pick a universal constant $\epsilon_0\in(0,1)$,
and define $\delta_0\ge 9/10$ from $\epsilon_0$. For instance, we can set $\epsilon_0=1/1000$.
We have defined a constant $\ell_0^*$ depending on $\mu$,
so that 
\[\ell_0^*,\mu(\ell_0^*)\in(0,\epsilon_0].\]

We will choose $K^*\in\bN$, and let
\[\ell_i=1/\left((i+K^*)\log^2(i+K^*)\right).\]
We have defined another sequence
\[
N_i=\lceil 1/\left({\ell_i^{k-1}\mu(\ell_i)}\right)\rceil.\]

Possibly after increasing $K^*>0$, we may assume that 
 $\ell_1\le\ell_0^*$
and that  \[ \kappa:=\ell_2/(2\ell_2+\ell_1) > 1/4.\]
In Corollary~\ref{cor:optimal}, we verified that
for all concave modulus $0\prec_k \omega\ll\mu$ we have
\[\lim_{i\to\infty} N_i(1/i)^{k-1}\omega(1/i)=0.\]

Recall we have a generating set $V^\dagger=\{\bba,\bbb,\bbc,\bbd,\bbe\}\sse G^\dagger$ as in Section~\ref{ss:univ}.
For $i\in\bN$, we let $v_{2i-1}=\bbb$ and $v_{2i}=\bba$.
Define a sequence $\{w_i\}_{i\in\bN}\sse G^\dagger $ 
by $w_0=1$ and $w_i=v_i^{N_i}\cdot w_{i-1}$.

\subsection{A configuration of intervals in $I$}
\label{ss:config}
Let us now build an infinite chain 
\[
\mathcal{F}=\left(\ldots,L_2^-,L_1^-,D^-,C^-,B^-,I_0,B^+,C^+,D^+,L_1^+,L_2^+,\ldots\right)\]
 of bounded open intervals in $\bR$
as shown in Figure~\ref{f:config1}.
The union of $\FF$ will be also bounded.
We will simultaneously define representations 
\[
\rho_0,\rho_1,\rho_2\co G^\dagger\to\Diff_+^{k,\mu}(\bR).\]

As in Lemma~\ref{lem:universal}, 
we put 
\[u^\dagger=\left[\left[\bbc^\bbd,\bbe\cdot\bbe^\bbd\cdot\bbe^{-1}\right],\bbc\right]\in G^\dagger\setminus\{1\}.\]
The standard affine action of $\BS(1,2)$ is conjugate to a $C^\infty$--action on $\bR$ supported in $[0,1]$; see \cite{Tsuboi1984Asterisque} or \cite[Section 4.3]{Navas2011}, for instance. 
Applying Lemma~\ref{lem:res} to  
\[1\ne u^\dagger\in(\form{c}\times\form{\bba,\bbe})\ast\form{\bbd}\le G^\dagger,\] we have an action
\[
\rho_0\co G^\dagger\to\Diff^\infty_+(\bR)\]
such that $\rho_0(\bbb)=1$, 
$\rho_0(u^\dagger)\ne1$ and moreover,
 $I_0:=(-1,1)=\supp\rho_0$.
By the same lemma, we can also require that 
 \[
 \rho_0\form{\bba,\bbe}\cong\form{\bba,\bbe}\cong\BS(1,2).\]
 
We will include six more open intervals \[B^\pm, C^\pm, D^\pm\] to the chain $\FF$ as shown in the configuration (Fig~\ref{f:config1}).
We will require that $B^-=-B^+$ and so forth, where we use the notation \[ -(r,s) = (-s,-t)\] for $0\le r<s\le \infty$. 
Also, we set $\sup C^+=2$ and $\sup D^+=3$.

By  Corollary~\ref{cor:chain-smooth}, 
there exists a $C^{\infty}$ diffeomorphisms $c_1^+,d_1^+$ supported on $C^+,D^+$ respectively such that $\form{c_1^+,d_1^+}\cong F$ and $\form{c_1^+,d_1^+}$ acts locally densely on $C^+\cup D^+$. 
We may require $c_1^+(x)>x$ for $x\in C^+$ and $d_1^+(x)>x$ for $x\in D^+$.
We define $c_1^-, d_1^-$ symmetrically so that
$c_1^-(-x)=-c_1^+(x)$ and $d_1^-(-x)=-d_1^+(x)$.
In particular, 
\[ \supp d_1^\pm =D^\pm,\quad \supp c_1^\pm =C^\pm.\]
We choose $b_1\in\Diff^\infty_+(\bR)$ supported on $B^+\cup B^-$ such that 
$b_1(x)>x$ for $x\in B^+$ and $b_1(x)<x$ for $x\in B^-$.
We define 
\[
\rho_1\co G^\dagger\to\Diff^\infty_+(\bR)\]
by $\rho_1(\bba)=\rho_1(\bbe)=1$ and $\rho_1(\bbb)=b_1,\rho_1(\bbc)=c_1^+c_1^-, \rho_1(\bbd)=d_1^+d_1^-$.

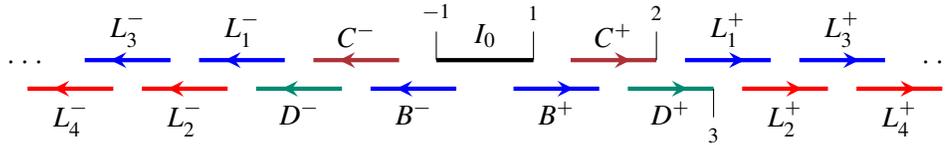
\begin{figure}[h!]
{
\tikzstyle {Av}=[red,draw,shape=circle,fill=red,inner sep=1pt]
\tikzstyle {Bv}=[blue,draw,shape=circle,fill=blue,inner sep=1pt]
\tikzstyle {Cv}=[Maroon,draw,shape=circle,fill=Maroon,inner sep=1pt]
\tikzstyle {Dv}=[PineGreen,draw,shape=circle,fill=PineGreen,inner sep=1pt]
\tikzstyle {Ev}=[Orange,draw,shape=circle,fill=Plum,inner sep=1pt]
\tikzstyle {A}=[red,postaction=decorate,decoration={%
    markings,%
    mark=at position .7 with {\arrow[red]{stealth};}}]
\tikzstyle {B}=[blue,postaction=decorate,decoration={%
    markings,%
    mark=at position .7 with {\arrow[blue]{stealth};}}]
\tikzstyle {C}=[Maroon,postaction=decorate,decoration={%
    markings,%
    mark=at position .7 with {\arrow[Maroon]{stealth};}}]
\tikzstyle {D}=[PineGreen,postaction=decorate,decoration={%
    markings,%
    mark=at position .7 with {\arrow[PineGreen]{stealth};}}]
\tikzstyle {E}=[Orange,postaction=decorate,decoration={%
    markings,%
    mark=at position .7 with {\arrow[Orange]{stealth};}}]
\begin{tikzpicture}[ultra thick,scale=.38]
\path (1,0) edge [B] node  {}  (4,0);
\draw (2.5,0) node [below] {\small $B^+$}; 

\path (3,1) edge [C] node  {}   (6,1);
\draw (4.5,1) node [above] {\small $C^+$};

\path (5,0)  edge [D] node  {}   (8,0);
\draw (6.5,0) node [below] {\small $D^+$}; 
\path    (7,1)  edge [B] node  {}   (10,1);
\draw (8.5,1) node [above] {\small $L_1^+$};

\path  (9,0)  edge [A] node  {}   (12,0);
\draw (10.5,0) node [below] {\small $L_2^+$};

\path  (11,1)  edge [B] node  {}   (14,1);
\draw (12.5,1) node [above] {\small $L_3^+$};
\path  (13,0)  edge [A] node  {}   (16,0);
\draw (14.5,0) node [below] {\small $L_4^+$};

\path   (-1.7,1)  edge node  {}   (1.7,1);
\draw (0,1) node [above] {\small $I_0$};

\path   (-1,0)  edge [B] node  {}   (-4,0);
\draw (-2.5,0) node [below] {\small $B^-$}; 
\path   (-3,1)  edge [C] node  {}   (-6,1);
\draw (-4.5,1) node [above] {\small $C^-$};
\draw (16,.2) node [above] {$\cdots$};
\draw (-16,.2) node [above] {$\cdots$};

\path (-5,0)  edge [D] node  {}   (-8,0);
\draw (-6.5,0) node [below] {\small $D^-$}; 

\path  (-7,1)  edge [B] node  {}   (-10,1);
\draw (-8.5,1) node [above] {\small $L_1^-$};

\path  (-9,0)  edge [A] node  {}   (-12,0);
\draw (-10.5,0) node [below] {\small $L_2^-$};

\path  (-11,1)  edge [B] node  {}   (-14,1);
\draw (-12.5,1) node [above] {\small $L_3^-$};
\path  (-13,0)  edge [A] node  {}   (-16,0);
\draw (-14.5,0) node [below] {\small $L_4^-$};

\draw  [thin,black] (8,0) -- (8,-1) node [below,black] {\tiny $3$};
\draw  [thin,black] (6,1) -- (6,2) node [above,black] {\tiny $2$};
\draw  [thin,black] (1.7,1) -- (1.7,2) node [above,black] {\tiny $1$};
\draw  [thin,black] (-1.7,1) -- (-1.7,2) node [above,black] {\tiny $-1$};

\end{tikzpicture}}
\caption{The family of bounded open intervals $\mathcal{F}$.
Intervals of the same color are supporting the same generator
in Figure~\ref{f:gdagger}. }
\label{f:config1}
\end{figure}

Note that $\ell_1/\ell_2<2$ and that the sequence $\{\ell_i/\ell_{i+1}\}$ decreases to $1$.
Hence, 
\[
\frac13>\frac{\ell_{i+1}}{2\ell_{i+1}+\ell_i}\ge\kappa
=\frac{\ell_2}{2\ell_2+\ell_1}>\frac14.
\]

Let us inductively define 
\begin{align*}
L^+_1 &= (3-\kappa\ell_1,3-\kappa\ell_1+\ell_1),\\
L^+_{i+1} &= (\sup L^+_i - \kappa\ell_i,\sup L^+_i-\kappa\ell_i+\ell_{i+1}).
\end{align*}
Note that $|L^+_i\cap L^+_{i+1}|=\kappa\ell_i$;
see  Figure~\ref{f:lili}.
Since \[\kappa\ell_{i-1} + \kappa\ell_{i}\le \ell_{i}-\kappa\ell_{i}<\ell_{i},\] we see that 
$L^+_{i-1}\cap L^+_{i+1}=\varnothing$. 
In other words, the collection $\{L_i^+\}$ has no triple intersections.
Then we define symmetrically $L_i^-=-L_i^+$ and add $L_i^\pm$ to $\FF$.
This completes the definition of the infinite chain $\FF$.
As $\sum_i \ell_i<\infty$, there exists some compact interval $I$ such that  \[\overline{\bigcup\FF}=I\sse\bR.\]

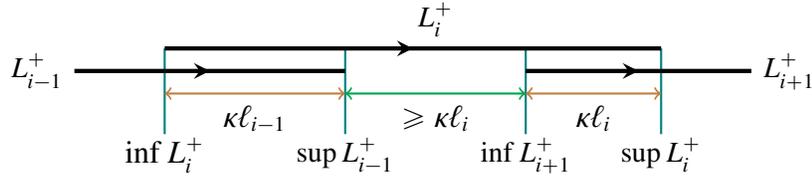
\begin{figure}[h!]
\tikzstyle {Ar}=[black,postaction=decorate,decoration={%
    markings,%
    mark=at position .5 with {\arrow[black]{stealth};}}]
\tikzstyle {Av}=[red,draw,shape=circle,fill=red,inner sep=1pt]
\tikzstyle {Bv}=[blue,draw,shape=circle,fill=blue,inner sep=1pt]
\tikzstyle {Cv}=[Maroon,draw,shape=circle,fill=Maroon,inner sep=1pt]
\tikzstyle {Dv}=[PineGreen,draw,shape=circle,fill=PineGreen,inner sep=1pt]
\tikzstyle {Ev}=[Orange,draw,shape=circle,fill=Plum,inner sep=1pt]
\begin{tikzpicture}[thick,scale=.6]
\draw  [teal] (-6,.5) -- (-6,-1.4);
\draw  [teal] (5,.5) -- (5,-1.4);
\draw  [teal] (-2,.5) -- (-2,-1.4);
\draw  [teal] (2,.5) -- (2,-1.4);
\path (-6,-.5) edge [brown,<->] node  {}  (-2,-.5);
\path (-2,-.5) edge [Green,<->] node  {}  (2,-.5);
\path (5,-.5) edge [brown,<->] node  {}  (2,-.5);
\draw  (-6,-1.2) node [below] {$\inf L_{i}^+$};
\draw  (5,-1.2) node [below] {\small $\sup L_{i}^+$};
\draw  (2,-1.2) node [below] {\small $\inf L_{i+1}^+$};
\draw  (-2,-1.2) node [below] {\small $\sup L_{i-1}^+$};
\draw (-4,-.5) node [below] { $\kappa\ell_{i-1}$};
\draw (3.5,-.5) node [below] { $\kappa\ell_{i}$};
\draw (0,-.5) node [below] { $\ge\kappa\ell_{i}$};
\draw (0,.5) node [above] {\small $L_i^+$}; 
\path (-8,0) edge [ultra thick,Ar] (-2,0);
\path  (-6,.5)  edge [ultra thick,Ar] (5,.5);
\draw (-8,0) node [left]   {\small $L_{i-1}^+$};
\path  (2,0)  edge [ultra thick,Ar] (7,0);
\draw (7,0) node [right]   {\small $L_{i+1}^+$}; 
\end{tikzpicture}
\caption{The bounded open intervals $L_i$'s.}
\label{f:lili}
\end{figure}

By applying Theorem~\ref{thm:optimal} to the parameter
\[\left(k,\mu,\delta_0,\{N_{2i-1}\}_{i\in\bN},\left\{\overline{L_{2i-1}^+}\right\}_{i\in\bN}\right),\]
we obtain a diffeomorphism $b_2^+\in\Diff_+^{k,\mu}(\bR)$ supported on $\cup_i L_{2i-1}^+$ such that $b_2^+$ is $\delta_0$--fast on each $L_{2i-1}^+$,
and such that  $b_2^+(x)>x$ for each $x\in \cup_i L_{2i-1}^+$.
Note that we are invoking the hypothesis that
\[N_{2i-1}\cdot \ell_{2i-1}^{k-1}\cdot\mu(\ell_{2i-1})\ge 1.\]
We define $b_2^-(x)=-b_2^+(-x)$.
We also define $a_2^\pm$ completely analogously with respect to the parameter
\[\left(k,\mu,\delta_0,\{N_{2i}\}_{i\in\bN},\left\{\overline{L_{2i}^+}\right\}_{i\in\bN}\right).\]
Then we define 
\[
\rho_2\co G^\dagger\to\Diff_+^{k,\mu}(\bR)\]
by $\rho_2(\bba)=a_2^+a_2^-,
\rho_2(\bbb)=b_2^+b_2^-$
and $\rho_2(\bbc)=\rho_2(\bbd)=\rho_2(\bbe)=1$.

For each $v\in V^\dagger$, we define \[\phi(v) = \rho_0(v) \rho_1(v) \rho_2(v).\]
We see from the construction that
\begin{itemize}
\item $\supp\phi\form{\bba,\bbe}\cap\supp\phi(\bbc)=\varnothing$;
\item $\phi(\bba)\phi(\bbe)\phi(\bba)^{-1}=\rho_0(\bba)\rho_0(\bbe)\rho_0(\bba)^{-1}=\rho_0(\bbe)^2 = \phi(\bbe)^2$.
\end{itemize}
Hence, the map $\phi$ extends to a group action
\[\phi\co G^\dagger\to\Diff_0^{k,\mu}(I).\]

Let us summarize the properties of $\phi$ below.
The proofs are obvious from construction and from Theorem~\ref{thm:optimal}.
We continue to use the notation from Section~\ref{ss:global}.

\begin{lem}\label{lem:phi}
The following hold for 
$\phi=\phi_{k,\mu}\co G^\dagger\to\Diff_0^{k,\mu}(I)$.
\be
\item
$\supp\phi=I\setminus \partial I$.
\item
For each $g\in G^\dagger$, the restriction $\phi(g)\restriction_{I\setminus\partial I}$ is a $C^\infty$ diffeomorphism.
\item
For each $i\ge1$, the map $\phi(v_i^{N_i})$ is $\delta_0$--fast on $L_i^\pm$.
\item\label{p:acc}
Every orbit of 
 $\phi\form{\bba,\bbc,\bbd,\bbe}$ in $I_0$ is accumulated at $\partial I_0$.
 \ee
\end{lem}

\subsection{The behavior of $\{w_i\}_{i\geq 0}$ under $\phi$}
Whereas we have good control over the compactly supported diffeomorphism $\phi(u)$, we will need to have good control over commutators of conjugates of $\phi(u)$.

\begin{lem}\label{lem:stretch-interval}
For each nonempty open interval 
$U_0\sse \supp\phi(G^\dagger )$, there exists a suitably chosen $f\in\phi(G^\dagger )$ such that $f(U_0)\cap L_1^+\neq\varnothing$ and such that $f(U_0)\cap L_1^-\neq\varnothing$.
\end{lem}

Intuitively, Lemma~\ref{lem:stretch-interval} says that no matter how small an interval we choose inside $\supp\phi(G^\dagger )$, we may find an element of $f\in \phi(G^\dagger )$ so that $f(U_0)$ stretches across
 \[ I_0\cup B^\pm \cup C^\pm \cup D^\pm.\]
Of course, $f(U_0)$ might be much larger than this union, though this is unimportant.

\begin{proof}[Proof of Lemma~\ref{lem:stretch-interval}]
Let $U_0=(z_1,z_2)$ be given as in the hypotheses of the lemma. By Lemmas~\ref{lem:phi} (\ref{p:acc}) and~\ref{lem:slide}, there exists an $f\in\phi(G^\dagger )$ such that $f(z_2)\in D^+\cap L_1^+$. 
So, we may assume $z_2\in D^+\cap L_1^+$. 
We may then assume that $z_1\ge \sup L_1^-$; for, otherwise there is nothing to show.
There are four (overlapping) cases to consider.

{\bf Case 1: $z_1\in B^-\cup C^-\cup D^-$}.

For sufficiently large $n_1,n_2,n_3\in\bN$
and for $f_1=\phi(\bbd^{n_3}\bbc^{n_2}\bbb^{n_1})$, 
we have 
$f_1(z_2)\in L_1^+\setminus D^+$
and
 $f_1(z_1)\in D^-\cap L_1^-$. This is the desired configuration.
 
{\bf Case 2: $z_1\in I_0$}.
 
By Lemma~\ref{lem:phi}  (\ref{p:acc}), there is  $f_1\in\phi\form{\bba,\bbc,\bbd,\bbe}$ such that 
$f_1(z_1)\in B^-\cap I_0$. 
Note that 
\[Q:=(B^+\cup L_1^+)\setminus (I_0\cup C^+\cup D^+\cup L_2^+)\] is a nonempty set which is disjoint from $\supp\phi\form{\bba,\bbc,\bbd,\bbe}$. Hence, $f_1(z_2)\not\in Q$; see Figure~\ref{f:rightofE}.
We have $f_1(z_2)\in C^+\cup D^+$.
As in Case 1, we can find sufficiently large $n_1,n_2,n_3\in\bN$
such that for $f_2=\phi(\bbd^{n_3}\bbc^{n_2}\bbb^{n_1})$
we have $f_2f_1(z_1)\in  L_1^-$. 
and  $f_2f_1(z_2)\in L_1^+$. This is the desired.

{\bf Case 3: $z_1\in B^+$}.

There exist sufficiently large $n_1\in\bN$
such that for $f_1=\phi(\bbb^{-n_1})$, 
we have 
$f_1(z_2)\in D^+  \cap L_1^+$
and
 $f_1(z_1)\in I_0\cap B^+$. So, we again have Case 2.

{\bf Case 4: $z_1\in C^+\cup D^+$}. 

We use the fact that the restriction of $\phi\form{\bbc,\bbd}$ to $C^+\cup D^+$ generates a locally dense copy of Thompson's group $F$. 
As we have seen in Section~\ref{ss:dense-F}, 
 for some suitable $f_1\in\phi\form{\bbc,\bbd}$
  we may arrange$f_1(z_1)\in B^+\cap C^+$  and $f_1(z_2)\in D^+\cap L_1^+$, thus reducing to the previous case.
\end{proof}
 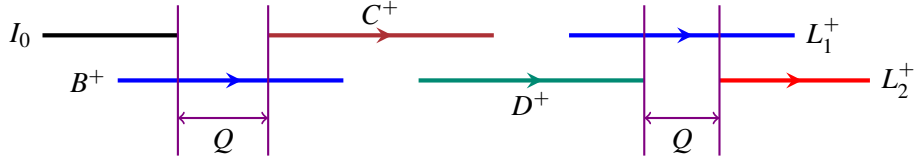
\begin{figure}[h!]
{
\tikzstyle {Av}=[red,draw,shape=circle,fill=red,inner sep=1pt]
\tikzstyle {Bv}=[blue,draw,shape=circle,fill=blue,inner sep=1pt]
\tikzstyle {Cv}=[Maroon,draw,shape=circle,fill=Maroon,inner sep=1pt]
\tikzstyle {Dv}=[PineGreen,draw,shape=circle,fill=PineGreen,inner sep=1pt]
\tikzstyle {Ev}=[Orange,draw,shape=circle,fill=Plum,inner sep=1pt]
\tikzstyle {A}=[red,postaction=decorate,decoration={%
    markings,%
    mark=at position .55 with {\arrow[red]{stealth};}}]
\tikzstyle {B}=[blue,postaction=decorate,decoration={%
    markings,%
    mark=at position .55 with {\arrow[blue]{stealth};}}]
\tikzstyle {C}=[Maroon,postaction=decorate,decoration={%
    markings,%
    mark=at position .55 with {\arrow[Maroon]{stealth};}}]
\tikzstyle {D}=[PineGreen,postaction=decorate,decoration={%
    markings,%
    mark=at position .55 with {\arrow[PineGreen]{stealth};}}]
\tikzstyle {E}=[Orange,postaction=decorate,decoration={%
    markings,%
    mark=at position .55 with {\arrow[Orange]{stealth};}}]
\begin{tikzpicture}[scale=1]
\path (1,0) edge [B,ultra thick] node  {}  (4,0);
\draw (1,0) node [left] {\small $B^+$}; 

\path (3,.6) edge [C,ultra thick] node  {}   (6,.6);
\draw (4.5,.6) node [above] {\small $C^+$};

\path (5,0)  edge [D,ultra thick] node  {}   (8,0);
\draw (6.5,0) node [below] {\small $D^+$}; 
\path    (7,.6)  edge [B,ultra thick] node  {}   (10,.6);
\draw (10,.6) node [right] {\small $L_1^+$};

\path  (9,0)  edge [A,ultra thick] node  {}   (11,0);
\draw (11,0) node [right] {\small $L_2^+$};

\path   (0,.6)  edge [ultra thick] node  {}   (1.8,.6);
\draw (0,.6) node [left] {\small $I_0$};

\path (1.8,-.5) edge [violet,<->,thick] node  {}  (3,-.5);
\draw [violet,thick] (1.8,1) -- (1.8,-1) (3,1) -- (3,-1) (8,1) -- (8,-1) (9,1) -- (9,-1);
\path (8,-.5) edge [violet,<->,thick] node  {}  (9,-.5);
\draw (2.4,-.5) node [below] {\small $Q$};
\draw (8.5,-.5) node [below] {\small $Q$};
\end{tikzpicture}
}
\caption{The point $f_1(z_2)$ stays in $C^+\cup D^+$.}
\label{f:rightofE}
\end{figure}

We retain the elements $\{w_i\}_{i\in\bN}$ as defined in Section~\ref{ss:global}. The following lemma measures the complexity of certain diffeomorphisms in $\phi(G^\dagger)$ and shows that the complexities grow linearly.
\begin{lem}\label{lem:linear-growth}
Let $u\in G^\dagger\setminus\ker\phi$
be an element such that 
$\supp\phi(u)$ is compactly contained in $\supp\phi$.
Then for some conjugate $u'\in G^\dagger$ of $u$, and for some component $U_1$ of $\supp\phi(u')$, 
we have that whenever $i\in \bN$
the bounded open interval $\phi(w_i)U_1$ intersects 
both $L_{i+1}^+$ and $L_{i+1}^-$.  
 In particular, 
we have that  \[\cl(\phi(w_i)U_1) > 2i,\]
and that $\partial (\phi(w_i)U_1)\sse\supp \phi(a)\cup \supp\phi(b)$.
\end{lem}
\begin{proof}
Choose an open interval $U_0\in\pi_0\supp\phi(u)$
compactly contained in $I$. By Lemma~\ref{lem:stretch-interval}, 
there is a conjugate $u'\in G^\dagger$ of $u$ such that the image $U_1$ of $U_0$ under this conjugation intersects $L_1^\pm$. 
Conjugating by a further power of $b$ if necessary, we may assume $(s^-,s^+)\sse U_1$ 
for some $s^\pm$ satisfying the following.
\begin{align*}
&\inf L_1^+ + (1-\delta_0)\ell_1< s^+ <\sup L_1^+,\\
&\inf L_1^- <  s^- <\sup L_1^- - (1-\delta_0)\ell_1.
\end{align*}
Note $1-\delta_0\le 1/10$. See Figure~\ref{f:p4l4}. 
We now apply $\phi$ to the conjugates $w_iu'w_i^{-1}$.

\begin{figure}[h!]
{
\tikzstyle {Av}=[red,draw,shape=circle,fill=red,inner sep=1.5pt]
\tikzstyle {Bv}=[blue,draw,shape=circle,fill=blue,inner sep=1.5pt]
\tikzstyle {Cv}=[Maroon,draw,shape=circle,fill=Maroon,inner sep=1.5pt]
\tikzstyle {Dv}=[PineGreen,draw,shape=circle,fill=PineGreen,inner sep=1.5pt]
\tikzstyle {Ev}=[Orange,draw,shape=circle,fill=Orange,inner sep=1.5pt]
\tikzstyle {A}=[red,postaction=decorate,decoration={%
    markings,%
    mark=at position .5 with {\arrow[red]{stealth};}}]
\tikzstyle {B}=[blue,postaction=decorate,decoration={
    markings,
    mark=at position .5 with {\arrow[blue]{stealth};}}]
\tikzstyle {C}=[Maroon,postaction=decorate,decoration={
    markings,
    mark=at position .7 with {\arrow[Maroon]{stealth};}}]
\tikzstyle {D}=[PineGreen,postaction=decorate,decoration={ 
    markings, 
    mark=at position .7 with {\arrow[PineGreen]{stealth};}}]
\tikzstyle {E}=[Orange,postaction=decorate,decoration={ 
    markings, 
    mark=at position .7 with {\arrow[Orange]{stealth};}}]
\begin{tikzpicture}[thick,scale=.6]
\draw (-6,.5) -- (6,.5);
\draw  [teal] (-6,.5) -- (-6,-1.4);
\draw  [teal] (6,.5) -- (6,-1.4);
\draw  [teal] (-2,.5) -- (-2,-1.4);
\draw  [teal] (2,.5) -- (2,-1.4);
\path (-6,-.5) edge [brown,<->] node  {}  (-2,-.5);
\path (6,-.5) edge [brown,<->] node  {}  (2,-.5);
\draw  (-6,-1.2) node [below] {\small $s^-$};
\draw  (6,-1.2) node [below] {\small $s^+$};
\draw  (2,-1.2) node [below] {\small $\inf L_1^+$};
\draw  (-2,-1.2) node [below] {\small $\sup L_1^-$};
\draw (-4,-.5) node [below] {\tiny $> (1-\delta_0)\ell_1$};
\draw (4,-.5) node [below] {\tiny $> (1-\delta_0)\ell_1$};
\path (-2,0) edge [ultra thick,B] node {} (-8,0);
\path (2,0) edge [ultra thick,B] node {} (8,0);
\draw (-8,0) node [left]   {\small $L_1^-$};
\draw (8,0) node [right]   {\small $L_1^+$};
\end{tikzpicture}}
\caption{Replacing $u$ by a suitable conjugate $u'$.}
\label{f:p4l4}
\end{figure}
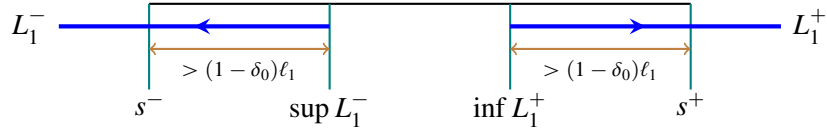

Assume by induction that 
\begin{align*}
&\inf L_i^+ +(1-\delta_0)\ell_i < \phi(w_{i-1})s^+<\sup L_i^+,\\
&\inf L_i^-<  \phi(w_{i-1})s^-<\sup L_i^- - (1-\delta_0)\ell_i.
\end{align*}
As $\phi(v_i^{N_i})$ is $\delta_0$--fast on $L_i^+$,
there is $x_i\in L_i^+$ such that 
$\phi(v_i^{N_i})x_i\ge x_i+\delta_0 \ell_i$.
Then 
\begin{align*}
\phi(w_{i-1})s^+
&\ge\inf L_i^++(1-\delta_0)\ell_i
=\sup L_i^+ -\delta_0\ell_i \ge \phi(v_i^{N_i})x_i-\delta_0\ell_i\ge x_i.\\
\phi(w_i)s^+
&\ge \phi(v_i^{N_i})x_i\ge 
x_i+\delta_0\ell_i
\ge
\inf L_i^+ +\delta_0\ell_i
=\sup L_i^+-(1-\delta_0)\ell_i\\
&=\inf L_{i+1}^+ +(\kappa -1+\delta_0)\ell_i
>
\inf L_{i+1}^+ +(1-\delta_0)\ell_{i+1}.
\end{align*}
Here, we used $\kappa>1/4>2(1-\delta_0)$.
By induction,
 we see that $\phi(w_i)s^\pm\in L_{i+1}^\pm$.

In order to cover $\phi(w_i)U_1$ by intervals in $\FF$,
we need at least
\[
\{I_0,B^\pm, C^\pm,D^\pm,L_1^\pm,\ldots,L_{i}^\pm\}.\]
The conclusion is now obvious.
\end{proof}

\subsection{Certificates of non-commutativity}

The following fact will be used in order to show that $\phi(G^\dagger )$ cannot be smoothed algebraically.

\begin{lem}\label{lem:nontrivial}
Suppose we have $u\in G^\dagger$ such that
 $\supp\phi(u)$ is compactly contained in $\supp\phi=I\setminus\partial I$,
 and $U\in\pi_0\supp\phi(u)$.
If $h\in G^\dagger$ satisfies that $\phi(h)U\ne U$ and that $||h||<\cd(\inf U, \sup U)$,
then $[\phi(u),\phi(huh^{-1})]\neq 1$.
\end{lem}
\begin{proof}
Write $U=(z_1,z_2)$ and $\cd(z_1,z_2)=N<\infty$.
We claim $\phi(h)U\cap U\neq\varnothing$.
For, otherwise 
we have 
either $\phi(h)z_1\geq z_2$ or $\phi(h)z_2\leq z_1$.
But this would imply that one of the following holds:
\begin{itemize}
\item $\cd(\phi(h)z_1,z_1)\geq N$;
\item $\cd(\phi(h)z_2,z_2)\geq N$.
\end{itemize}
This then violates Lemma~\ref{lem:syll-cover}.

Let $f = \phi(u)$ and $g = \phi(huh^{-1})$. Since $\supp f$ is compactly contained in $I$, 
there exists a compact interval $J$ such that 
\[\supp f\cup\supp g \sse J\sse I\setminus \partial I.\]
Since $\phi(G^\dagger)$ is $C^\infty$ at each point $x\in I\setminus\partial I$, we may regard
$f,g\in \Diff^\infty_+(J)$. 
A corollary to Kopell's Lemma (Corollary~\ref{lem:disjoint}) implies that if 
$f$ and $g$ commute, then $U$ and $\phi(h)U$ must either be equal or disjoint. 
They are not disjoint by the previous paragraph and they are not equal 
by the hypothesis.
\end{proof}

We remark that the above fact can be generalized to arbitrary compactly supported representations which are $C^2$ in the interior. The following lemma extracts the main content of this section which will be necessary in the sequel.

\begin{lem}\label{lem:kernel-phi}
Suppose $u\in G^\dagger $ satisfies that $\supp\phi(u)$ is a nonempty set compactly contained in $\supp\phi(G^\dagger )$.
Then there exists a conjugate $u'$ of $u$ in $G^\dagger $
such that
for all $i\in\bN$, for all $s,t\in\{-1,1\}$ and for all $h\in G^\dagger $ satisfying $\|h\|<2i$,
we have 
\[
\phi[w_i u' w_i^{-1} , h' w_i u' w_i^{-1} (h')^{-1}]\ne1\]
for at least one $h'\in\{h,\bba^s\cdot h,\bbb^t\cdot h\}$.
\end{lem}
\begin{proof}
Using Lemma~\ref{lem:linear-growth}, we obtain a conjugate $u'$ of $u$ such that for each $i\in\bN$, the set $\supp\phi(w_iu'w_i^{-1})$ has a component $U_i$ whose covering length is larger than $2i$.

Note that for at least one $h'\in\{h,\bba^s\cdot h,\bbb^t\cdot h\}$, we have that \[\{\inf U_i,\sup U_i\}\not\sse\Fix \phi(h'),\] and that $||h'||\leq 2i$. The nontriviality of $\phi[w_i u' w_i^{-1} , h' w_i u' w_i^{-1} (h')^{-1}]$ follows immediately from Lemma~\ref{lem:nontrivial}.
\end{proof}

\subsection{Finishing the proof of Theorem~\ref{thm:main}}
So far, we have constructed \[ \phi=\phi_{k,\mu}\co G^\dagger \to \Diff^{k,\mu}_0(I).\]
\begin{thm}\label{thm:psi}
Suppose $\omega$ is a concave modulus satisfying $0\prec_k \omega\ll\mu$,
or suppose $\omega=\mathrm{bv}$.
If we have a representation
\[\psi\co G^\dagger \to \Diff^{k,\omega}_+(I),\]
then we have that 
 \[
 [G^\dagger,G^\dagger]\cap\ker\psi\setminus\ker\phi\ne\varnothing.\]
\end{thm}
\bp
Let $u_1:=u^\dagger\in [G^\dagger,G^\dagger] $ be the element considered in Lemma~\ref{lem:universal} and  Section~\ref{ss:config}.
By the same lemma, $\supp\psi(u_1)$ is compactly contained in $\supp\psi(G^\dagger )$.
We see from the construction that $\phi(u_1)\ne1$.
So, we may assume $\psi(u_1)\ne1$.
Let us choose a minimal collection $\{U_1,\ldots,U_n\}\sse\pi_0 \supp\psi(G^\dagger )$
such that
\[\overline{\supp\psi(u_1)}\sse U_1\cup \cdots\cup U_n.\]

There exists a conjugate $u_1'$ of $u_1$ 
satisfying
the conclusion of Lemma~\ref{lem:kernel-phi}.
Recall from Section~\ref{ss:global} that we have
\[
\lim_{i\to\infty} N_i(1/i)^{k-1}\omega(1/i)=0.\]
Hence, we
can apply Lemma~\ref{lem:kernel-psi} to $u_1'$ and $U_1$.
We obtain some $i\in\bN$, some $h_1\in G^\dagger $ with $\|h_1\|<2i$,
and some $s,t\in\{1,-1\}$ such that 
\[
U_1\cap
\supp\psi
\left[
w_i u_1' w_i^{-1},
h_1'w_i u_1' w_i^{-1}(h_1')^{-1}
\right]=\varnothing\]
for all choice of $h_1'\in\{h_1,\bba^{s}\cdot h_1,\bbb^{t}\cdot h_1\}$.
As $u_1'$ has been chosen to satisfy Lemma~\ref{lem:kernel-phi},
there exists a choice of $h_1'$
such that 
\[
u_2:=\left[
w_i u_1' w_i^{-1},
h_1'w_i u_1' w_i^{-1}(h_1')^{-1}
\right]\in [G^\dagger,G^\dagger]\setminus\ker\phi.\]

Note that $\supp\phi(u_2)$ is still compactly contained in $\supp\phi(G^\dagger )$.
We now have
\[
\overline{\supp\psi(u_2)}\sse U_2\cup \cdots\cup U_n.\]
Inductively, we use $u_2$ to obtain $u'_2$ satisfying Lemma~\ref{lem:kernel-phi}. 
The same argument as above yields
$u_3\in [G^\dagger,G^\dagger]\setminus\ker\phi$  such that
\[
\overline{\supp\psi(u_3)}\sse U_3\cup \cdots\cup U_n.\]
Continuing this way, we obtain an element $u_m\in[G^\dagger,G^\dagger]\cap \ker\psi\setminus\ker\phi$
for some $m\le n+1$.
\ep

\begin{rem}
The idea of finding a nontrivial kernel element of an interval action by successively taking commutators appeared in~\cite{BS1985}, where Brin and Squier proved that $\mathrm{PL}[0,1]$ does not contain a nonabelian free group.
One can trace this idea back to the proof of the Zassenhaus Lemma on Zassenhaus neighborhoods of semisimple Lie groups~\cite{Raghunathan1972}.
This idea was also used in~\cite{BKK2016,KKFreeProd2017}.
\end{rem}

\bp[Proof of Theorem~\ref{thm:main}]
Let $\phi_{k,\mu}=\phi$ be the representation constructed in this section. 
Theorem~\ref{thm:psi} implies the conclusion (\ref{p:main-ker}).
We have already verified (\ref{p:cinf}).
\ep

\begin{rem}
The group $\phi_{k,\mu}(G^\dagger)$ we constructed 
is never a subgroup of a right-angled Artin group,
or even a subgroup of a braid group;
see~\cite[Theorem 3.12]{KKFreeProd2017}
and~\cite[Corollary 1.2]{KK2015GT}.
\end{rem}

\section{Proof of the Main Theorem}\label{sec:consequences}
Let us now complete the proofs of all the results in the introduction.

\subsection{The Rank Trick}\label{ss:rank-trick}
If $\phi\co G\to \Homeo_+[0,1]$ be a representation,
then a priori, it is possible that the rank of the abelianization $H_1(\phi(G),\bZ)$ is less than that of $H_1(G,\bZ)$.
Let us now describe a systematic way of producing another representation $\phi_0$ such that 
the rank of $H_1(\phi_0(G),\bZ)$ is maximal.

\begin{lem}[Rank Trick]\label{lem:rank-trick}
Let $G$ be a group such that  $H_1(G,\bZ)$ is finitely generated free abelian.
If we have a representation
\[\rho\co G\to \Homeo^+(\bR)\]
such that $\supp\rho$ is bounded,
then there exists another representation 
\[\rho_0\co G\to \form{\rho(G),\Diff_+^\infty(\bR)}\le \Homeo^+(\bR)\]
satisfying the following:
\be[(i)]
\item\label{p:bdd} $\supp\rho_0$ is bounded;
\item\label{p:gsgs} $\rho_0(g)=\rho(g)$ for each $g\in [G,G]$;
\item\label{p:max-abel} $H_1(\rho_0(G),\bZ)\cong H_1(G,\bZ)$.
\ee
\end{lem}
\bp
Let $ H_1(G,\bZ)\cong\bZ^m$ for some $m\ge0$.
We can pick compactly supported $C^\infty$--diffeomorphisms $h_1,\ldots,h_m$ such that 
\[\supp h_i\cap\supp\rho(G)=\varnothing=\supp h_i\cap\bigcup_{j\ne i}\supp h_j.\]
for each $i$. The abelianization of $G$ can be realized as some surjection
\[
\alpha\co G\to\form{h_1,\ldots,h_m}\cong\bZ^m.\]
We define a representation 
$\rho_0\co G\to \Homeo^+(\bR)$ by the recipe
\[
\rho_0(g)=\rho(g)\alpha(g)\]
for each $g\in G$. It is clear that $\rho_0$ satisfies parts (\ref{p:bdd}) and (\ref{p:gsgs}).
Since $\alpha$ decomposes as
\[
\xymatrix{
G\ar[r]^>>>>{\rho_0}
&
\rho(G)\times\form{h_1,\ldots,h_m}\ar[r]^>>>>>{\mathrm{proj.}}
&\form{h_1,\ldots,h_m},
}
\]
we see that $\rho_0(G)$ surjects onto $\bZ^m$. This proves part (\ref{p:max-abel}).\ep

\begin{rem}\label{rem:rank-trick}
Algebraically, the group $\rho_0(G)$ is a subdirect product of $\rho(G)$ and $\bZ^m$.
\end{rem}

\subsection{The Chain Group Trick}
Let us describe a general technique of embedding a finitely generated orderable group into a countable simple group.
In Remark~\ref{rem:chain}, we defined the notion of a chain group, which is a certain finitely generated subgroup of $\Homeo^+(\bR)$. 
We will need the following result of the authors with Lodha:

\begin{thm}[{\cite[Theorem 1.3]{KKL2017}}]\label{thm:chain-group}
If $H\le \Homeo^+(I)$ is a chain group acting minimally on $I\setminus \partial I$,
then
 $[H,H]$ is simple
 and every proper quotient of $H$ is abelian.\end{thm}

In~\cite{KKL2017}, it is shown that every finitely generated orderable group embeds into some minimally acting chain group. 
We will need a variation of this result for diffeomorphisms.
Let us use notations $\rho_{\mathrm{GS}}$, $h_{\mathrm{GS}}$  and $\{a_0,a_1\}$ as defined in Section~\ref{ss:dense-F}.
By an \emph{$n$--generator group}, we mean a group generated by at most $n$ elements.

\begin{lem}[Chain Group Trick]\label{lem:chain-smooth}
Let $G$ be an $n$--generator subgroup of $\Homeo_+(\bR)$ such that ${\supp G}$ is compactly contained in $(0,1)$.
We put
 \[\yt G=  \form{G,\rho_{\mathrm{GS}}(F)}.\]
\be
\item
Then $\yt G$ is 
an $(n+2)$--chain group acting minimally on $(0,1)$.
In particular, $[\yt G,\yt G]$ is simple and every proper quotient of $\yt G$ is abelian.
\item
If $H_1(G,\bZ)$ is free abelian,
then 
there is an embedding from $G$ into $[\yt G,\yt G]$.
\ee
\end{lem}

\bp
We will follow the proof of \cite[Theorem 1.3]{KKL2017}, taking extra care with elements of $\rho_{\mathrm{GS}}(F)$.
Let us fix a generating set $\{g_1,\ldots,g_n\}$ of $G$.

(1)
Denote by $\bQ_{\mathrm{GS}}$ the set of $h_{\mathrm{GS}}$--images of all dyadic rationals in $[0,1]$.
We set 
\[0< s_1: =h_{\mathrm{GS}}(1/2) < s_2 := a_1^{-2}a_0.s_1 < s_3 := a_1^{-1}a_0.s_1 < s_4:= a_0.s_1 < 1.\]
Since $s_i\in\bQ_{\mathrm{GS}}$,
 we can find $f_1\in\rho_{\mathrm{GS}}(F)$ 
such that $\supp f_1 = (s_2,s_3)$ and such that $f_1(t)\ge t$ for all $t\in [0,1]$.
We fix $t_0\in(s_2,s_3)\cap \bQ_{\mathrm{GS}}$,
 so that \[s_2=f_1(s_2)<t_0  < f_1(t_0) < s_3=f_1(s_3).\]
After conjugating $G$ by a suitable element of $\rho_{\mathrm{GS}}(F)$ if necessary, 
we may assume that the closure of $\supp G$ is contained in $(t_0,f_1(t_0))$.

\begin{claim*}
If $g=g_i$ for some $1\le i\le n$, then we have that
\[
a_1\circ g(t)\begin{cases}
=t&\text{ if }t\le s_1,\\
\in(t,a_0(t))&\text{ if }t\in(s_1,s_4),\\
=a_0(t)&\text{ if }t >s_4.
\end{cases}
\]
\end{claim*}
If $t\not\in(s_2,s_3)$, then $a_1\circ g(t) = a_1(t)$ 
and the claim is obvious. If $t\in(s_2,s_3)$, then
\[
a_1^{-1}(t) < a_1^{-1}(s_3)=s_2<g(t)<s_3=a_1^{-1}(s_4)<a_1^{-1}\circ a_0(t).\]
This proves the claim.

We define 
$u_0=a_1$,
and 
$u_i=a_1 g_i$ for $i=1,\ldots,n$.
We also let
$u^*_0 = u_0^{-1}a_0$, $u^*_{n+1} = a_0^{n} u_{n} a_0^{-n}$
and
\[u^*_i=(a_0^i u_{i}^{-1} a_0^{-i})\cdot
(a_0^{i-1} u_{i-1} a_0^{1-i}),\quad i=1,\ldots,n.\]
Then we have 
\[\yt G= \form{G, a_0,a_1}=\form{u^*_0,\ldots,u^*_{n+1}}.\]
The group $\yt G$ acts minimally on $(0,1)$ since so does $\rho_{\mathrm{GS}}(F)$.

It now suffices to show that the collection $\{u^*_0,u^*_1,\ldots,u^*_{n+1}\}$ is a generating set for an $(n+2)$--chain group;
this is a routine computation of the supports using the above claim, and worked out in \cite[Lemma 4.2]{KKL2017}.

(2) 
Recall we have defined $f_1\in\rho_{\mathrm{GS}}(F)$ in part (1). We put
\[G_1=\form{ G,f_1}=\form{g_1,\ldots,g_n,f_1}\le\yt G.\]
For all distinct $i,j\in\bZ$ we have \[f_1^i(\supp G)\cap f_1^j(\supp G)=\varnothing.\]
Let $H_1(G,\bZ)\cong\bZ^m$ for some $m\le n$.
Possibly after increasing the value of $n$ if necessary, we may require that $\{g_1,\ldots,g_m\}$
generates  $H_1(G,\bZ)$, and that \[\{g_{m+1},\ldots,g_n\}\sse[G,G].\]
we have an embedding
$ G\hookrightarrow [G_1,G_1]$ defined by
\[
\begin{cases}
g_i  \mapsto g_i \cdot f_1^i  g_i ^{-1} f_1^{-i},&\text{ if }i\le m;\\
g_i  \mapsto g_i,&\text{ if }m< i\le n.
\end{cases}
\]
The proof is complete since $[G_1,G_1]\le[\yt G,\yt G]$.
\ep

\begin{rem}\label{rem:chain-smooth}
In the above lemma, put
\[ V:=\{g_1,\ldots,g_n\}\setminus\rho_{\mathrm{GS}}(F).\]
Then the group $\yt G = \form{G,\rho_{\mathrm{GS}}(F)}=\form{V,\rho_{\mathrm{GS}}(F)}$
is a $(|V|+2)$--chain group.
\end{rem}

Let us make a general observation.

\begin{lem}\label{lem:fi}
Let $G$ be an infinite group such that every proper quotient of $G$ is abelian.
Then every finite index subgroup of $G$ contains $[G,G]$.
\end{lem}
\bp
Let $G_0\le G$ be a finite index subgroup.
Then $G$ acts on the coset space $G/G_0$ by multiplication
and hence there is a representation from 
$G$ to the symmetric group of $G/G_0$.
Since every proper quotient is abelian, we see that $[G,G]$ acts trivially on $G/G_0$.
This implies $[G,G]\le G_0$.\ep

\subsection{Proof of Theorem~\ref{thm:main0}}\label{ss:main0}
We will prove the theorem by establishing several claims.
Let $k$ and $\mu$ be as given in the hypothesis of the theorem.
We denote by \[\phi=\phi_{k,\mu}\co G^\dagger\to\Diff^{k,\mu}_0(I)\]
the representation $\phi$ constructed in the previous section.
We put $T_1:=\phi(G^\dagger)$. 
From now on, we will assume $\supp T_1$ is sufficiently smaller than $I$ whenever necessary.

By the Rank Trick (Lemma~\ref{lem:rank-trick}), we can find 
\[
\phi_0\co G^\dagger\to \Diff^{k,\mu}_0(I)\]
such that the conclusions of Lemma~\ref{lem:rank-trick} hold.
We put $T_2: = \phi_0(G^\dagger)$ so that 
\[H_1(T_2,\bZ)\cong H_1(G^\dagger,\bZ)\cong\bZ^4.\]
We may assume
$\supp T_2\sse I\sse (0,1)$.

\begin{claim}\label{cla:g0}
We have that 
$T_1,T_2\le\Diff^{k,\mu}_0(I)$
and that
\[T_1,T_2\not\in
\bigcup_{0\prec_k\omega\ll\mu}
\GG^{k,\omega}(I)
\cup
\GG^{k,\mathrm{bv}}(I).\]
\end{claim}

This claim for $T_1$ follows from Theorem~\ref{thm:psi}.
In order to prove the claim for $T_2$, 
we let $0\prec_k\omega\ll\mu$
or let $\omega=\mathrm{bv}$.
Suppose $\psi\co T_2\to\Diff_+^{k,\omega}(I)$ is a representation. By applying Theorem~\ref{thm:psi} again to the composition
\[
\xymatrix{
G^\dagger\ar[r]^<<<<<{\phi_0} &
T_2\ar[r]^>>>>\psi&
\Diff_+^{k,\omega}(I)}
\]
we see that there exists $g\in [G^\dagger,G^\dagger]\setminus \ker\phi$ such that $\psi\circ\phi_0(g)=1$.
Since $\phi_0(g)=\phi(g)\ne1$ by Lemma~\ref{lem:rank-trick} (\ref{p:gsgs}),
we have $\phi_0(g)\in\ker\psi\setminus\{1\}$. This proves the claim.

We can apply the Chain Group Trick (Lemma~\ref{lem:chain-smooth}) to $T_2$,
and obtain
\[T_3 :=\form{T_2,\rho_{\mathrm{GS}}(F)}\le\Diff^{k,\mu}_0[0,1]\]
acting minimally on $(0,1)$ as a seven--generator chain group. 
From
Claim~\ref{cla:g0} 
and from the fact
 $T_2\hookrightarrow[T_3 ,T_3 ]$,
 we obtain the following and complete the proof of Theorem~\ref{thm:main0} for $M=I$.

\begin{claim}\label{cla:i-simple}
The countable simple group $[T_3,T_3]\le\Diff^{k,\mu}_0[0,1]$ satisfies that
\[[T_3,T_3]\not\in\bigcup_{0\prec_k\omega\ll\mu}\GG^{k,\omega}(I)\cup\GG^{k,\mathrm{bv}}(I).\]
\end{claim}

Let us now consider the case $M=S^1$.
After a conjugation, we may assume $\supp T_3\sse I\sse (0,1)$.
As $\BS(1,2)$ embeds into $\Diff^\infty_0(I)$,
we may regard \[T_3 \times \BS(1,2)\le \Diff_+^{k,\mu}(S^1).\]

\begin{claim}\label{cla:g6}
We have the following:
\[
[T_3 ,T_3 ]\times \BS(1,2)
\in
\GG^{k,\mu}(S^1)
\setminus
\left(
\bigcup_{0\prec_k\omega\ll\mu}\GG^{k,\omega}(S^1)
\cup\GG^{k,\mathrm{bv}}(S^1)\right).\]
\end{claim}

Let $0\prec_k\omega\ll\mu$, or let $\omega=\mathrm{bv}$.
Suppose that
 \[\psi\colon [T_3 ,T_3 ]\times \BS(1,2)\to\Diff^{{k,\omega}}_+(S^1)\] is an injective homomorphism.
By Lemma~\ref{lem:c1-denjoy} (\ref{p:z0fi}), 
a proper compact subset of $S^1$ contains $\supp\psi[T_3 ,T_3]$.
Here, we used Lemma~\ref{lem:fi} for the simple group $[T_3,T_3]$.
By Claim~\ref{cla:i-simple}, the group $\psi[T_3 ,T_3 ]$
admits no nontrivial homomorphisms to $\Diff^{{k,\omega}}_+(I)$.
It follows that $[T_3 ,T_3 ]\le\ker\psi$, a contradiction. 
This proves the claim.

Recall $F$ denotes the Thompson's group acting on $[0,1]$.
We have a natural map
\[\rho\co T_2\ast F\to T_3\le\Diff^{k,\mu}_0(I).\]
We can apply the Rank Trick to  $\rho$, since 
\[H_1(T_2\ast F,\bZ)\cong H_1(T_2,\bZ)\oplus H_1(F,\bZ)\cong \bZ^6.\]
Then we obtain a representation
\[\rho_0\co T_2\ast F\to \form{T_3,\Diff^\infty_+(\bR)}\le\Diff^{k,\mu}_+(\bR).\]
Let $T_4$ be the image of $\rho_0$. We may require that $\supp T_4\sse I\sse (0,1)$
and that $H_1(T_4,\bZ)$ is free abelian. Moreover, we have
$[T_3,T_3]\cong [T_4,T_4]$.

Regard
$T_5:=T_4\times \BS(1,2)\le\Diff^{k,\mu}_0(I)$
so that $\supp T_5 \sse I\sse (0,1)$.
We have 
\[
[T_3 ,T_3 ]\times \BS(1,2)
\cong[T_4,T_4]\times \BS(1,2)
\le T_5 .\]
Claim~\ref{cla:g6} now implies the following.

\begin{claim}\label{cla:g7}
The group $T_5$ is a nine--generator group such that 
\[T_5 \in\GG^{k,\mu}(S^1)\setminus\left(\bigcup_{0\prec_k\omega\ll\mu}\GG^{k,\omega}(S^1)\cup\GG^{k,\mathrm{bv}}(S^1)\right).\]
\end{claim}

Since $H_1(T_5 ,\bZ)\cong H_1(T_4,\bZ)\oplus\bZ$ is free abelian,
we can finally apply the Chain Group Trick to obtain  a minimally acting eleven--chain group
$Q=Q(k,\mu)$ with
\[
T_5 
\hookrightarrow[Q ,Q ]
\le Q \le \Diff^{k,\mu}_0(I)\hookrightarrow\Diff_+^{k,\mu}(S^1).\]

Summarizing, we have the following.

\begin{prop}\label{prop:bound}
Let $k\in\bN$, and let $\mu\gg\omega_1$ be a concave modulus.
Then there exists an eleven--generator group $Q=Q(k,\mu)$  such that the following hold.
\be
\item\label{p:simple} $[Q,Q]$ is simple
and every proper quotient of $Q$ is abelian.
\item\label{p:is1} $Q\le\Diff_0^{k,\mu}(I)$.
\item\label{p:inot} $[Q,Q] \not\in\bigcup_{0\prec_k\omega\ll\mu}\left(
\GG^{k,\omega}(I)\cup\GG^{k,\omega}(S^1)\right)\cup\GG^{k,\mathrm{bv}}(I)\cup\GG^{k,\mathrm{bv}}(S^1)$.
\item\label{p:fi}
Let $0\prec_k\omega\ll\mu$, or let $\omega=\mathrm{bv}$.
Then for an arbitrary finite index subgroup $A$ of $Q$,
and for 
all homomorphism
\[\psi\co A\to\Diff_+^{k,\omega}(M),\]
the image is abelian, whenever $M\in\{I,S^1\}$.
\ee
\end{prop}

\bp Part (\ref{p:simple}) follows from that $Q$ is a minimally acting chain group (Theorem~\ref{thm:chain-group}).
Part (\ref{p:is1}) is established above. 
We deduce part (\ref{p:inot}) from
\[[T_3,T_3]\hookrightarrow T_5\hookrightarrow [Q,Q].\]
Part (\ref{p:fi}) is a consequence of parts (\ref{p:simple}) and (\ref{p:inot}) along with Lemma~\ref{lem:fi}.
\ep
We have now proved Theorem~\ref{thm:main0}.
For a later use, we record the inclusion relations between the groups appearing above:
\[
[T_1,T_1]\cong[T_2,T_2]\le T_2\hookrightarrow[T_3,T_3]\cong[T_4,T_4]\le T_4\le T_5
\hookrightarrow[Q,Q]\le Q.\]
In the above diagram, the isomorphisms $\cong$ come from the Rank Trick and the embeddings $\hookrightarrow$ come from the Chain Group Trick.

\subsection{Continua of groups of the same critical regularity}\label{ss:continua}
Recall a \emph{continuum} means a set that has the cardinality of $\bR$. 
The Main Theorem is an immediate consequence of the following stronger result, combined with Theorem~\ref{thm:chain-group}.

\begin{thm}\label{thm:continua}
For each real number $\alpha\ge1$, 
there exist continua  $X_\alpha, Y_\alpha$ of minimal chain groups acting on $I$ such that the following conditions hold.
\be[(i)]
\item
For each $A\in X_\alpha$, we have that $A\le\Diff_0^\alpha(I)$
and that 
 \[ [A,A]\not\in\bigcup_{\beta>\alpha}\GG^\beta(I)\cup\GG^\beta(S^1).\]
\item
For each $B\in Y_\alpha$, we have that
$B\le\bigcap_{\beta<\alpha}\Diff_0^\beta(I)$ and that
\[ [B,B]\not\in\GG^\alpha(I)\cup \GG^\alpha(S^1).\]
\item
No two groups in $X_\alpha\cup Y_\alpha$ have isomorphic commutator subgroups.
\ee
\end{thm}

In order to prove Theorem~\ref{thm:continua}, we set up some notations. For a complex number $z\in\bC$, we let $\form{z}$ denote the largest integer $m$ such that $m<_\bC z$. For instance, we have 
$\form{k}=\form{k-\sqrt{-1}}=k-1$ 
and
$\form{k+1/2}=\form{k+\sqrt{i}}=k$ for an integer $k$. Let $z>_\bC 1$, written as $z = k+\tau+s\sqrt{-1}$ for $k=\lfloor\operatorname{Re} z\rfloor$ and $\tau,s\in\bR$.
We put
\[
\kappa(z):=(\form{z},\omega_{z-\form{z}}).\]
If $\alpha>1$ is a real number and if $k=\lfloor\alpha\rfloor$, then we see that
\[
\kappa(\alpha)=
\begin{cases}
(k,\omega_{\alpha-k}),&\text{ if }\alpha\ne k,\\
(k-1,\omega_1),&\text{ if }\alpha=k.
\end{cases}\]
Using the notation $Q(k,\mu)$ from Proposition~\ref{prop:bound},
we observe the following.
\begin{lem}\label{lem:kappa}
The following hold for all complex numbers 
$1<_\bC z <_\bC w$.
\be
\item\label{p:order} We have that $\Diff_+^{\kappa(z)}(M)\ge\Diff_+^{\kappa(w)}(M)$.
\item\label{p:o1} If $z\not\in\bN$, then $\omega_{z-\form{z}}\gg \omega_1$.
\item\label{p:sub}
If $\operatorname{Re} z>1$, then $\omega_{z-\form{z}}$ is sub-tame or $\operatorname{Re} z\ge2$.
\item\label{p:emb}
If $z\not\in\bN$ and $\operatorname{Re} w>1$, then we have that
\[
[Q\circ\kappa(z),Q\circ\kappa(z)]\not\in\GG^{\kappa(w)}(S^1).\]
\ee
\end{lem}
Note that $\GG^{\kappa(w)}(I)\sse \GG^{\kappa(w)}(S^1)$ by Theorem~\ref{thm:embeddability}.
\bp[Proof of Lemma~\ref{lem:kappa}]
Parts (1) and (2) are obvious from Lemma~\ref{lem:order}.
For part (3), 
let us write $z= k+\tau+s\sqrt{-1}$ as above.
Suppose $\omega_{z-\form{z}}$ is not sub-tame. By Lemma~\ref{lem:sup-sub}, we have that $z-\form{z}={s\sqrt{-1}}$ for some $s>0$, and that $z=k+s\sqrt{-1}$. It follows  that $k\ge2$.

For part (4), we first assume $\operatorname{Re} z\ne 1$. There exists a real number $t>s$ 
such that $st\ge0$ and such that $z<_\bC w':=k+\tau+t\sqrt{-1}<_\bC w$.
Using part (\ref{p:order}), we may assume $w=w'$.
Part (\ref{p:sub}) implies  $\omega_{w-\form{w}}\succ_k0$.
We have that
\[
\left(\kappa(z),\kappa(w)\right)=
\left((\form{z},\omega_{z-\form{z}}),
(\form{z},\omega_{w-\form{z}})\right).\]
The conclusion of (4) follows from Lemma~\ref{lem:order} and Proposition~\ref{prop:bound}.

Let us assume $\Re z=1$, so that $z=s\sqrt{-1}$ for some $s>0$.
We can pick $w'=1+\tau<_\bC w$ for some $\tau\in(0,1)$. Again, we may set $w'=w$ so that $\omega_{w-\form{w}}$ is sub-tame. The desired conclusion follows from the comparison 
\[
\left(\kappa(z),\kappa(w)\right)=
\left((1,\omega_{s\sqrt{-1}}),
(1,\omega_\tau)\right).\qedhere\]
\ep
\begin{rem}
In the case when $z=1+s\sqrt{-1}$ and $w=1+t\sqrt{-1}$ for some $0<s<t$,
we \emph{cannot} conclude that part (\ref{p:emb}) above holds. This is because  $\omega_{w-\form{w}}=\omega_{t\sqrt{-1}}$ may not be sub-tame.
\end{rem}

Let us now prove Theorem~\ref{thm:continua} for the case $\alpha>1$. We define
\begin{align*}
X_\alpha&:=\{Q\circ\kappa(\alpha+s\sqrt{-1})\co s>0\},\\
Y_\alpha&:=\{Q\circ\kappa(\alpha+s\sqrt{-1})\co s<0\}.\end{align*}
Pick a real number $s>0$ and put
$A = Q\circ\kappa\left(\alpha+s\sqrt{-1}\right)\in X_\alpha$.
Note that 
\[
A
\le
\Diff_0^{\kappa\left(\alpha+s\sqrt{-1}\right)}(I)
\le
\Diff_0^{\alpha}(I).\]
Let $\beta>\alpha$ be a non-integer real number.
By Lemma~\ref{lem:kappa}, we have that 
$[A,A]\not\in\GG^{\kappa(\beta)}(S^1)=\GG^{\beta}(S^1)$.
The conclusion (i) of the Theorem is satisfied.

Let us now pick a real number $s<0$
and put 
$
B = Q\circ\kappa\left(\alpha+s\sqrt{-1}\right)\in Y_\alpha$.
Let $\beta<\alpha$ be non-integer real number larger than $1$.
We have that
\[
B
\le
\Diff_0^{\kappa(\alpha+s\sqrt{-1})}(I)
\le
\Diff_0^{\kappa(\beta)}(I)
=
\Diff_0^{\beta}(I).\]
Since $\alpha>_\bC \alpha+s\sqrt{-1}>_\bC 1$,
we see from Lemma~\ref{lem:kappa}  that 
\[[B,B]\not\in\GG^{\kappa(\alpha)}(S^1)\supseteq \GG^\alpha(S^1).\]
This proves the conclusion (ii).

It is obvious from the conclusions (i) and (ii) that 
whenever $A\in X_\alpha$ and $B\in Y_\alpha$, we have
$[A,A]\not\cong[B,B]$.
Suppose we have real numbers $0<s_1<s_2$,
and put $A_i = Q\circ\kappa(\alpha+s_i\sqrt{-1})$.
Using $\alpha>1$
we deduce from
Lemma~\ref{lem:kappa} that 
\[ [A_1,A_1]\not\in\GG^{\kappa(\alpha+s_2\sqrt{-1})}(S^1).\]
In particular, $[A_1,A_1]\not\cong[A_2,A_2]$.
Similarly, no two groups in $Y_\alpha$ have isomorphic commutator subgroups. This proves the conclusion (iii).

Let us now construct a continuum $X_1$.
For each $\beta>1$, we pick $G_\beta\in X_\beta$.
We put $G_1 := Q\circ\kappa(1+\sqrt{-1})$
so that $[G_1,G_1]\not\in\GG^\gamma(S^1)$ for each $\gamma>1$.
By the Rank Trick for the natural surjection 
from a free group onto $G_\beta$ for $\beta\ge1$, we obtain another group
$\bar G_\beta\le \Diff_0^\beta(I)$ whose abelianization is free abelian
such that
$
[G_\beta,G_\beta]\cong[\bar G_\beta,\bar G_\beta]$.
It follows that $\bar G_\beta\not\in \GG^{\gamma}(S^1)$
for all $\gamma>\beta\ge1$.

For each $\beta>1$, we 
can apply the Chain Group Trick to $\bar G_1\times\bar G_\beta$
to obtain a minimally acting chain group $\Gamma(\beta)$
such that 
\[
\bar G_1\times\bar G_\beta\hookrightarrow [\Gamma(\beta),\Gamma(\beta)]\le \Gamma(\beta)\le\Diff_0^1(I).\]
It follows that $[\Gamma(\beta),\Gamma(\beta)]\not\in \GG^{\gamma}(S^1)$ for all $\gamma>1$.
From the consideration of critical regularities, we note that $\bar G_\beta\not\cong\bar G_\gamma$ whenever $1\le \beta<\gamma$.
Note also that $\bar G_\beta\le[\Gamma(\beta),\Gamma(\beta)]$
and that
 a countable group contains at most countably many finitely generated subgroups.
So, there exists a continuum $X^*\sse (1,\infty)$ such that 
for all distinct $\beta,\gamma$ in $X^*$, we have 
\[
[\Gamma(\beta),\Gamma(\beta)]
\not\cong
[\Gamma(\gamma),\Gamma(\gamma)].\]
Then $X_1=\{\Gamma(\beta)\mid \beta\in X^*\}$ is the desired continuum of the theorem.

Finally, let us construct a continuum $Y_1$.
To be consistent with the notations in Section~\ref{ss:main0}, let us set  
\[
T_2=\form{A,B,C\mid A^2=B^3=C^7=ABC}\le\widetilde{\PSL(2,\bR)}\le\Homeo_+(\bR).\]
As we noted in Remark~\ref{rem:not a group},
we have that
$\GG^0(M)=\GG^{\mathrm{Lip}}(M)$. So, it suffices to compare the regularities $C^0$ and $C^1$.
Kropholler and Thurston (see~\cite{Bergman1991PJM}) observed that the group $T_2$ is a finitely generated perfect group, and by Thurston Stability, that every homomorphism from $T_2$ to $\Diff_+^1(I)$ has a trivial image. In particular, $H_1(T_2,\bZ)$ is trivial and $T_2\in\GG^0(I)\setminus\GG^1(I)$. 
We continue as in Section~\ref{ss:main0}, after substituting $(k,\mu)=(0,0)$ and $(k,\omega)=(1,0)$ (and forgetting $k,\mathrm{bv}$).
We obtain groups $T_3, T_4, T_5$ and 
a minimally acting chain group $Q\le\Homeo_+(I)$ such that
\[T_2\hookrightarrow  [Q,Q]\not\in \GG^1(S^1).\]

Let us put $H_1:=Q$.
The construction of $Y_1$ is very similar to that of $X_1$.
For each $\beta>1$, 
we can find a finitely generated group 
$\bar H_\beta\le\bigcap_{\gamma<\beta}\Diff_0^\gamma(I)$
such that $H_1(\bar H_\beta,\bZ)$ is free abelian,
and such that  
$\bar H_\beta\not\in\GG^\beta(S^1)$.
For each $\beta>1$, 
we apply the Chain Group Trick to $\bar H_1\times\bar H_\beta$ and obtain a minimal chain group $\Lambda(\beta)$
such that \[\bar H_1\times\bar H_\beta\hookrightarrow [\Lambda(\beta),\Lambda(\beta)]
\le\Lambda(\beta)\le\Homeo_+(I).\]
As before, there exists a continuum $Y^*\sse(1,\infty)$
such that $Y_1=\{\Lambda(\beta)\mid\beta\in Y^*\}$ is the desired collection.
Note that no two groups in the collection $X_1\cup Y_1$ have isomorphic commutator subgroups.

\begin{rem}
Calegari~\cite{CalegariForcing} exhibited a finitely generated group in $\GG^0(S^1)\setminus\GG^1(S^1)$. 
Lodha and the authors~\cite{KKL2017} gave (continuum many distinct) finitely generated groups inside $\GG^0(I)\setminus\GG^1(I)$ having simple commutator groups, building on~\cite{LM2016GGD}. The last part of the above proof strengthens both of these results. 
\end{rem}

\subsection{Algebraic and topological smoothability}\label{ss:fol}
Theorem~\ref{thm:main} also implies that if $\alpha\geq 1$ is a real number, then there are very few homomorphism $\Diff_+^{\alpha}(S^1)\to\Diff_+^{\beta}(S^1)$ and $\Diff_{c}^{\alpha}(\bR)\to\Diff_{c}^{\beta}(\bR)$ for all $\beta>\alpha$. 

\begin{proof}[Proof of Corollary~\ref{cor:intro-continuous}]
By the Main Theorem, none of the maps in (1) through (3) are injective. The desired conclusion now follows from Theorem~\ref{thm:continuous}.
\end{proof}

Group actions of various regularities on manifolds are closely related to foliation theory (see~\cite{CandelConlonI}, for instance). One of the canonical constructions in foliation theory is the \emph{suspension} of a group action, a version of which we recall here for the convenience of the reader. 
Recall our hypothesis that $M\in\{I,S^1\}$.
Let $B$ be a closed manifold with a universal cover $\tilde{B}\to B$.
Suppose we have a representation 
\[\psi\colon \pi_1(B)\to \Diff_+^{\alpha}(M).\]
The manifold $\tilde{B}\times M$ has a natural product foliation so that each copy of $\tilde{B}$ is a leaf. The group $\pi_1(B)$ has a diagonal action on $\tilde{B}\times M$, given by the deck transformation $\pi_1(B)\to\Homeo(\tilde{B})$ and by the map $\psi$. 
The quotient space 
\[
E(\psi) = \left(\tilde{B}\times M\right) /\pi_1(B)\]
is a $C^\alpha$--foliated bundle. This construction is called the \emph{suspension} of $\psi$; see~\cite{CandelConlonI} for instance. Two representations 
$\psi,\psi'\in\Hom(\pi_1(B), \Diff_+^{\alpha}(M))$
yield homeomorphic suspensions $E(\psi),E(\psi')$ as foliated bundles if and only if $\psi$ and $\psi'$ are topologically conjugate~\cite[Theorem2]{CN1985}.

Let us now consider the case $M=I$
and $B=S_g$, a closed surface of genus $g\ge2$.
Let $k\ge0$ be an integer.
Cantwell--Conlon~\cite{CC1988} and Tsuboi~\cite{Tsuboi1987}
independently proved the existence of a representation $\psi_k\in\Hom(\pi_1(S_g),\Diff_+^k(I))$ such that $\psi_k$ is not topologically conjugate to a representation in $\Hom(\pi_1(S_g),\Diff_+^{k+1}(M))$. 
So, they concluded:

\begin{thm}[{See~\cite{CC1988} and~\cite{Tsuboi1987}}]\label{thm:cctsu}
For each integer $k\ge0$, there exists a $C^k$--foliated bundle structure on $S_2\times I$ which is not homeomorphic to a $C^{k+1}$--foliated bundle.\end{thm}

We will now prove Corollary~\ref{cor:foliation}, which is the only remaining result in the introduction that needs to be shown.
Assume $\alpha\ge1$ is a real number and $g\ge5$.
Theorem~\ref{thm:main} implies that there exists a representation \[\psi_\alpha\in\Hom(\pi_1(S_g),\Diff_0^\alpha(I))\] such that $\psi_\alpha$ is not topologically conjugate to a representation in \[\Hom(\pi_1(S_g),\bigcup_{\beta>\alpha}\Diff_+^\alpha(I)).\] Hence, we may replace the hypotheses $C^k$ and $C^{k+1}$ in Theorem~\ref{thm:cctsu} by $C^\alpha$ and $\bigcup_{\beta>\alpha}C^\beta$, respectively.

We can further extend this result to more general $3$--manifolds, using the techniques in~\cite{CC1982} described as follows. Every closed $3$--manifold $Y$ with $H_2(Y,\bZ)\ne0$ contains an embedded $2$--sided closed surface $S_g$ for all sufficiently large $g>0$. Goodman used this observation to prove that $Y\setminus\operatorname{Int}(S_g\times I)$ admits a smooth foliation structure, based on Thurston's result; see~\cite[Corollary 3.1]{Goodman1975} and \cite{Thurston1976AM}. By adding in the aforementioned foliated bundle structure of $S_g\times I$ inside $Y$, we complete the proof of Corollary~\ref{cor:foliation}.

\section{Further questions}\label{s:que}
Let $M\in\{I,S^1\}$.
One can ask for a finer distinction at integer regularities.
A difficulty with part (1) below is that there does not exist a concave modulus below $\omega_1$, by definition.

\begin{que}\label{que:reg2}
\be
\item
Let $k\ge1$. Does there exist a finitely generated subgroup 
$G\le\Diff_+^{k,\mathrm{Lip}}(M)$ that does not admit an injective homomorphism into $\Diff_+^{k+1}(M)$?
\item
Does there exist a finitely generated group in the set
\[
\bigcap_{\beta\in\bN}\GG^\beta(M)\setminus\GG^{\infty}(M)?\]
\ee
\end{que}

Many questions also persist about algebraic smoothability of groups. For instance, finite presentability as well as all other higher finiteness properties of the groups we produce are completely opaque at this time. We ask the following, in light of Theorem~\ref{thm:continua}:

\begin{que}
For which choices of $\alpha$ and $\beta$ do there exist finitely presented groups $G\in\GG^{\alpha}(M)\setminus\GG^{\beta}(M)$? What if $\alpha,\beta\in\bN$?
\end{que}

Moreover, the constructions we carry out in this paper are rather involved. It is still quite difficult to prove that a give group does not lie in $\GG^{\beta}(M)$.

\begin{que}
Let $G$ be a finitely generated group. Does there exist an easily verifiable algebraic criterion which precludes $G\in\GG^{\beta}(M)$?
\end{que}

\appendix
\section{Diffeomorphism groups of intermediate regularities}\label{appendix}
Let $M\in\{I,S^1\}$.
We will record some basic properties of $\Diff^{k,\omega}_+(M)$.
Most of these properties are well-known for the case $\omega=0$, but not explicitly stated in the literature for a general concave modulus $\omega$. We will also include brief proofs.

\subsection{Group structure}
Let $k\in\bN$, and let $\omega$ be a concave modulus. 
In~\cite{Mather1}, it is proved that for a smooth manifold $X$,
the set $\Diff_c^{k,\omega}(X)_0$ is actually a group. 
We sketch a proof of this fact for one--manifolds, and also include the case $\omega=\mathrm{bv}$.

The following lemma is useful for inductive arguments on the regularities.

\begin{lem}\label{lem:omega1}
Suppose $\omega$ is a concave modulus, or $\omega\in\{0,\mathrm{bv}\}$.
Let $k\in\bN$, and let 
\[
F,G\co M\to\bR\]
be maps such that
 $F$ is $C^{k-1,\omega}$ and such that $G$ is $C^k$.
Then the following hold.
\be
\item
The multiplication $F\cdot G$ is $C^{k-1,\omega}$.
\item
The composition
$F\circ G$ is $C^{k-1,\omega}$.
\ee
\end{lem}
\bp
This lemma is proved in~\cite{Mather1} when $\omega=0$ or when $\omega$ is a concave modulus. So we assume $\omega=\mathrm{bv}$.
We let $\{x_i\}$ be a partition of $M$. 

(1) First consider the case $k=1$. We note
\[
|F\cdot G(x_i)-F\cdot G(x_{i-1})|
\le
|F(x_i)-F(x_{i-1})|\cdot\|G\|_\infty+
\|F\|_\infty \cdot \|G\|_{1,\infty} |x_i-x_{i-1}|.\]
Hence, if $F\cdot G$ is $C^{\mathrm{bv}}$.
If $k>1$, then we use an induction to see that 
\[
(F\cdot G)' = F'\cdot G + F\cdot G'\]
is  $C^{k-2,\mathrm{bv}}$. This proves part (1).

(2) 
The map $F\circ G$ is well-defined for all $x\in M$.
Let us first assume $k=1$, so that $F\in C^\mathrm{bv}$.
Since $G$ is bijective, we see that
\[
\sum_i|F\circ G(x_i)-F\circ G(x_{i-1})|
\le\operatorname{Var}(F,M)<\infty.\]
The induction step follows from
\[(F\circ G)'=(F'\circ G)\cdot G'.\qedhere\]
\ep

\begin{prop}\label{prop:group}
Let $\omega$ be a concave modulus, or let $\omega=\{0,\mathrm{bv}\}$. Then for each $k\in\bN$, the following is a group  where the binary operation is the group composition:
\[\Diff^{k,\omega}_+(M).\]
\end{prop}

\bp
Let $f,g\in\Diff_+^{k,\omega}(M)$.
It is well-known that $\Diff^k_+(M)$ is a group. 
So, we have $f^{-1},f\circ g\in\Diff_+^{k}(M)$.
It suffices to show that both are $C^{k,\omega}$.

Note that $(f\circ g)'=(f'\circ g)\cdot g'$.
Since $f'$ is $C^{k-1,\omega}$ and $g$ is $C^k$,
Lemma~\ref{lem:omega1} implies that 
$f'\circ g$ is $C^{k-1,\omega}$.
By the same lemma, we see that $(f\circ g)'$ is $C^{k-1,\omega}$.
This proves $f\circ g$ is $C^{k,\omega}$.

We can write 
\[
(f^{-1})'=r\circ f'\circ f^{-1}\]
where $r\co (0,\infty)\to(0,\infty)$ is the
$C^\infty$ diffeomorphism $r(x)=1/x$.
Note that $f'$ stays away from $0$. As $f'$ is $C^{k-1,\omega}$ and $f^{-1}$ is $C^k$,
we again see that $f^{-1}$ is $C^{k,\omega}$.
\ep

\subsection{Groups of compactly supported diffeomorphisms}
We now establish a topological conjugacy between certain diffeomorphism groups.


\begin{thm}\label{thm:embeddability}
Let $\omega$ be a concave modulus.
Then for each $k\in\bN$, the group 
$  \Diff_+^{k,\omega}(I)$
 is topologically conjugate to a subgroup of $\Diff_c^{k,\omega}(\bR)$.
\end{thm}

Muller~\cite{Muller} and Tsuboi~\cite{Tsuboi1984Asterisque} established the above result for the case $\omega=0$.
Our proof follows the same line, but an extra care is needed for a general concave modulus $\omega$ as described in the lemmas below.

When we say a function $f$ is \emph{defined for $x\ge0$}, we implicitly assume to have a small number $A>0$ so that $f$ is defined as
\[
f\co [0,A]\to \bR.\]
We let $k$ and $\omega$ be as in Theorem~\ref{thm:embeddability}.

\begin{lem}\label{lem:ckw}
Suppose $f$ is a $C^{k,\omega}$ map defined for $x\ge0$
such that
\[
f(0)=f'(0)=\cdots=f^{(k)}(0)=0.\]
Then the following hold.
\be
\item
We have that
\[
f(x) = \int_{t_1=0}^x\int_{t_2=0}^{t_1}\cdots \int_{t_k=0}^{t_{k-1}} f^{(k)}(t_k)\;dt_k\cdots dt_1.\]
\item
The map $f/x^k$ extends to a $C^\omega$ map on $x\ge0$.
\item
The map $f/x$ extends to a $C^{k-1,\omega}$ map on $x\ge0$.
\ee
\end{lem}
We thank Nam-Gyu Kang for suggesting a key idea for the proof below.
\bp[Proof of Lemma~\ref{lem:ckw}]
Part (1) is simply an application of the Fundamental Theorem of Calculus. Let us consider part (2). We note for all small $h>0$ that
\[
\left|
\frac{f(h)}{h^k}\right|
\le
\frac1{h^k}\int_{t_1=0}^h\int_{t_2=0}^{t_1}\cdots \int_{t_k=0}^{t_{k-1}} |f^{(k)}(t_k)-f^{(k)}(0)|
\le[f^{(k)}]_\omega\cdot \omega(h).\]
So, $f/x^k$ is $C^\omega$ at $x=0$. For all small $0<x<x+h$, we see that
\begin{align*}
\left|
\frac{f(x+h)}{(x+h)^k}
-\frac{f(x)}{x^k}\right|
&\le
\left|
\frac1{(x+h)^k}
\int_{t_1=x}^{x+h}\int_{t_2=0}^{t_1}\cdots \int_{t_k=0}^{t_{k-1}}f^{(k)}(t_k)
\right|
\\
&+\left(1-\frac{x^k}{(x+h)^k}\right)
\cdot\frac1{x^k}
\left|
\int_{t_1=0}^x\int_{t_2=0}^{t_1}\cdots \int_{t_k=0}^{t_{k-1}}f^{(k)}(t_k)
\right|\\
&\le
\left(
\frac{h(x+h)^{k-1}\cdot  \omega(x+h)}{(x+h)^k}+
\left(1-\frac{x^k}{(x+h)^k}\right)
\cdot \omega(x)\right)[f^{(k)}]_\omega.
\end{align*}
Using the inequalities
$1-1/(1+t)^k\le kt$ 
and 
$\omega(s)/s\le \omega(t)/t$
for all $0<t<s$, we conclude that 
\[
\left|
\frac{f(x+h)}{(x+h)^k}
-\frac{f(x)}{x^k}\right|
\le 
(k+1)[f^{(k)}]_\omega\cdot \omega(h).\]
This proves part (2).

For (3), we have some $a_i\in\bZ$ such that
\[
({f}/{x})^{(k-1)}=\sum_{i=1}^{k} a_i f^{(k-i)}/x^i.\]
Since $f^{(k-i)}$ is $C^{i,\omega}$, we see from part (2) that 
$(f/x)^{(k-1)}$ is $C^\omega$.
\ep

The rest of the proof for Theorem~\ref{thm:embeddability} closely follows the argument in~\cite{Tsuboi1984Asterisque}, as we summarize below.
Let us fix a map that is defined near $x=0$:
\[\phi(x)=e^{-1/x}.\] 

\begin{lem}\label{lem:fmodx}
For a $C^{k,\omega}$ map $g$ defined for $x\ge0$
satisfying $g(0)=0$ and $g'(0)>0$, the following hold.
\be
\item The map $h=g/x$ is a $C^{k-1,\omega}$ map defined for $x\ge0$.
\item
The map $\psi\circ g\circ\phi$ is a $C^{k,\omega}$ map defined for $x\ge0$.
Moreover, we have
\[\psi\circ g\circ\phi(0)=0,\quad (\psi\circ g\circ\phi)'(0)=1.\]
\item
The map $\Phi(g):=\psi^2\circ g\circ\phi^2$ is a $C^{k,\omega}$ map defined for $x\ge0$,
and moreover, $\Phi(g)^{(i)}(0)=\Id^{(i)}(0)$ for all $0\le i\le k$.
\ee
\end{lem}
\bp
(1)
If $T_kg(x)$ denotes the $k$--th degree Taylor polynomial for $g$,
then $f = g-T_kg$ satisfies the condition of Lemma~\ref{lem:ckw}.
The conclusion follows since $g/x - f/x = T_kg/x$ is a polynomial.

(2) 
Put $G = \psi\circ g\circ \phi$, so that
\[
G(x) = \frac{-1}{\log(g\circ\phi)}=\frac{-1}{-1/x+\log( (g\circ\phi)/\phi)}=\frac{x}{1-x\log h\circ\phi}.\]
By part (1), the map $h$ is $C^{k-1,\omega}$ for $x\ge0$.
As $x$ approaches to $0$, the denominator of the above expression for $G$ stays away from $0$ because
\[
\lim_{x\to0}1-x\log h\circ\phi=1-0\cdot\log g'(0)=1.\]
It follows that $G$ is $C^{k-1,\omega}$ for $x\ge0$. Moreover, $G$ is $C^{k,\omega}$ for $x>0$.

We compute the following:
\begin{align*}
G'(0) &= \lim_{x\to0} 1/(1- x\log h\circ\phi(x))=1,\\
G'(x) &=
\frac{1+ \phi \cdot (h'\circ\phi) / (h\circ\phi)}{(1-x\log h\circ\phi)^2}.\end{align*}
From $xh' = g' - h$, we see that $\phi\cdot (h'\circ\phi)$ is $C^{k-1,\omega}$ and that $\lim_{x\to0}G'(x)=1$.
We conclude that $G'$ exists for $x\ge0$ (even when $k=1$),
and is $C^{k-1,\omega}$. 
It follows that $G$ is $C^{k,\omega}$. 

(3) We only need  to compute $\Phi(g)^{(i)}(0)$.
By setting $y=\phi^2(x)$, 
we have that 
\[
\frac{\Phi(g)(x)-x}{\phi}
=\frac{\psi^2 g(y)-\psi^2(y)}{\psi(y)}
=(-\log y)\left(
\frac1{\log(-\log g)}-\frac1{\log(-\log y)}
\right).\]
It is a simple
exercise on L'Hospital's Rule to see that
\[
\lim_{x\to0}\frac{\Phi(g)(x)-x}{\phi}
=
\lim_{y\to0}\frac{-\log y}{(\log(-\log y))^2}
\left(\frac{\log y}{\log g}-1\right)=0.\]
For all $0\le i\le k$, we have that 
\[
\lim_{x\to0} \frac{\Phi(g)(x)-x}{x^i}
\lim_{x\to0} \frac{\Phi(g)(x)-x}{\phi(x)}\cdot \frac{\phi(x)}{x^i}=0.\]
By L'Hospital's Rule again, we have $(\Phi(g)-\Id)^{(i)}=0$ for all $0\le i\le k$.\ep

\bp[Proof of Theorem~\ref{thm:embeddability}]
Consider a $C^\infty$--\emph{homeomorphism} $\phi\co I\to I$ such that $\phi(x)=e^{-1/x}$ near $x=0$,
and such that $\phi(x)=1-e^{-1/(1-x)}$ near $x=1$. We put $\psi=\phi^{-1}$.
For each $g\in \Diff_+^{k,\omega}(I)$, we define $\Phi(g) = \psi^2\circ g\circ \phi^2$.
Then Lemma~\ref{lem:fmodx} (3) (after using the symmetry at $x=0$ and $x=1$) implies that $\Phi(g)\in \Diff_c^{k,\omega}(\bR)$.
\ep

\subsection{Simplicity}
Let us use the following terminology from~\cite{KKL2017}.
Let $X$ be a topological space, and let $H\le\Homeo(X)$.
We say $H$ acts \emph{CO-transitively} (or, \emph{compact--open-transitively}) if for each proper compact subset $A\sse X$ and for each nonempty open subset $B\sse X$, there is $u\in H$ such that $u(A)\sse B$. 
Lemma~\ref{lem:higman} is a variation of a result commonly known as Higman's Theorem.

\begin{lem}[{\cite[Lemma 2.5]{KKL2017}}]\label{lem:higman}
Let $X$ be a non--compact Hausdorff space, and let $\Homeo_+(X)$ denote the group of compactly supported homeomorphisms of $X$.
If $H\le\Homeo_+(X)$ is CO-transitive,
then $[H,H]$ is simple.
\end{lem}

Let $X$ be a topological space.
We say $H\le\Homeo(X)$ has the \emph{fragmentation property}
for an open cover $\UU$ of $X$,
if each element $h\in H$
can be written as \[
h = h_1\cdots h_\ell\]
such that the support of $h_i$ is contained in some element of $\UU$. The following lemma is very useful when proving simplicity of homeomorphism groups. This lemma is originally due to Epstein~\cite{Epstein1970}; let us state a generalization by Ling~\cite{Ling1984}.

\begin{lem}[\cite{Epstein1970,Ling1984}]\label{lem:ling}
Let $X$ be a paracompact Hausdorff space with a basis $\BB$,
and let $H\le\Homeo(X)$.
Assume the following.
\be[(i)]
\item $H$ has the fragmentation property
for each subcover $\UU$ of $\BB$;
\item for each $U,V\in\BB$ there exists some $h\in H$ such that $h(U)\sse V$.
\ee
Then $[H,H]$ is simple.
\end{lem}

The following lemma is known for $\omega=0$~\cite{PS1970}, detailed proofs of which can be found in~\cite{Banyaga1997,Mann2016NYJM}. The proof for a concave modulus $\omega$ is the same almost in verbatim.

\begin{lem}\label{lem:frag}
Let $k\in\bN$, and let $\omega$ be a concave modulus.
Then for a smooth manifold $X$ without boundary,
 the group $\Diff_c^{k,\omega}(X)_0$ has the fragmentation property for an arbitrary open cover of $X$.\end{lem}
 
From now on, we let $X\in\{S^1,\bR\}$.
We let $C_c^\omega(X,\bR)$ denote the set of real-valued compactly supported $\omega$-continuous maps $X\to\bR$. 
For each $f\in C_c(X,\bR)=C_c^0(X,\bR)$, we define the \emph{optimal modulus function} of $f$ as
\[
\mu^f(t):=\sup\{|fx-fy|\co x,y\in X\text{ and }|x-y|\le t\}.\]
It is trivial that for all $x,y\in X$ we have
$|fx-fy|\le \mu^f(|x-y|)$.

\begin{lem}\label{lem:conc-mod}
For $X\in\{S^1,\bR\}$
and for $f\in C_c(X,\bR)$, the following hold.
\be
\item
The optimal modulus function $\mu^f\co [0,\infty)\to[0,\infty)$ is continuous, monotone increasing and subadditive.
\item
For all $s,t>0$, we have that $\mu^f(t)\le(1+t/s)\mu^f(s)$.
\item
There exists a concave modulus $\mu$ such that 
$f\in C_c^\mu(X,\bR)$
and such that 
\[
C_c^{\mu}(X,\bR)=\bigcap \{C_c^\omega(X,\bR)\mid \omega\text{ is a concave modulus and }f\in C_c^\omega(X,\bR)\}.\]
\ee\end{lem}

\bp
Part (1) is a consequence of the convexity of $X$ and the uniform continuity of $f$.
Part (2) is obvious when $t\le s$. 
If $t>s$, then part (2) follows from
 \[\mu^f(t)\le\mu^f(t-s\lfloor t/s\rfloor)+\lfloor t/s\rfloor \mu^f(s)
\le
(1+t/s)\mu^f(s).\]

For part (3), we will use the idea described in~\cite[p.194]{BL1976book}. Let  $\FF$ be the family of continuous, monotone increasing, concave functions $h\co [0,\infty)\to[0,\infty)$ such that $\mu^f(t)\le h(t)$ for all $t\ge0$. 
For instance, part (2) implies that the line
\[
h_s(t)=(1+t/s)\mu^f(s)\]
belongs to $\FF$ for each $s>0$.
Define
\[
\mu_1(t):=\inf_{h\in\FF}h(t)\le h_t(t)=2\mu^f(t).\]
Then 
$\mu_1$ is continuous, monotone increasing and concave.
Put $\mu:=\mu_1+\Id$,  so that
\[\mu^f\le\mu_1\le \mu\le 2\mu^f+\Id.\]
We see that $\mu$ is a concave modulus such that 
$f\in C_c^\mu(X,\bR)$. 

Put $T:=\diam \supp f\ge0$.
Suppose $f\in C_c^\omega(X,\bR)$ for some concave modulus $\omega$.
It only remains to show that
$C_c^\mu(X,\bR)\sse C_c^\omega(X,\bR)$.
For each $t>0$, we have 
\[
\mu^f(t)
=\sup_{|x-y|\le t} |fx-fy|
\le
[f]_\omega \cdot \omega(t).\]
For all $0<t\le T$, we see that
\[
\mu(t)
\le2\mu^f(t)+t
\le \left(2[f]_\omega+ T/\omega(T)\right)\cdot \omega(t).\]
There exists a constant $K$ such that 
for each $g\in C^{\mu}_c(X,\bR)$, we have
\[[g]_\omega\le
\sup_{|x- y|\le T}\frac{|gx-gy|}{\omega(|x-y|)}
+\sup_{|x- y|\ge T}\frac{|gx-gy|}{\omega(|x-y|)}
\le
K[g]_\mu+\frac{2\|g\|_\infty}{\omega(T)}<\infty.
\]
It follows that $g\in C^\omega_c(X,\bR)$ and the lemma is proved.
\ep

We are now ready to prove the simplicity of certain diffeomorphism groups.

\begin{thm}[Theorem~\ref{thm:continuous}]\label{thm:continuous2}
For each $X\in\{S^1,\bR\}$, the following hold.
\be
\item
If $\alpha\ge1$ is a real number, then 
every proper quotient 
of $\Diff_c^\alpha(X)_0$
is abelian.
If, furthermore, $\alpha\ne2$, then $\Diff_c^\alpha(X)_0$ is simple.
\item
If $\alpha>1$ is a real number, then every proper quotient 
of $\bigcap_{\beta<\alpha}\Diff_c^\beta(X)_0$
is abelian.
If, furthermore, $\alpha>3$, then $\bigcap_{\beta<\alpha}\Diff_c^\beta(X)_0$ is simple.
\ee
\end{thm}

\bp
We prove the theorem through a series of claims.

\begin{claim}\label{cla:r-com}
The following groups have simple commutator groups:
\begin{itemize}
\item $\Diff_c^\alpha(\bR)$ for $\alpha\ge1$;
\item $\bigcap_{\beta<\alpha}\Diff_c^\beta(\bR)$  for $\alpha>1$.\end{itemize}
\end{claim}
Both of the above groups contain $\Diff_c^\infty(\bR)$.
Since $\Diff_c^\infty(\bR)$ acts CO-transitively on $\bR$,
the claim follows from Lemma~\ref{lem:ling}.

\begin{claim}\label{cla:s1-com}
For each $\alpha\ge1$, the commutator group of $\Diff_+^{\alpha}(S^1)$ is simple.\end{claim}
By Lemma~\ref{lem:frag}, the group $\Diff_+^{\alpha}(S^1)$  satisfies the condition (i) of Lemma~\ref{lem:ling}. The condition (ii) follows from $\Diff_+^\infty(S^1)\le\Diff_+^{\alpha}(S^1)$.

\begin{claim}\label{cla:prop-abel1}
If $\alpha\ge1$, then every proper quotient of $\Diff_c^\alpha(X)_0$ is abelian.
\end{claim}
By an easy application of Kopell's Lemma and Denjoy's Theorem~\cite{Denjoy1932}, we see that $\Diff_c^\alpha(X)_0$ has trivial center. Combined with Claims~\ref{cla:r-com} and~\ref{cla:s1-com}, this implies the assertion.

Recall from Section~\ref{ss:continua} that we defined the notation $\form{z}$ for $z\in \bC$.
\begin{claim}\label{cla:cap}
If $\alpha>1$, then
there exists a collection of concave moduli $\FF(\alpha)$ such that 
\[
\bigcap_{\beta<\alpha}\Diff_c^\beta(X)_0=\bigcup_{\mu\in\FF(\alpha)}\Diff_c^{\form{\alpha},\mu}(X)_0.\]
\end{claim}
Put $k=\lfloor\alpha\rfloor$.
Assume first  $\alpha\ne k$, so that $k=\form{\alpha}$.
Suppose  we have \[f\in \bigcap_{\beta<\alpha}\Diff_c^\beta(X)_0\le\Diff_c^k(X)_0.\]
Let $\mu$ be a concave modulus
as in Lemma~\ref{lem:conc-mod} for the map $f^{(k)}\in C_c(X,\bR)$.
Whenever $k<\beta<\alpha$, we have $f^{(k)}\in C^{\beta-k}(X,\bR)$. The same lemma implies that
\[C_c^\mu(X,\bR)\sse C_c^{\beta-k}(X,\bR).\]
So, we have $f\in \Diff_c^{k,\mu}(X)_0\sse \bigcap_{\beta<\alpha}\Diff_c^\beta(X)_0$ and completes the proof when $\alpha\ne k$. 
The proof of the case that $\alpha=k=\form{\alpha}+1$ is almost identical.

\begin{claim}\label{cla:diff-comm-s1}
For each $\alpha>1$,
every proper quotient of $\bigcap_{\beta<\alpha}\Diff_c^\beta(X)_0$
is abelian.
\end{claim}
The case $X=\bR$ follows from Claim~\ref{cla:r-com},
so we may only consider the group
\[
G = \bigcap_{\beta<\alpha}\Diff_+^\beta(S^1).\]
By Lemma~\ref{lem:frag} and Claim~\ref{cla:cap},
the group $G$ has the fragmentation property for an arbitrary cover. Since $\Diff_+^\infty(S^1)\le G$, we can deduce  Claim~\ref{cla:diff-comm-s1} from Lemma~\ref{lem:ling}.

Coming back to the proof of the theorem,
we only need to prove the latter parts of (1) and (2).
The latter part of (1) is a special case of Corollary~\ref{cor:mather}. 
For the latter part of (2), assume $\alpha>3$. 
We see from Mather's Theorem and from Claim~\ref{cla:cap} that the group $\bigcap_{\beta<\alpha}\Diff_c^\beta(X)_0$ is  a union of perfect groups. The conclusion follows.
\ep

\section*{Acknowledgements}
The authors thank J. Bowden, M. Brin, D. Calegari, \'E. Ghys, C. Gordon, N. Kang, M. Kapovich, K. Mann, Y. Minsky, M. Mj, M. Pengitore, R. Schwartz, S. Taylor, and T. Tsuboi for helpful discussions. 
We are particularly thankful to A. Navas for detailed comments on an earlier version of this paper, and for pointing out the reference~\cite{BMNR2017MZ}, which allowed us to greatly strengthen our original result. The authors thank an anonymous referee.
The first author is supported by Samsung Science and Technology Foundation (SSTF-BA1301-06 and SSTF-BA1301-51).
The first author is partially supported by a KIAS Individual Grant (MG073601) at Korea Institute for Advanced Study.
The second author is partially supported by an Alfred P. Sloan Foundation Research Fellowship and NSF Grant DMS-1711488.


\def\cprime{$'$} \def\soft#1{\leavevmode\setbox0=\hbox{h}\dimen7=\ht0\advance
  \dimen7 by-1ex\relax\if t#1\relax\rlap{\raise.6\dimen7
  \hbox{\kern.3ex\char'47}}#1\relax\else\if T#1\relax
  \rlap{\raise.5\dimen7\hbox{\kern1.3ex\char'47}}#1\relax \else\if
  d#1\relax\rlap{\raise.5\dimen7\hbox{\kern.9ex \char'47}}#1\relax\else\if
  D#1\relax\rlap{\raise.5\dimen7 \hbox{\kern1.4ex\char'47}}#1\relax\else\if
  l#1\relax \rlap{\raise.5\dimen7\hbox{\kern.4ex\char'47}}#1\relax \else\if
  L#1\relax\rlap{\raise.5\dimen7\hbox{\kern.7ex
  \char'47}}#1\relax\else\message{accent \string\soft \space #1 not
  defined!}#1\relax\fi\fi\fi\fi\fi\fi}
\providecommand{\bysame}{\leavevmode\hbox to3em{\hrulefill}\thinspace}
\providecommand{\MR}{\relax\ifhmode\unskip\space\fi MR }
\providecommand{\MRhref}[2]{%
  \href{http://www.ams.org/mathscinet-getitem?mr=#1}{#2}
}
\providecommand{\href}[2]{#2}

\end{document}